\documentclass{amsart}
\usepackage{amsmath}
\usepackage{graphicx}
\usepackage{latexsym}
\usepackage{color}
\usepackage{amscd}
\usepackage[all]{xy}
\usepackage{enumerate}
\usepackage{enumitem}
\usepackage{hyperref}
\usepackage{subfigure}
\usepackage{soul}
\usepackage{comment}
\usepackage{epsfig}
\usepackage{epstopdf}
\usepackage{marginnote}
\usepackage{listings}
\usepackage{xcolor}
\DeclareTextCommand{\k}{OT1}[1]{\UseTextAccent{T1}{\k}{#1}}

\usepackage{multirow}

\newtheorem{thma}{Theorem}

\newtheorem{thmb}{Theorem}

\newtheorem{thmc}{Theorem}

\newtheorem{thm}{Theorem}
\newtheorem{lem}[thm]{Lemma}
\newtheorem{cor}[thm]{Corollary}
\newtheorem{prop}[thm]{Proposition}

\newtheorem{corollary}[thm]{Corollary}
\newtheorem{proposition}[thm]{Proposition}
\newtheorem{question}[thm]{Question}

\theoremstyle{definition}

\newtheorem*{definition*}{Definition}

\newtheorem{remark}[thm]{Remark}

\newtheorem{properties}[thm]{Properties}

\newcommand{\CPb}{\overline{\mathbb{CP}}{}^{2}}
\newcommand{\CP}{{\mathbb{CP}}{}^{2}}

\newcommand{\R}{\mathbb{R}}

\newcommand{\Z}{\mathbb{Z}}

\newcommand{\twprod}{\mathbin{\mathchoice%
    {\ooalign{\raise1.15ex\hbox{$\scriptstyle\sim$}\cr\hidewidth$\times$\hidewidth\cr}}%
    {\ooalign{\raise1.15ex\hbox{$\scriptstyle\sim$}\cr\hidewidth$\times$\hidewidth\cr}}%
    {\ooalign{\raise.85ex\hbox{$\scriptscriptstyle\sim$}\cr\hidewidth$\scriptstyle\times$\hidewidth\cr}}%
    {\ooalign{\raise.65ex\hbox{$\scriptscriptstyle\sim$}\cr\hidewidth$\scriptscriptstyle\times$\hidewidth\cr}}%   
    }}

\newcommand{\QED}{\hfill \ensuremath{\Box}}
\newcommand{\M}{\operatorname{Mod}}

\usepackage{scalerel} %Used to change sizes of subscripts
\newcommand{\DT}[1]{t_{\scaleto{\mathstrut #1}{9pt}}}

\def \eu{{\text{e}}}

\allowdisplaybreaks[4]

\begin{document}

\title[Sections, pseudosections, and multisections] 
%{Not every Lefschetz fibration over the $2$--sphere has a section}
{Sections, pseudosections, and multisections \\ of Lefschetz fibrations}

\author[R. \.{I}. Baykur]{R. \.{I}nan\c{c} Baykur}
\address{Department of Mathematics and Statistics, University of Massachusetts, Amherst, MA 01003, USA}
\email{inanc.baykur@umass.edu}

\author[N. Hamada]{Noriyuki Hamada}
\address{Institute of Mathematics for Industry, Kyushu University, Fukuoka 819-0395, Japan}
\email{hamada@imi.kyushu-u.ac.jp}
%\date{\today}

\begin{abstract}
We show that not every symplectic Lefschetz fibration over the $2$--sphere admits a smooth section, settling a long-standing open problem. To prove this, we develop a method for studying sections via lifts of monodromy factorizations with point-pushing maps and the induced handle decomposition of the total space, and apply it to an infinite family of Lefschetz fibrations. In contrast, we show that every Lefschetz fibration over $S^2$ admits positive multisections. Some of these examples admit no multisections with spherical components.
We also construct signature-zero Lefschetz fibrations of every odd genus $g\geq 7$, which in turn give symplectic Lefschetz fibrations of any prescribed signature in each of these genera.
\end{abstract}

%\vspace{-0.2in}

\maketitle

\setcounter{secnumdepth}{2}
\setcounter{section}{0}

\tableofcontents

%\vspace{0.15in}
\section{Introduction}
%======================================================================================

A fundamental yet elusive question in the theory of Lefschetz fibrations has been whether every such fibration over the $2$--sphere admits a smooth section. Listed as Problem~2.22 in the K3 problem list \cite{K3}, it has been raised repeatedly throughout the literature; see e.g. \cite{smith:sections, auroux:stable, onaran:genus2sections, lee-servan:sections, gompf:sections}, and Problems~2.8, 4.1, 7.1 and 8.1 in \cite{korkmaz:problems}, \cite{stipsicz:problems}, \cite{korkmaz-stipsicz:problems}, and \cite{endo:problems} respectively. Our main theorem gives a negative answer to this long-open question.

\smallskip
\begin{thma}\label{thm:a}
For each $g\in\{7,9\}$, there exist infinitely many genus--$g$ symplectic Lefschetz fibrations over the $2$--sphere that do not admit smooth sections. These examples admit pseudosections. Moreover, the genus--$9$ fibrations have no multisections with spherical components.
\end{thma}

A smoothly embedded surface $S$ in the smooth, oriented $4$--manifold $X$ is a \emph{section} of the Lefschetz fibration $f\colon X\to S^2$ if it intersects every fiber transversely at one smooth point. We call $S$ a \emph{pseudosection}, or a \emph{homological section}, if it intersects every fiber algebraically once instead. More generally, a \emph{$d$--section}, or \emph{multisection of degree $d$}, is an embedded surface $S\subset X$ such that $f|_S\colon S\to S^2$ is a degree--$d$ simple branched cover with only positive branch points; when branch points of either sign are allowed, we call it an \emph{achiral $d$--section}. Multisections appear naturally in relation to relative Hilbert schemes and Seiberg--Witten theory \cite{donaldson-smith:pencils, Perutz:matching-2, baykur-hayano:multisections}. Thus Theorem~\ref{thm:a} says that the obstruction to the existence of sections in our examples is not homological, whereas Theorem~\ref{thm:b} below shows that the situation is dramatically different for multisections. 

By convention, Lefschetz fibrations in this paper have non-empty critical locus and only positive nodes. This excludes surface bundles and achiral Lefschetz fibrations, where different phenomena occur. Indeed, over $S^2$ there are genus--$1$ surface bundles \cite[Lemma~10]{baykur-kamada} and achiral Lefschetz fibrations of arbitrarily high genus \cite{matsumoto:genus1, gompf:sections} that do not even admit homological sections.

Our next result answers the corresponding multisection variant of Problem~2.22 in \cite{K3}, building on~\cite{Perutz:matching-2}.

\smallskip
\begin{thmb}\label{thm:b}
Let $f\colon X\to S^2$ be a Lefschetz fibration, and let $m(f)$ denote the divisibility of the fiber class. Then, for every $d\geq 2$, the fibration $f$ admits a $d$--section if and only if $m(f)$ divides $d$. In particular, every Lefschetz fibration over the $2$--sphere admits multisections.
\end{thmb}

Our examples are built from explicit positive factorizations in mapping class groups and the breeding technique introduced in \cite{baykur-genus3, hamada:MCKsections, baykur-hamada:Lefschetz}. Two atomic examples are a new genus--$7$ Lefschetz fibration, whose construction generalizes to any odd genus $g\geq 7$, and a genus--$9$ Lefschetz fibration we built earlier in \cite{baykur-hamada:Lefschetz}. Neither admits a section. Both have signature zero, and the genus--$7$ construction and its generalizations provide building blocks that allow us to improve and generalize some of the results in \cite{baykur-hamada:Lefschetz}. Our new examples involve a variation of the breeding technique (see Section~\ref{subsec:genus7}), which may be of independent interest. 

\smallskip
\begin{thmc}\label{thm:c}
There exist infinitely many symplectic Lefschetz fibrations over the $2$--sphere, with any prescribed signature and odd fiber genus $g\geq 7$.
\end{thmc}

In Section~\ref{sec:prelim}, we associate to a monodromy factorization
an obstruction $\omega(f)\in G$, where $G$ is the fundamental group of
the complement of a regular fiber. It is obtained from lifts fixing a
point in the reference fiber and is independent of the auxiliary
choices, under the corresponding identifications of $G$. A section can
exist only if $\omega(f)=1$. Geometrically, this class records the
attaching circle of the last $2$--handle and helps us recover complete
Kirby diagrams even when the existence of a section is unknown.

For Theorem~\ref{thm:a}, we compute this obstruction for the two explicit
examples in Section~\ref{sec:examples}. In both cases it is nontrivial
in $G$, although its abelianization vanishes. For the genus--$9$
example, its image has infinite order in a class--$2$ nilpotent
quotient; for the genus--$7$ example, its image has order $2$ in the
integral class--$5$ nilpotent quotient. The GAP verifications are given
in Appendices~\ref{app:gap} and~\ref{app:genus7-gap}, respectively.
These calculations also rule out sections for every untwisted self
fiber sum of the genus--$9$ example and for the untwisted fiber sum of
any odd number of copies of the genus--$7$ example. The vanishing of
the abelianized obstructions accounts for the existence of
pseudosections. The order of $\omega(f)$ also determines the degrees
of spherical classes over the base. In genus $9$, its infinite order
rules out any spherical components for multisections and Corollary~\ref{cor:genus9-aspherical}
further gives compatible symplectic forms for which there are no
symplectic spheres. We expect that similar methods will show that the
higher-genus examples in Section~\ref{subsec:higher-genera} also admit
no sections.

In contrast, the existence of multisections is governed by the
abelianized obstruction. Its order determines the divisibility of the
fiber class and hence the possible degrees of multisections. Combined
with the results on relative Hilbert schemes in
\cite{Perutz:matching-2}, this gives the general existence result in Theorem~\ref{thm:b}. 

Alongside these nonexistence and existence results, we aim to provide a
practical guide to finding pseudosections, sections, and multisections
from monodromy factorizations in Subsection~\ref{subsec:finding-sections}.
We illustrate these procedures in relatively challenging examples. The genus--$7$
calculation in Subsection~\ref{subsec:genus7-pi1} shows how to track the
point-pushing loop through a breeding construction to get a pseudosection. In Appendix~\ref{app:xiao-section}, we show how to find four disjiont sections of Xiao's holomorphic genus--$2$ Lefschetz
fibration of type $(6,7)$, using the monodromy calculation of Szab\'{o} in \cite{szabo:xiao-monodromy}.

\smallskip
\noindent \textit{Conventions.}
We denote by $\Sigma_{g,k}^b$ a compact orientable surface of genus $g$ with $b$ boundary components and $k$ marked points in its interior, suppressing $b$ or $k$ from the notation when they are zero. The mapping class group $\M(\Sigma_{g,k}^b)$ consists of orientation-preserving diffeomorphisms of $\Sigma_{g,k}^b$ fixing the boundary and marked points pointwise, modulo isotopies of the same type. A right-handed Dehn twist about a simple closed curve $c$ is denoted by $t_c$, and our products of mapping classes act on curves from right to left. We write $W^\varphi=\varphi W\varphi^{-1}$ and $[x,y]=xyx^{-1}y^{-1}$. We write a monodromy factorization of a genus--$g$ Lefschetz fibration as $t_{c_1}\cdots t_{c_n}=1$ in $\M(\Sigma_g)$, and refer to such factorizations as positive factorizations.
For based loops, $\alpha\beta$ means that $\alpha$ is traversed first.
The point-pushing map $\mathcal P_\gamma$ is obtained by pushing the
marked point along $\gamma$; thus
$\mathcal P_\alpha\mathcal P_\beta=\mathcal P_{\beta\alpha}$.
For a simple oriented based loop $\gamma$, we use
$\mathcal P_\gamma=t_{\gamma_L}^{-1}t_{\gamma_R}$, where
$\gamma_L,\gamma_R$ are its left and right parallel copies.

A Lefschetz fibration in this paper means a relatively minimal map from a closed, connected, smooth, oriented $4$--manifold to $S^2$, with connected regular fibers, distinct critical values, non-empty critical locus, and only positive Lefschetz singularities. Here relatively minimal means that no fiber contains an embedded sphere of self-intersection $-1$. For background on Lefschetz fibrations, mapping class groups, and monodromy factorizations, see \cite{gompf-stipsicz:book, baykur-hamada:Lefschetz}.

\smallskip
\noindent \textit{Acknowledgements.}
The first author was partially supported by NSF grants DMS-2005327 and DMS-2506431, and by a Simons Foundation Travel Grant. The second author was partially supported by JSPS KAKENHI Grant Numbers 25K24558, 25K06980, and 25K15247. We are grateful to Tim Perutz for suggesting to the first author, some years ago, the use of relative Hilbert schemes and symmetric products to approach the multisection question---any errors or shortcomings in these arguments are, of course, our own. We thank Seraphina Lee, Carlos Andres Marcelo Serv\'an, and Riccardo Pedrotti for their interest in this work and for their questions, which helped shape the final exposition. We also thank Zolt\'{a}n Szab\'{o} for several helpful conversations on Xiao's $(6,7)$ fibration discussed in Appendix~\ref{app:xiao-section}.

%We thank our colleagues for their interest and patience since some of these results were first announced.

The constructions, figures, and fundamental group calculations in all of this article but Appendix~B were worked out the old-fashioned way---by hand. AI tools were used to assist with checking the consistency of the calculational inputs, preparing verification codes for GAP, and in final revision of the manuscript. In the later amended Appendix~\ref{app:xiao-section}, the curve configuration was checked and the compatible figures were generated using AI tools. The quotient group computations themselves were carried out in GAP.

\section{Sections, pseudosections, and multisections}  \label{sec:prelim}
%======================================================

This section develops the basic tools used throughout the paper. We first associate to a monodromy factorization a point-pushing obstruction to the existence of a section, then explain how this obstruction appears in the handle decomposition, and finally relate its abelianization to multisection degrees and its order to spherical classes.

\subsection{Monodromy lifts}
\label{subsec:pointpush}

For the point-pushing description in this subsection, assume that
$g\geq2$. Let
\[
t_{c_1}\cdots t_{c_n}=1
\]
be a positive factorization in $\M(\Sigma_g)$. Fix a marked point
$p\in\Sigma_g$, and choose representatives
of the curves $c_i$ disjoint from $p$. The corresponding Dehn twists
define lifts $\widetilde t_{c_i}\in\M(\Sigma_{g,1})$. By the Birman
exact sequence, their product is a point-pushing map:
\[
\widetilde t_{c_1}\cdots\widetilde t_{c_n} =
\mathcal P_{\gamma}
\]
for a unique $\gamma\in F:=\pi_1(\Sigma_g,p)$.

Set
\[
N=\langle\!\langle c_1,\ldots,c_n\rangle\!\rangle_F,
\]
where the vanishing cycles are based at $p$ by arbitrary paths. The
normal closure $N$ is independent of these choices.

\begin{lem}\label{lem:pointpush-class}
The image of $\gamma$ in $F/N$ is independent of the pointed
representatives of the vanishing cycles and is invariant under Hurwitz
moves.
\end{lem}

\begin{proof}
Suppose that $\widehat c_i$ is another pointed representative of the
same unmarked vanishing cycle $c_i$. Choose an isotopy from $c_i$ to
$\widehat c_i$ that passes through $p$ transversely at finitely many
times. Between these crossings, the corresponding twists agree in
$\M(\Sigma_{g,1})$. At each crossing, the curves immediately before
and after are isotopic, away from $p$, to the boundary curves of an
annulus containing $p$. Their twists therefore differ by a point push
along a based copy of $c_i$ or $c_i^{-1}$. Such a loop is a conjugate
of $c_i^{\pm1}$ and lies in $\langle\!\langle c_i\rangle\!\rangle_F$.
This also covers the separating case in which the marked point
changes sides. Consequently,
\[
t_{\widehat c_i}\,\widetilde t_{c_i}^{-1}
\in\mathcal P\bigl(\langle\!\langle c_i\rangle\!\rangle_F\bigr)
\subseteq\mathcal P(N).
\]
Moreover, each $t_{c_j*}$ acts trivially on $F/N$, so it preserves $N$.
Changing the pointed representatives one at a time therefore changes
the total product only by elements of $\mathcal P(N)$ and their
conjugates by the preceding twists. It follows that the resulting
point-pushing loop has the same image in $F/N$.

A Hurwitz move performed on the pointed lifts leaves the total product
unchanged. It also replaces the vanishing cycles by another collection
with the same normal closure. Hence the quotient class of $\gamma$ is
Hurwitz-invariant.
\end{proof}

We denote this class by
\[
\omega(f):=[\gamma]\in F/N
\]
and call it the \emph{point-pushing obstruction} associated to the
factorization. A change of the marked point or the identification of
the reference fiber transports both $F/N$ and $\omega(f)$ by the
corresponding natural isomorphism, possibly changing a representative
by conjugation. Similarly, simultaneous conjugation of the
factorization transports $\omega(f)$ by the induced isomorphism. Thus
its triviality and its order are invariants of the Lefschetz fibration.

\subsection{The last $2$--handle and the point-pushing obstruction}
\label{subsec:last2handle}
% ================================================

Choose a disk $D\subset S^2$ containing all the critical values, and let
\[
X^\circ=f^{-1}(D):=
X\setminus\operatorname{int}\nu(F_0),
\]
where $F_0$ is a regular fiber over the complementary disk. The standard
handle decomposition of $X^\circ$ is obtained from
$\Sigma_g\times D^2$ by attaching Lefschetz $2$--handles along the
vanishing cycles. Consequently,
\[
G:=\pi_1(X^\circ)\cong F/N.
\]

A meridian of $F_0$ determines a conjugacy class in $G$, up to
inversion if its orientation is not specified. The proposition below
identifies this class, up to inversion, with $\omega(f)$ when $g\geq2$.
In genus $1$ we use $\omega(f)$ for either oriented meridian class.
In every genus, write
\[
\overline\omega(f)\in H_1(G;\Z)
\]
for its abelianization. Its vanishing and order do not depend on the
choice of orientation.

To recover $X$, one glues back the trivial fibration over the
complementary disk. Relative to its boundary, this trivial fibration
contributes one $2$--handle, followed by $2g$ $3$--handles and one
$4$--handle. The only additional handle affecting the fundamental
group is therefore the last $2$--handle.

\begin{prop}\label{prop:pointpush-obstruction}
The attaching circle of the last $2$--handle is a meridian of $F_0$.
When $g\geq2$, under the identification $G\cong F/N$ and after
orienting the attaching circle appropriately, it represents the
point-pushing obstruction
\[
\omega(f)=[\gamma].
\]
In particular,
\[
\pi_1(X)\cong G/\langle\!\langle\omega(f)\rangle\!\rangle,
\]
and, when $g\geq2$,
\[
\pi_1(X)
\cong
F/\langle\!\langle c_1,\ldots,c_n,\gamma\rangle\!\rangle_F.
\]
Moreover, $\omega(f)=1$ is a necessary condition for $f$ to admit a
section. On the other hand, when $g\geq2$, if $\gamma=1$ in $F$, then
the marked factorization determines a section.
\end{prop}

\begin{proof}
The $2$--handle in the relative handle decomposition of
$\Sigma_g\times D^2$ is attached along the loop
$\{p\}\times\partial D^2$. Under the boundary identification determined
by the monodromy, this loop is a meridian of the removed fiber. When
$g\geq2$, the pointed monodromy identifies it with a loop whose fiber
component is $\gamma$. Thus, up to orientation, its class in
$\pi_1(X^\circ)\cong F/N$ is represented by $\gamma$. Attaching this
handle kills the meridian, while the subsequent $3$-- and $4$--handles
do not affect the fundamental group.

If $f$ admits a section, then removing its intersection with
$\nu(F_0)$ leaves a properly embedded disk in $X^\circ$. The boundary
of this disk is a meridian of $F_0$. Meridians at
different points of the connected fiber are conjugate in
$\pi_1(X^\circ)$, so the meridian class is already trivial in $G$.
Hence $\omega(f)=1$.

Finally, if $g\geq2$ and $\gamma=1$ in $F$, then
$\mathcal P_\gamma=1$ in $\M(\Sigma_{g,1})$. Thus the chosen lifts
multiply to the identity in $\M(\Sigma_{g,1})$, which gives a section.
\end{proof}

In fact, $\omega(f)$ lies in the center of $G$. Indeed, the fiber
subgroup of $\pi_1(\partial X^\circ)\cong\pi_1(\Sigma_g)\times\Z$
surjects onto $G$ and commutes with the meridian. Thus, once $G$ and
the orientation of the meridian are fixed, the conjugacy ambiguity
disappears.

\begin{remark}\label{rmk:hillman-pedrotti}
Hillman and Pedrotti give an explicit algebraic criterion for when a
section over the boundary mapping torus of a Lefschetz fibration over
$D^2$ extends to a continuous, respectively smooth, section over the
disk \cite[Theorem~3.19]{hillman-pedrotti}. Their condition can be
viewed as the ordered algebraic record of modifying the boundary loop
by the local vanishing-cycle terms, or geometrically by sliding across
the Lefschetz $2$--handles. Thus it retains the full ordered
handle-slide data, rather than reducing to a quotient test. In
particular, their criterion implies triviality of the corresponding
class in $F/N$, whereas our obstruction forgets this additional data.
We therefore use $\omega(f)$ only as an obstruction to sections, not as
a sufficient criterion for constructing them.
\end{remark}

\subsection{Finding sections and pseudosections}
\label{subsec:finding-sections}
%======================================================

We next spell out how the preceding obstruction can be found in
examples. The goal is to give a practical guide for detecting sections,
pseudosections, and multisections from monodromy factorizations.

\subsubsection*{Marked lifts and pseudosections.}
Suppose first that a Lefschetz fibration of genus $g\geq2$ is given by
an explicit factorization
\[
t_{c_1}\cdots t_{c_n}=1
\]
in $\M(\Sigma_g)$. Choose a marked point $p$ and choose representatives of all $c_i$ disjoint from
$p$. This gives a lifted product in $\M(\Sigma_{g,1})$, hence a
point-pushing map $\mathcal P_\gamma$. If $\gamma=1$ in
$F=\pi_1(\Sigma_g,p)$, the lift gives a section. In general, the loop
$\gamma$ records the defect of the chosen lift; its image in $G=F/N$
is the obstruction class $\omega(f)$.

To identify $\gamma$ in practice, one may choose a filling collection
of curves and arcs in $\Sigma_g\setminus\{p\}$ and compute its image
under the lifted product. Alexander's method then determines the
mapping class \cite{FarbMargalit:modbook}. Since the lifted product is
already known to lie in the point-pushing subgroup, it remains only to
recognize which point-push carries the original filling system
to its image under the lifted product. In computations, a single well-chosen arc often
reveals the point-pushing loop, but a filling curve--arc system gives a
rigorous certificate.

The point-pushing word by itself does not yet give a pseudosection. As
shown in the next subsection, if
$\overline\omega(f)=0$ in $H_1(G;\Z)$, then there is a class
$A\in H_2(X;\Z)$ with $A\cdot[F_0]=1$. Such a class has a smoothly
embedded oriented representative, after joining components by tubes if
necessary. After a small perturbation away from the critical points,
this representative is a pseudosection in our sense.

\subsubsection*{Sections from breeding.}
There is also a useful way to find sections by tracking how a positive
factorization for a Lefschetz fibration is derived from basic mapping
class group relations supported on surfaces such as $\Sigma_0^4$ for
the lantern relation and $\Sigma_1^1$ for the $2$--chain relation; see
\cite{hamada:MCKsections,baykur-hamada:Lefschetz}. This is the breeding
technique underlying many of our constructions, with the full ancestry
of the boundary components kept throughout the process. Under the
embeddings used in breeding, some boundary components of the smaller
pieces may become null-homotopic curves in the larger fiber surface;
equivalently, they are capped by disks. Before capping, these boundary
components record sections of the corresponding lifted factorization.
Boundary components that survive the construction give honest
sections. On the other hand, when a boundary twist is canceled
against a positive Dehn twist, the cancellation may hold only after
forgetting the marked point, leaving a residual point-pushing term
in the marked surface. If the corresponding abelianized obstruction
vanishes, the resulting
fibration nevertheless admits a pseudosection.

\subsubsection*{Multisections from several marked points.}
For multisections one can use the analogous marked-point strategy with
several marked points. Here one uses the mapping class group
\[
\M(\Sigma_g;\{p_1,\ldots,p_d\})
\]
of diffeomorphisms preserving the marked set, rather than fixing its
points individually. For $g\geq2$, the kernel of the forgetful map is
the full surface braid group, and positive branch points of the multisection
appear as positive half twists in this kernel. This is the point of
view used in \cite{baykur-hayano:multisections} to encode multisections
by monodromy factorizations.

The next subsection gives a complementary existence result for
multisections, using the abelianized obstruction.

\subsection{Multisections}
\label{subsec:multisections}

A \emph{$d$--section}, or \emph{multisection of degree $d$}, of a
Lefschetz fibration $f\colon X\to S^2$ is a possibly disconnected
embedded oriented surface $S\subset X$, disjoint from
$\operatorname{Crit}(f)$, such that
\[
f|_S\colon S\to S^2
\]
is a simple branched covering of degree $d$ whose branch points are all
positive \cite[Section~2.1]{baykur-hayano:multisections}. Thus, near a
branch point, there are orientation-compatible local coordinates in
which
\[
S=\{z=w^2\}
\qquad\text{and}\qquad
f(z,w)=z.
\]
These are the \emph{standard surfaces} of Donaldson--Smith
\cite[Definition~2.7]{donaldson-smith:pencils}. We use the term
\emph{achiral $d$--section} when branch points of either sign are
allowed; a negative branch point has the local model $z=\bar w^2$.
After a small perturbation, the branch values may be taken to be
distinct and disjoint from the critical values of $f$, as we do below.

We first record an elementary local modification.

\begin{lem}\label{lem:positivize-multisection}
If a Lefschetz fibration $(X,f)$ admits an achiral $d$--section
with $d\geq2$, then it admits a positive $d$--section.
\end{lem}

\begin{proof}
Let $x$ be a negative branch point. Choose a disk $D\subset S^2$
containing $f(x)$ but no other branch or critical values, and fix a
trivialization
\[
f^{-1}(D)\cong \Sigma\times D,
\]
where $\Sigma$ is a regular fiber. The intersection of the achiral
multisection with $\Sigma\times\partial D$ is a closed $d$--braid whose
class in the surface braid group $B_d(\Sigma)$ is a negative half twist
$\tau_a^{-1}$ along an arc $a$ joining the two sheets which meet at
$x$.

Let $h=g(\Sigma)$. We first choose a connected orientation-preserving
simple branched covering
\[
q\colon \widehat\Sigma\to S^2
\]
of degree $d$, where $\widehat\Sigma$ has genus $h$. To see that such a
cover exists, set
\[
r=2h+2d-2.
\]
In the symmetric group $S_d$ of degree $d$, let $\sigma_i=(i\ i+1)$. The word
\[
\sigma_1\sigma_1\,\sigma_2\sigma_2\cdots
\sigma_{d-1}\sigma_{d-1}\,
\underbrace{\sigma_1\sigma_1\cdots\sigma_1\sigma_1}_{2h\ \text{factors}}
\]
has length $r$ and product equal to the identity, and its factors
$\sigma_1,\dots,\sigma_{d-1}$ generate a transitive subgroup
of $S_d$. By the Hurwitz existence theorem
\cite{husemoller:ramified}, it is realized as the monodromy of a
connected simple branched cover of $S^2$ of degree $d$ with $r$ branch points. The
Riemann--Hurwitz formula gives
\[
\chi(\widehat\Sigma)=2d-r=2-2h,
\]
so the source has genus $h$.

The cover $q$ also gives a surface-braid monodromy. Over the complement
of the branch values, the configuration
$q^{-1}(z)\subset\widehat\Sigma$ defines a map to the unordered
configuration space of $d$ points in $\widehat\Sigma$, and hence a
homomorphism from $\pi_1(S^2 \setminus \{ \text{branch values} \} )$ to $B_d(\widehat\Sigma)$. Around each branch point the
braid monodromy is a positive half twist. Choose one branch point and a
nearby regular value $z_0$, and let $a_0$ be its vanishing arc in $\widehat\Sigma$, with endpoints in
$q^{-1}(z_0)$ and interior disjoint from this set. After identifying $\widehat\Sigma$ with the fiber
$\Sigma$, we may arrange that $q^{-1}(z_0)$ is the given configuration
of $d$ points and that $a_0$ is identified with the prescribed arc
$a$. Thus the local braid monodromy at the distinguished branch point
is $\tau_a$.

Choose a Hurwitz system in which the distinguished branch point comes
first. If $\beta$ is the product of the positive half twists about all
the remaining branch points, then the total braid monodromy on $S^2$
gives
\[
\tau_a\beta=1,
\qquad\text{so}\qquad
\beta=\tau_a^{-1}.
\]
The graph
\[
\Gamma_q=\{(y,q(y))\mid y\in\widehat\Sigma\}
\subset \widehat\Sigma\times S^2
\]
is an embedded positive $d$--section of the trivial fibration $\widehat\Sigma \times S^2 \to S^2$. Remove
a small disk $D_0$ around the distinguished branch value. After the
above identification $\widehat\Sigma\cong\Sigma$, the boundary braid of
\[
\Gamma_q\cap
\bigl(\Sigma\times(S^2\setminus\operatorname{int}D_0)\bigr)
\]
is $\tau_a^{-1}$. It therefore agrees, up to braid isotopy, with the
boundary braid of the original negatively branched piece over $D$.

After identifying $S^2\setminus\operatorname{int}D_0$ with $D$, making
both surfaces product on collars of their boundaries, and applying the
braid isotopy there, we replace the entire original piece over $D$ by
this positive braided surface. This removes the negative branch point
and introduces only positive branch points. Repeating the construction
at every negative branch point proves the claim.
\end{proof}

\begin{remark}
We can carry out this modification in much more general settings;
for example, over a higher genus base, in an achiral fibration, or in a
broken Lefschetz fibration. The modification may change the topology,
the number of components, and the homology class of the surface, but it
preserves its degree over the base.
\end{remark}

Retain the notation
\[
F=\pi_1(\Sigma,p),\qquad
N=\langle\!\langle c_1,\ldots,c_n\rangle\!\rangle_F,
\qquad
G=F/N\cong\pi_1(X^\circ),
\]
and let
\[
\overline\omega(f)\in H_1(G;\Z)
\]
denote the abelianized meridian class, with the orientation chosen
to agree with the point-pushing convention when $g\geq2$.

\begin{thm}\label{thm:multisection-degrees}
Let $f\colon X\to S^2$ be a Lefschetz fibration with connected regular
fiber $\Sigma$. For every integer $d\geq2$, the following conditions
are equivalent:
\begin{enumerate}[label=(\roman*)]
\item the fibration $f$ admits a $d$--section;
\item there is a class $A\in H_2(X;\Z)$ such that
$A\cdot[\Sigma]=d$;
\item
$d\,\overline\omega(f)=0$ in $H_1(G;\Z)$.
\end{enumerate}
\end{thm}

\begin{proof}
The implication ${\rm (i)}\Rightarrow{\rm (ii)}$ is immediate by taking
$A$ to be the homology class of the multisection.

Suppose ${\rm (ii)}$ holds. We first modify $A$, without changing its
intersection with a regular fiber, so that it has nonnegative
intersection with every irreducible component of every singular fiber.
For each reducible singular fiber, write
\[
F_v=F_{v,1}\cup F_{v,2},
\]
where the two components meet once positively. In homology,
\[
[F_v]=[F_{v,1}]+[F_{v,2}],\qquad
F_{v,1}^2=F_{v,2}^2=-1,
\qquad
F_{v,1}\cdot F_{v,2}=1.
\]
Set $a_v=A\cdot F_{v,1}$ and define
\[
A'=A+\sum_v a_v[F_{v,1}],
\]
where the sum is over the reducible singular fibers. Components
belonging to distinct fibers are disjoint, so these modifications do
not interfere with one another. Moreover,
\[
A'\cdot[\Sigma]=d,
\]
and for each reducible singular fiber,
\[
A'\cdot F_{v,1}=0,
\qquad
A'\cdot F_{v,2}=d.
\]
For an irreducible singular fiber, its underlying $2$--cycle represents
$[\Sigma]$, so its intersection with $A'$ is $d$. Thus $A'$ has
nonnegative intersection with every irreducible component of every
singular fiber.

Let
\[
f^{[d]}\colon X^{[d]}\to S^2
\]
be the relative Hilbert scheme of $d$ points on the fibers of $f$.
We use the patching argument in the proof of
\cite[Proposition~4.10]{Perutz:matching-2}, spelling out the local
choices at the singular fibers.

Choose disjoint small disks $D_v$ about the critical values and put
$E_v=f^{-1}(D_v)$. At a reducible fiber, choose $d$ distinct smooth
points on $F_{v,2}$; at an irreducible fiber, choose $d$ distinct smooth
points away from the node. These points extend to $d$ disjoint local
sections over $D_v$, and hence to a section $\sigma_v$ of the relative
Hilbert scheme. Their relative cycle represents the restriction of
$A'$ to $H_2(E_v,\partial E_v;\Z)$. Indeed, $E_v$ retracts onto its
nodal fiber, $H_1(E_v;\Z)$ is torsion-free, and $H_2(E_v;\Z)$ is
freely generated by the irreducible components of that fiber.
By Poincar\'e--Lefschetz duality and the universal coefficient theorem,
a relative $2$--class is therefore determined by its intersections
with these components. For both classes the intersections are $(0,d)$
in the reducible case and $d$ in the irreducible case.

Over
\[
U=S^2\setminus\bigcup_v\operatorname{int}D_v
\]
the fibration is a surface bundle, and $U$ retracts onto a graph.
Since $d\geq2$,
\[
\pi_1(\operatorname{Sym}^d\Sigma)\cong H_1(\Sigma;\Z).
\]
Consequently, the cycle-class map from sections of the relative
symmetric product over $U$ onto relative $2$--classes of fiber degree
$d$ is surjective. To see this directly, choose a section over the
graph, which is possible because $\operatorname{Sym}^d\Sigma$ is
connected. Changing it along the edges is parametrized, modulo changes
at the vertices, by $H^1(U;\mathcal H_1)$, where $\mathcal H_1$ is the
local system of first homology groups of the fibers. By fiberwise
Poincar\'e duality, these changes give exactly the kernel of restriction
$H^2(f^{-1}(U);\Z)\to H^2(\Sigma;\Z)$: the cohomology sequence of a
bundle over a graph identifies this kernel with
$H^1(U;\mathcal H_1)$. The cycle map induces this identification,
since moving a point around a loop in the fiber sweeps out that loop.
We can therefore choose a section $\sigma_U$ whose relative cycle is
the restriction of $A'$.

The restrictions of $\sigma_U$ and $\sigma_v$ to each common boundary
have the same cycle class in $H_1(f^{-1}(\partial D_v);\Z)$, with the
same boundary orientation. For a surface bundle over a circle with
monodromy $T$, homotopy classes of sections of its degree--$d$ symmetric
product are parametrized, after choosing one section, by
\[
H_1(\Sigma;\Z)/(T_*-1)H_1(\Sigma;\Z).
\]
This follows from the usual twisted-conjugacy description and the
abelian fundamental group of $\operatorname{Sym}^d\Sigma$. The same
quotient parametrizes degree--$d$ classes in the first homology of the
boundary mapping torus. Thus equality of the boundary cycle classes
implies that the boundary sections are homotopic through sections.
Using these homotopies on collars, we glue $\sigma_U$ and the
$\sigma_v$ to a smooth section $\sigma$ of $f^{[d]}$. Only its degree
$d$, rather than its precise global cycle class, will be needed.

Perturb $\sigma$ through sections, keeping it fixed near the singular
fibers, so that over the regular values it is transverse to the
diagonal strata. The strata other than the top diagonal have real
codimension at least four, so they are avoided. Near each transverse
intersection with the top diagonal, we may further arrange the local
holomorphic or antiholomorphic model, according to the intersection
sign. As in the proof of \cite[Lemma~4.11]{Perutz:matching-2}, the
pullback of the universal family is then a smooth surface $T_\sigma$
whose map to $X$ is an embedding and whose projection to $S^2$ is a
degree--$d$ simple branched covering. Its local models are
\[
z=w^2\qquad\text{and}\qquad z=\bar w^2.
\]
The image is therefore an embedded achiral $d$--section, disjoint from
$\operatorname{Crit}(f)$. Lemma~\ref{lem:positivize-multisection} turns
it into a positive $d$--section. This last modification may change its
homology class, but preserves the degree $d$. This proves
${\rm (ii)}\Rightarrow{\rm (i)}$.

It remains to compare ${\rm (ii)}$ and ${\rm (iii)}$. Let
\[
X^\circ=X\setminus\operatorname{int}\nu(\Sigma).
\]
By excision,
\[
H_2(X,X^\circ;\Z)
\cong
H_2(\Sigma\times D^2,\Sigma\times S^1;\Z)
\cong\Z.
\]
The relevant part of the long exact sequence of the pair
$(X,X^\circ)$ is therefore
\[
H_2(X;\Z)
\xrightarrow{\;\cdot[\Sigma]\;}
\Z
\xrightarrow{\;\partial\;}
H_1(X^\circ;\Z)
\cong H_1(G;\Z).
\]
The first map is algebraic intersection with the fiber. A generator of
the relative group is represented by a normal disk to $\Sigma$, whose
boundary is a positively oriented meridian. Consequently,
\[
\partial(1)=\pm\overline\omega(f).
\]
Exactness now gives
\[
\text{there is an }A\in H_2(X;\Z)
\text{ with }A\cdot[\Sigma]=d
\quad\Longleftrightarrow\quad
 d\,\overline\omega(f)=0,
\]
which proves the remaining equivalence.
\end{proof}

When $f$ is symplectic, the multisections above can be made symplectic
for a symplectic form compatible with $f$. Indeed, start with a
compatible symplectic form $\omega_0$ on $X$ and an area form
$\omega_{S^2}$ on the base. For a given positive multisection $S$,
\[
\omega_0+\kappa f^*\omega_{S^2}
\]
is symplectic and restricts positively to $S$ for all sufficiently
large $\kappa$ \cite[Lemma~2.8]{donaldson-smith:pencils}. At a branch point,
$T S$ is the positively oriented fiber tangent plane, where
$\omega_0$ is already positive; away from small neighborhoods of these
points, the base term gives the required positivity.

\begin{remark}
The existence result in Theorem~\ref{thm:multisection-degrees} is also
natural in light of the role of standard surfaces in the
Donaldson--Smith and Perutz approach to Seiberg--Witten theory. In that
framework, sections of relative Hilbert schemes for Lefschetz pencils
give standard surfaces downstairs, i.e. multisections of the pencil,
and these are the geometric objects through which Seiberg--Witten
basic classes, including the canonical class, enter the theory
\cite{donaldson-smith:pencils,Perutz:matching-2}.
\end{remark}

Let
\[
m(f)=
\min\bigl\{
A\cdot[\Sigma]>0
\ \big|\
A\in H_2(X;\Z)
\bigr\},
\]
whenever this set is nonempty. Equivalently, $m(f)$ is the divisibility
of the fiber class $[\Sigma]$ in
$H_2(X;\Z)/\operatorname{Tor}$.

\begin{cor}\label{cor:multisection-divisibility}
For every Lefschetz fibration $f\colon X\to S^2$,
\[
m(f)=\operatorname{ord}\bigl(\overline\omega(f)\bigr),
\]
where the zero element is understood to have order $1$. Moreover, the
multisection degrees $d\geq2$ are exactly those
divisible by $m(f)$. In particular, every Lefschetz fibration over
$S^2$ admits a multisection, and if $[\Sigma]$ is primitive, then $f$
admits a $d$--section for every $d\geq2$, hence a $2$--section.
\end{cor}

\begin{proof}
First note that the fiber class is homologically essential in our
conventions. The genus--$0$ case does not occur under our relative
minimality convention. In genus $1$, the classification of nontrivial
genus--$1$ Lefschetz fibrations over $S^2$ identifies the fibration with
an elliptic fibration on $E(n)$, for some $n\geq1$; in particular, it
admits a section and $[\Sigma]$ is primitive
\cite[Section~4]{baykur-kamada}. For fiber genus $g\geq2$, an
almost complex structure adapted to the fibration gives
\[
\langle c_1(TX,J),[\Sigma]\rangle
=
\chi(\Sigma)=2-2g\neq0,
\]
so $[\Sigma]$ is non-torsion. Thus, in all cases, $[\Sigma]$ is
non-torsion; by unimodularity of the intersection form on the free part
of $H_2(X;\Z)$, the integer $m(f)$ is defined.

The image of the intersection map
\[
H_2(X;\Z)\xrightarrow{\;\cdot[\Sigma]\;}\Z
\]
is $m(f)\Z$. On the other hand, the exact sequence in the proof of
Theorem~\ref{thm:multisection-degrees} shows that its image is the
kernel of the map sending $1$ to
$\pm\overline\omega(f)$. This kernel is
$\operatorname{ord}(\overline\omega(f))\Z$, which proves the equality.
The statement about degrees now follows from
Theorem~\ref{thm:multisection-degrees}. In particular, one may take any
multiple of $m(f)$ which is at least two. Finally, $[\Sigma]$ is
primitive precisely when $m(f)=1$.
\end{proof}

\begin{cor}\label{cor:abelian-multisections}
If $G$ is abelian, then
\[
m(f)=\operatorname{ord}\bigl(\omega(f)\bigr).
\]
In particular, if $G=1$, then $f$ admits a $d$--section for every
$d\geq2$. This latter conclusion holds whenever the Dehn twists about
the vanishing cycles generate the mapping class group of the fiber.
\end{cor}

\begin{proof}
When $G$ is abelian, the conjugacy class $\omega(f)$ determines an
element of $G$, and the abelianization map
$G\to H_1(G;\Z)$ is an isomorphism. Hence
$\overline\omega(f)=\omega(f)$, and the equality follows from
Corollary~\ref{cor:multisection-divisibility}. If $G=1$, then $m(f)=1$,
and therefore every $d\geq2$ occurs as the degree of a multisection.

Finally, suppose that the Dehn twists about the vanishing cycles
generate $\M(\Sigma_g)$. The normal subgroup $N\lhd F$ is invariant
under each of these twists, since every $t_{c_i}$ acts trivially on
$F/N$, and hence it is invariant under the whole mapping class group.
At least one vanishing cycle is nonseparating. This is automatic in
genus $1$, while for $g\geq2$ the subgroup generated by separating
twists acts trivially on $H_1(\Sigma_g;\Z)$ and cannot be the whole
mapping class group. The mapping class group acts transitively on
nonseparating simple closed curves, so $N$ contains every such curve.
In particular, it contains a standard generating set for $F$, and
therefore $N=F$ and $G=1$.
\end{proof}

\begin{remark}
For the examples in Theorem~A, the obstruction $\omega(f)$ is
nontrivial in $G$, whereas its image $\overline\omega(f)$ in
$H_1(G;\Z)$ vanishes. Thus these fibrations admit $2$--sections, even
though they do not admit sections. This illustrates that the
obstruction to a section is genuinely nonabelian.
\end{remark}

The class $\omega(f)$ can also constrain the topology of multisections, rather than just their degrees.

\begin{proposition}\label{prop:spherical-degrees}
For a Lefschetz fibration $f\colon X\to S^2$,
\[
{\rm im}\bigl(f_*:\pi_2(X)\to\pi_2(S^2)\cong\mathbb Z\bigr)
=
\{d\in\mathbb Z:\omega(f)^d=1\}.
\]
Consequently, if $\omega(f)$ has finite order $r$, every spherical
multisection has degree divisible by $r$. If $\omega(f)$ has
infinite order, no multisection has a spherical component.
\end{proposition}

\begin{proof}
Let $\mu\in G$ be a positively oriented meridian of the regular fiber $\Sigma$.
By Section~\ref{subsec:last2handle}, it is central and agrees with $\omega(f)$
up to inversion.

Given a smooth map $u\colon S^2\to X$, make it transverse to $\Sigma$
and remove small disks around the intersection points. The
remaining punctured sphere maps into $X\setminus\Sigma$.
Its boundary relation, together with centrality of $\mu$, gives
\[
1=\prod_j g_j\mu^{\varepsilon_j}g_j^{-1}
=\mu^{\sum_j\varepsilon_j},
\]
where $\varepsilon_j$ are the local intersection signs. Since
\[
\sum_j\varepsilon_j=u_*[S^2]\cdot[\Sigma]=\deg(f\circ u),
\]
every spherical degree belongs to the indicated set.

Conversely, suppose $\mu^d=1$ for some $d>0$. Fix $p\in\Sigma$.
The disk $z\mapsto(p,z^d)$ in $\nu\Sigma\cong\Sigma\times D^2$ has
boundary representing $\mu^d$. Capping it with a null-homotopy
in $X\setminus{\rm int}\,\nu\Sigma$ and smoothing gives a sphere map
whose algebraic intersection with $\Sigma$, and hence degree over the
base, is $d$. Reversing the domain orientation
gives negative degrees, and a constant map gives degree zero.
The assertions about multisections follow because each
component has positive degree.
\end{proof}

In particular, $\omega(f)=1$ if and only if $f$ admits a
homotopy section. 

Recall that a class in $H^2(X;\mathbb R)$ is called \emph{aspherical}
if it evaluates trivially on every spherical class.
A symplectic form $\Omega$ on $X$ is called
\emph{symplectically aspherical} if
$\int_{S^2}u^*\Omega=0$ for every smooth map $u\colon S^2\to X$
\cite[Section~1]{kedra-rudyak-tralle:aspherical}.

\begin{corollary}\label{cor:aspherical-form}
Let $f\colon X\to S^2$ be a Lefschetz fibration with irreducible fibers
and regular fiber $\Sigma$. Suppose that $\omega(f)$ has infinite order and that an
aspherical class $\eta\in H^2(X;\mathbb R)$ satisfies
$\langle\eta,[\Sigma]\rangle>0$.
Then $f$ admits compatible symplectically aspherical forms
in the classes $\eta+ \kappa \,{\rm PD}[\Sigma]$
for all sufficiently large $\kappa$.
\end{corollary}

\begin{proof}
Let $\sigma$ be an area form of total area $1$ on $S^2$.
By Proposition~\ref{prop:spherical-degrees},
\[
\langle f^*[\sigma],u_*[S^2]\rangle=\deg(f\circ u)=0
\]
for every map $u\colon S^2\to X$. Thus
$f^*[\sigma]={\rm PD}[\Sigma]$ is aspherical.

Since the fibers are irreducible, $\eta$ evaluates positively on every
irreducible fiber component. The Gompf--Thurston construction gives
compatible symplectic forms in the classes $t\eta+f^*[\sigma]$ for all
sufficiently small $t>0$ \cite[Theorem~2.7(b)]{gompf:locally-holomorphic}.
Rescaling by $t^{-1}$ gives the indicated classes with $\kappa=t^{-1}$.
Both summands are aspherical, so these forms are
symplectically aspherical.
\end{proof}

\section{Constructions of examples} \label{sec:examples}
In this section, we first give a novel construction of Lefschetz fibrations with signature zero over the $2$--sphere of every genus $g=2h+1$, for $h \geq 3$. For simplicity, we first focus on the $g=7$ case, which allows us to highlight the essential ideas of the construction and also to work out an explicit example, and then generalize the construction to higher genera. We also revisit a genus--$9$ Lefschetz fibration (also with signature zero) we constructed earlier in \cite{baykur-hamada:Lefschetz}, now with a detailed fundamental group calculation to examine the existence of a section. 

\subsection{The genus--$7$ fibration} \label{subsec:genus7}
%======================================================

Our building block for this construction is the positive factorization of a genus--$3$ Lefschetz pencil on the four-torus given in \cite{baykur-genus3}, which, as shown in \cite{hamada-hayano:LFson4torus}, coincides with the monodromy of Smith's holomorphic pencil given in \cite{smith:LFson4torus}.

We use the following factorization in $\M(\Sigma_3^4)$, as given in \cite{baykur-hamada:Lefschetz}:
\begin{equation} \label{eq:SmithLP}
    t_{a} t_{a^\prime} t_{x} t_{b} t_{b^\prime} t_{y} t_{c} t_{c^\prime} t_{z} t_{d} t_{d^\prime} t_{w} = t_{\delta_1} t_{\delta_2} t_{\delta_3} t_{\delta_4},
\end{equation}
where the vanishing cycles are as depicted in Figure~\ref{fig:Genus3_SmithLP}.

\begin{figure}[htbp]
\centering
\includegraphics[height=150pt]{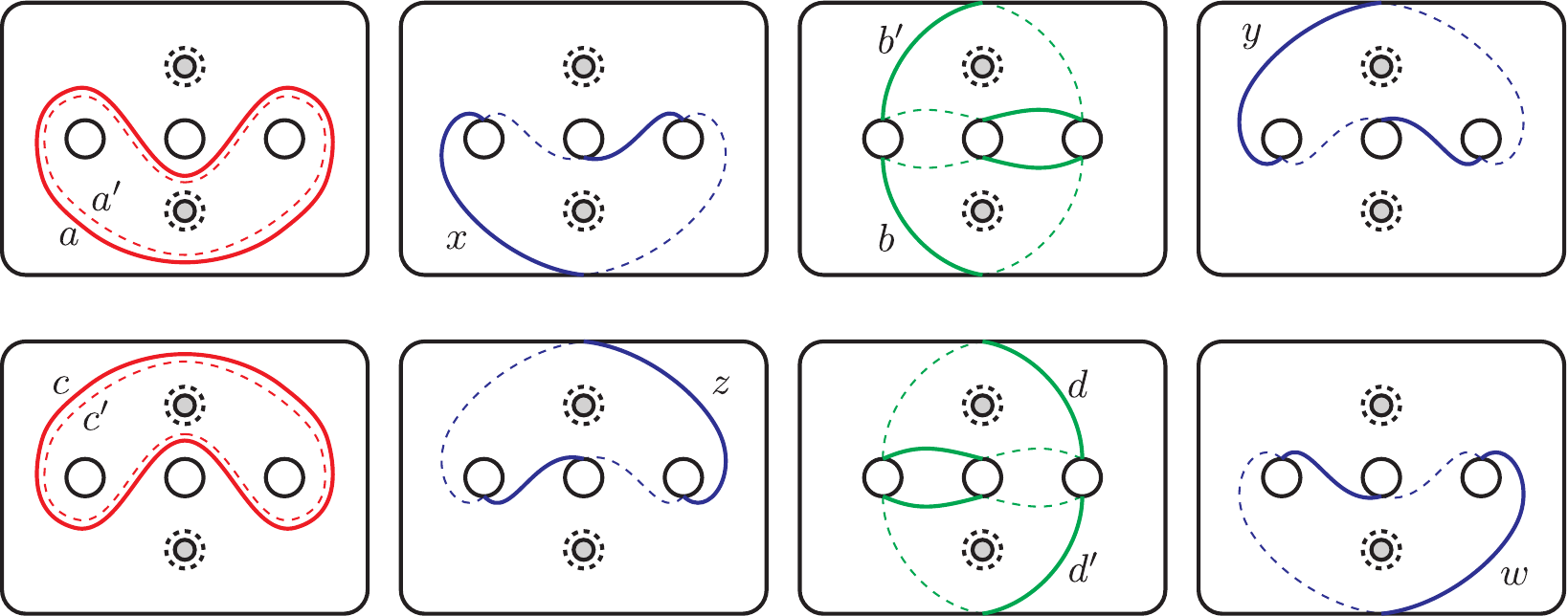}
\caption{Vanishing cycles for Smith's pencil. The reference surface is $\Sigma_3^4$, where the two boundary components drawn with solid lines are on the near side of the surface, while the two drawn with dashed lines are on the far side.}
\label{fig:Genus3_SmithLP}
\end{figure}
One remarkable feature of this particular monodromy factorization is that it contains surprisingly many bounding pairs.
In fact, in our previous construction of the genus--$9$ Lefschetz fibration with signature zero in \cite{baykur-hamada:Lefschetz}, we took six copies of this relation and effectively used the bounding pairs $(a,a^\prime)$, $(b,b^\prime)$, $(c,c^\prime)$, and $(d,d^\prime)$ to successfully breed them.
Here, we will take advantage of another set of bounding pairs $(x,z)$ and $(b,d)$ instead.
Note that in the monodromy their Dehn twists appear in the order $t_{x}, t_{b}, t_{z}, t_{d}$, and since $(x,z)$ and $(b,d)$ have intricate intersections, simple Hurwitz moves will not cluster the products $t_{x}t_{z}$ and $t_{b}t_{d}$ at the same time in a new monodromy factorization.
(We do not know if there exist any Hurwitz moves that do cluster them in this way.)
Although this may seem to be a significant drawback, these two bounding pairs enjoy another useful property: they are supported on a common subsurface in a configuration that will play a crucial role in our genus--$7$ construction.

Let us now describe this configuration.
First, we cluster $t_{b}, t_{z}, t_{d}, t_{x}$ in the factorization~\eqref{eq:SmithLP} by some Hurwitz moves to obtain a relation of the form 
\begin{equation} \label{eq:SmithLP_bzdx}
    P \, t_{b} t_{z} t_{d} t_{x} = t_{\delta_1} t_{\delta_2} t_{\delta_3} t_{\delta_4},
\end{equation}
where $P$ is the product of some eight Dehn twists.
Second, we conjugate this factorization by the product of Dehn twists
\begin{equation*}
    t_5 t_4 t_3 t_2 t_1
\end{equation*}
with the curves specified in Figure~\ref{fig:Genus3_BoundingPairs_Modification}, and then slide the boundary components of $\Sigma_3^4$ as also indicated in Figure~\ref{fig:Genus3_BoundingPairs_Modification}. 
From now on, we continue to denote the resulting surface by $\Sigma_3^4$.
Under this modification, the bounding pairs $(x,z)$ and $(b,d)$ are transformed, respectively, to the pairs $(X,Z)$ and $(B,D)$ in Figure~\ref{fig:Genus3_BoundingPairs_Configuration}.

\begin{figure}[htbp]
\centering
\includegraphics[height=70pt]{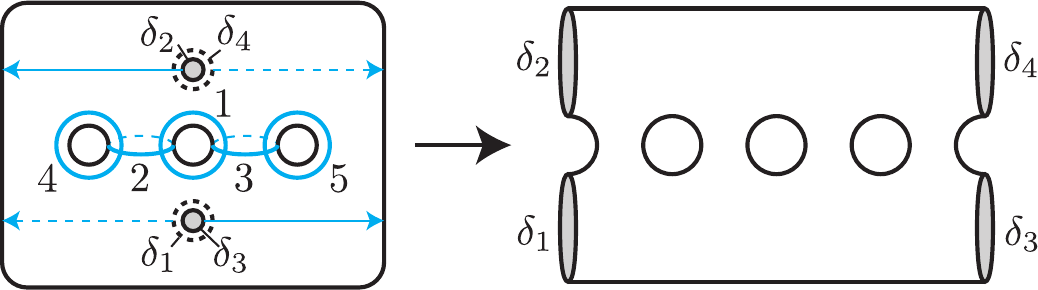}
\caption{The Dehn twist curves and sliding of the boundary components.}
\label{fig:Genus3_BoundingPairs_Modification}
\end{figure}
Thus, by Hurwitz moves and simultaneous conjugation, we obtain an equivalent monodromy factorization of the form
\begin{equation} \label{eq:SmithLP_BZDX}
    Q \, t_{B} t_{Z} t_{D} t_{X} = t_{\delta_1} t_{\delta_2} t_{\delta_3} t_{\delta_4},
\end{equation}
where $Q$ is the product of some eight Dehn twists.
Here, we observe a couple of properties of this particular configuration, which will be enough to construct our desired genus--$7$ Lefschetz fibration.
Consider the auxiliary curves labeled in Figure~\ref{fig:Genus3_BoundingPairs_Configuration} and let $S$ be the genus--$1$ subsurface of $\Sigma_3^4$ bounded by the curves $u_1, v_1, u_2, v_2$.

\begin{properties} \label{prop:XZBD-configuration}
The configuration has the following properties:
\begin{enumerate}
    \item 
    The bounding pairs $(X,Z)$, $(B,D)$, and also $(a,b)$ are all supported on the subsurface $S$.
    \item 
    Those three bounding pairs are all topologically equivalent. In fact, there are diffeomorphisms $\xi$ and $\eta$ that are supported on $S$ such that $\xi$ maps $(a,b)$ to $(X,Z)$ and $\eta$ maps $(a,b)$ to $(B,D)$.
    More explicitly, one can take $\xi=t_{c}^{-1} t_{d}^{-1} t_e^{-1}$ and $\eta=t_e^{-1}$.
    \item 
    There is a diffeomorphism $\mu$ supported on $S$ that interchanges $Z$ and $D$: $\mu(Z)=D$ and $\mu(D)=Z$.
    For example, take $\mu=t_{D} t_{c} t_{d}$.
\end{enumerate}
\end{properties}
For the last property, put $p=t_D$, $q=t_c$, and $r=t_d$.
The curves $c,D,d$ form a $3$--chain, and
$Z=(t_ct_d)^{-1}(D)$. Hence $\mu(Z)=D$. The braid relations and
$qr=rq$ give
\[
(pqr)^2p=p(pqr)^2.
\]
Thus $\mu^2(D)=D$, and consequently $\mu(D)=Z$ as well.
\begin{figure}[htbp]
\centering
\includegraphics[height=90pt]{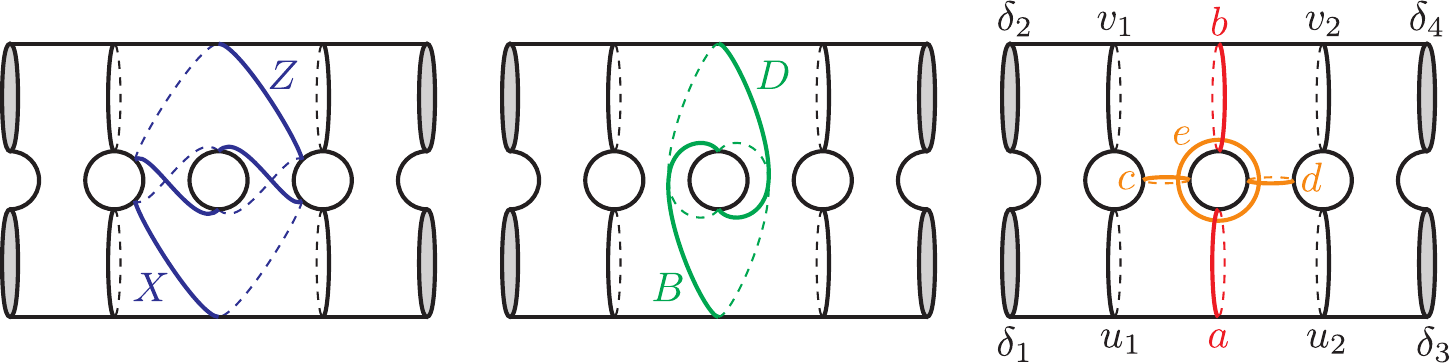}
\caption{The configuration of the bounding pairs $(X,Z)$ and $(B,D)$.}
\label{fig:Genus3_BoundingPairs_Configuration}
\end{figure}

Now we begin to construct a genus--$7$ Lefschetz fibration with signature $0$.

\noindent 
\underline{\it Schematic construction}:
We first embed $\Sigma_3^4$ into $\Sigma_7$ as indicated in the first
row of Figure~\ref{fig:EmbeddingGenus3IntoGenus7}, with the identifications
\[
\begin{aligned}
(\delta_1,\delta_2,\delta_3,\delta_4)&\longmapsto(a_3,b_3,a_2,b_2),\\
(u_1,v_1,u_2,v_2)&\longmapsto(u_1,v_1,u_2,v_2),\\
(a,b,X,Z,B,D)&\longmapsto(a_1,b_1,X_1,Z_1,B_1,D_1).
\end{aligned}
\]
Fix an orientation-preserving embedding with these identifications, as shown in Figure~\ref{fig:EmbeddingGenus3IntoGenus7}.

\begin{figure}[htbp]
\centering
\includegraphics[height=240pt]{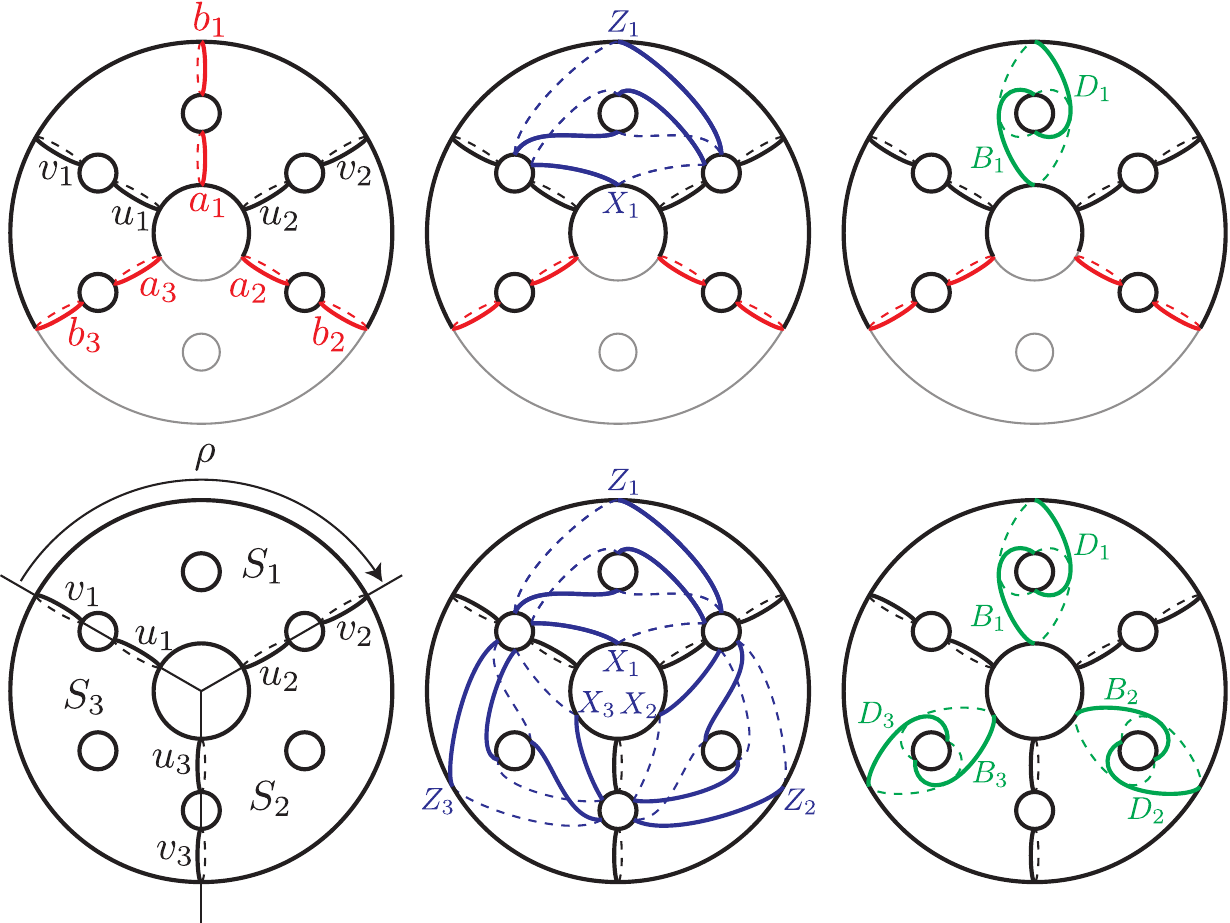}
\caption{The configuration of three circular copies of the genus--$3$ relation.}
\label{fig:EmbeddingGenus3IntoGenus7}
\end{figure}
From the relation~\eqref{eq:SmithLP_BZDX}, we deduce the relation
\begin{equation} \label{eq:embedded_SmithLP_BZDX}
    Q \, t_{B_1} t_{Z_1} t_{D_1} t_{X_1} = t_{a_3} t_{b_3} t_{a_2} t_{b_2},
\end{equation}
where, by slight abuse of notation, we continue to denote by $Q$ the image under the above embedding of the original product of eight Dehn twists. 
Now, as in Figure~\ref{fig:EmbeddingGenus3IntoGenus7}, consider the clockwise $\frac{2\pi}{3}$-rotation $\rho$ and the symmetric images of the involved curves under $\rho$ and $\rho^2$.
By these rotations, we obtain further copies of the relation~\eqref{eq:SmithLP_BZDX}:
\begin{align}
    Q^{\rho} \, t_{B_2} t_{Z_2} t_{D_2} t_{X_2} = t_{a_1} t_{b_1} t_{a_3} t_{b_3}, \\
    Q^{\rho^2} \, t_{B_3} t_{Z_3} t_{D_3} t_{X_3} = t_{a_2} t_{b_2} t_{a_1} t_{b_1}.
\end{align}
Let $S_i$ be the genus--$1$ subsurface of $\Sigma_7$ bounded by $\{u_i, v_i, u_{i+1}, v_{i+1}\}$ for $i=1,2,3$, with the subscripts understood modulo $3$.
In each of the three copied relations, $S_i$ corresponds to the subsurface $S \subset \Sigma_3^4$.
The subsurfaces $S_1$, $S_2$, and $S_3$ have pairwise disjoint interiors. We choose all supported mapping classes to be the identity on boundary collars, so mapping classes supported on different $S_i$ commute.
As observed in Properties~\ref{prop:XZBD-configuration}~(2), we have diffeomorphisms $\xi_i$ and $\eta_i$ supported on $S_i$ such that $\xi_i$ maps $(a_i,b_i)$ to $(X_i,Z_i)$ and $\eta_i$ maps $(a_i,b_i)$ to $(B_i,D_i)$.
Hence, by conjugating the relation~\eqref{eq:embedded_SmithLP_BZDX} by the composition $\eta_3 \xi_2$, we get
\begin{equation} \label{eq:RBZDX1}
    R \, t_{B_1} t_{Z_1} t_{D_1} t_{X_1} = t_{B_3} t_{D_3} t_{X_2} t_{Z_2},
\end{equation}
where $R$ is the image of $Q$ under $\eta_3 \xi_2$.
Applying the rotations again, the other two copied relations transform to
\begin{align}
    R^{\rho} \, t_{B_2} t_{Z_2} t_{D_2} t_{X_2} = t_{B_1} t_{D_1} t_{X_3} t_{Z_3}, \label{eq:RBZDX2} \\
    R^{\rho^2} \, t_{B_3} t_{Z_3} t_{D_3} t_{X_3} = t_{B_2} t_{D_2} t_{X_1} t_{Z_1}. \label{eq:RBZDX3}
\end{align}
By commutativity, we can rewrite the three relations as
\begin{align}
    R &= t_{X_1}^{-1} t_{D_1}^{-1} t_{Z_1}^{-1} t_{B_1}^{-1} t_{X_2} t_{Z_2} t_{B_3} t_{D_3}, \label{eq:R} \\ 
    R^{\rho} &= t_{B_1} t_{D_1} t_{X_2}^{-1} t_{D_2}^{-1} t_{Z_2}^{-1} t_{B_2}^{-1} t_{X_3} t_{Z_3}, \label{eq:Rrho} \\ 
    R^{\rho^2} &= t_{X_1} t_{Z_1} t_{B_2} t_{D_2} t_{X_3}^{-1} t_{D_3}^{-1} t_{Z_3}^{-1} t_{B_3}^{-1}. \label{eq:Rrho^2}
\end{align}

We combine these three relations and then, by commutativity and cancelations, we rearrange the right hand side as follows:
\begin{align*}
    R R^{\rho} R^{\rho^2} &= t_{X_1}^{-1} t_{D_1}^{-1} t_{Z_1}^{-1} t_{B_1}^{-1} t_{X_2} t_{Z_2} t_{B_3} t_{D_3} \cdot t_{B_1} t_{D_1} t_{X_2}^{-1} t_{D_2}^{-1} t_{Z_2}^{-1} t_{B_2}^{-1} t_{X_3} t_{Z_3} \cdot t_{X_1} t_{Z_1} t_{B_2} t_{D_2} t_{X_3}^{-1} t_{D_3}^{-1} t_{Z_3}^{-1} t_{B_3}^{-1} \\
    &= t_{X_1}^{-1} t_{D_1}^{-1} t_{Z_1}^{-1} t_{B_1}^{-1} t_{B_1} t_{D_1} t_{X_1} t_{Z_1} \cdot t_{X_2} t_{Z_2} t_{X_2}^{-1} t_{D_2}^{-1} t_{Z_2}^{-1} t_{B_2}^{-1} t_{B_2} t_{D_2} \cdot t_{B_3} t_{D_3} t_{X_3} t_{Z_3} t_{X_3}^{-1} t_{D_3}^{-1} t_{Z_3}^{-1} t_{B_3}^{-1} \\
    &= t_{X_1}^{-1} t_{D_1}^{-1} t_{Z_1}^{-1} t_{D_1} t_{Z_1} t_{X_1} \cdot t_{Z_2} t_{D_2}^{-1} t_{Z_2}^{-1} t_{D_2} \cdot t_{B_3} t_{D_3} t_{Z_3} t_{D_3}^{-1} t_{Z_3}^{-1} t_{B_3}^{-1} \\
    &= t_{X_1}^{-1} t_{D_1}^{-1} t_{Z_1}^{-1} [t_{D_1}, t_{Z_1}] t_{Z_1} t_{D_1} t_{X_1} \cdot t_{D_2}^{-1} [t_{D_2}, t_{Z_2}] t_{D_2} \cdot t_{B_3} [t_{D_3}, t_{Z_3}] t_{B_3}^{-1} \\
    &= t_{X_1}^{-1} t_{D_1}^{-1} t_{Z_1}^{-1} t_{D_2}^{-1} t_{B_3} \cdot [t_{D_1}, t_{Z_1}]  \cdot  [t_{D_2}, t_{Z_2}]  \cdot  [t_{D_3}, t_{Z_3}] \cdot t_{B_3}^{-1} t_{D_2} t_{Z_1} t_{D_1} t_{X_1}.
\end{align*}
Set $f=t_{B_3}^{-1} t_{D_2} t_{Z_1} t_{D_1} t_{X_1}$ and take the conjugate of the above relation by $f$:
\begin{equation} \label{eq:[D,Z]}
    (R R^{\rho} R^{\rho^2})^f = [t_{D_1}, t_{Z_1}]  \cdot  [t_{D_2}, t_{Z_2}]  \cdot  [t_{D_3}, t_{Z_3}].
\end{equation}
Now, as in Properties~\ref{prop:XZBD-configuration}, choose a diffeomorphism $\mu_i$ supported on $S_i$ that interchanges $D_i$ and $Z_i$ for $i=1,2,3$, and set $\mu = \mu_1 \mu_2 \mu_3$.
Then, by taking the conjugate of the relation~\eqref{eq:[D,Z]} by $\mu$, we obtain
\begin{equation} \label{eq:[Z,D]}
    ((R R^{\rho} R^{\rho^2})^f)^{\mu} = [t_{Z_1}, t_{D_1}]  \cdot  [t_{Z_2}, t_{D_2}]  \cdot  [t_{Z_3}, t_{D_3}].
\end{equation}

Finally, combining the relations~\eqref{eq:[D,Z]} and \eqref{eq:[Z,D]}, we see the commutators cancel out on the right hand side:
\begin{align*}
    (R R^{\rho} R^{\rho^2})^f ((R R^{\rho} R^{\rho^2})^f)^{\mu} &= [t_{D_1}, t_{Z_1}]  \cdot  [t_{D_2}, t_{Z_2}]  \cdot  [t_{D_3}, t_{Z_3}] \cdot [t_{Z_1}, t_{D_1}]  \cdot  [t_{Z_2}, t_{D_2}]  \cdot  [t_{Z_3}, t_{D_3}] \\
    &= [t_{D_1}, t_{Z_1}] [t_{Z_1}, t_{D_1}] \cdot  [t_{D_2}, t_{Z_2}] [t_{Z_2}, t_{D_2}] \cdot  [t_{D_3}, t_{Z_3}] [t_{Z_3}, t_{D_3}] \\
    &= 1.
\end{align*}
Therefore, we obtain a positive factorization
\begin{equation} \label{eq:genus7signature0LF_schematic}
    (R R^{\rho} R^{\rho^2})^f ((R R^{\rho} R^{\rho^2})^f)^{\mu} = 1.
\end{equation}
Since $R$ consists of eight Dehn twists, the left hand side contains $48$ Dehn twists.
For the signature, the genus--$3$ relation~\eqref{eq:SmithLP_BZDX}
is obtained from two $2$--chain and fourteen lantern relations, whose
Endo--Nagami contributions are $-7$ and $+1$, respectively
\cite{endo-nagami:signature,baykur-hamada:Lefschetz}.
We use six embedded conjugates of this relation. The remaining
operations are conjugations, Hurwitz moves, commutations, and
cancellations, which do not change the signature count. Thus the
signature of~\eqref{eq:genus7signature0LF_schematic} is
$12(-7)+84=0$.
Therefore, the genus--$7$ Lefschetz fibration associated with monodromy factorization~\eqref{eq:genus7signature0LF_schematic} has $48$ critical points and the signature of its total space is $0$.

\noindent
\underline{\it Explicit construction}:
Now we work out an explicit example.
Although the schematic construction above is systematic and theoretically clear, it turns out that the actual vanishing cycles obtained in this way are highly intricate and intertwined.
Consequently, the resulting monodromy factorizations are rather cumbersome to work with and are less amenable to explicit computations (particularly those involving fundamental groups) or to further applications. 
For this reason, we incorporate several ad hoc modifications throughout the construction in order to obtain simpler vanishing cycles and a more tractable monodromy factorization.

Starting with the relation~\eqref{eq:SmithLP}, we apply cyclic permutations and Hurwitz moves to cluster $t_{b} t_{z} t_{d} t_{x}$ as follows:
\begin{align*}
    t_{\delta_1} t_{\delta_2} t_{\delta_3} t_{\delta_4} &= t_{a} t_{a^\prime} t_{x} t_{b} t_{b^\prime} t_{y} t_{c} t_{c^\prime} t_{z} t_{d} t_{d^\prime} t_{w} \\
    &= t_{d^\prime} t_{w} \underline{t_{a} t_{a^\prime}} t_{x} \underline{t_{b} t_{b^\prime}} t_{y} t_{c} t_{c^\prime} t_{z} t_{d} \\
    &= t_{d^\prime} \underline{t_{w} t_{a^\prime}} t_{a} t_{x} t_{b^\prime} t_{b} t_{y} t_{c} t_{c^\prime} t_{z} t_{d} \\
    &= \underline{t_{d^\prime} t_{a^\prime} \DT{t_{a^\prime}^{-1}(w)} t_{a} t_{x}} t_{b^\prime} \underline{t_{b} t_{y} t_{c} t_{c^\prime}} t_{z} t_{d} \\
    &= t_{x} \DT{t_{x}^{-1}(d^\prime)} \DT{t_{x}^{-1}(a^\prime)} \DT{t_{x}^{-1} t_{a^\prime}^{-1}(w)} \DT{t_{x}^{-1}(a)} t_{b^\prime} \DT{t_{b}(y)} \DT{t_{b}(c)} \DT{t_{b}(c^\prime)} t_{b} t_{z} t_{d} \\
    &= \DT{t_{x}^{-1}(d^\prime)} \DT{t_{x}^{-1}(a^\prime)} \DT{t_{x}^{-1} t_{a^\prime}^{-1}(w)} \DT{t_{x}^{-1}(a)} t_{b^\prime} \DT{t_{b}(y)} \DT{t_{b}(c)} \DT{t_{b}(c^\prime)} t_{b} t_{z} t_{d} t_{x}.
\end{align*}

Then we follow the process of obtaining the relation~\eqref{eq:SmithLP_BZDX} from \eqref{eq:SmithLP_bzdx}, which was explained using Figure~\ref{fig:Genus3_BoundingPairs_Modification}, by conjugating the factorization by $t_5 t_4 t_3 t_2 t_1$ and sliding the boundary components.
This results in the following relation
\begin{equation}
    t_{q_1} t_{q_2} t_{q_3} t_{q_4} t_{q_5} t_{q_6} t_{q_7} t_{q_8} t_{B} t_{Z} t_{D} t_{X} = t_{\delta_1} t_{\delta_2} t_{\delta_3} t_{\delta_4},
\end{equation}
Here $q_i$ are, in order, the images of the first eight curves in the last line of the preceding Hurwitz calculation under this change of identification. They are depicted in Figure~\ref{fig:Genus3_OtherVanishingCycles}, together with $B,Z,D,X$ in Figure~\ref{fig:Genus3_BoundingPairs_Configuration}.
\begin{figure}[htbp]
\centering
\includegraphics[height=130pt]{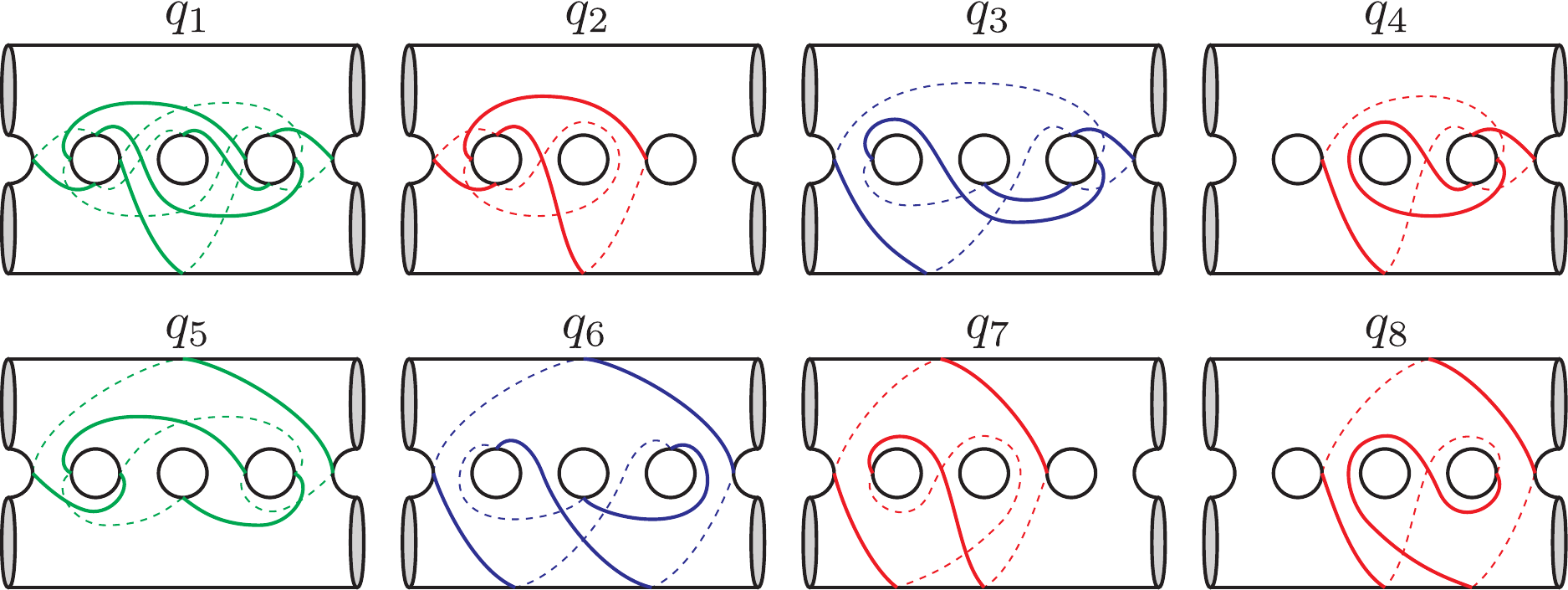}
\caption{Explicit Dehn twist curves for the factorization $Q$ in $\Sigma_3^4$.}
\label{fig:Genus3_OtherVanishingCycles}
\end{figure}

Next, we embed $\Sigma_3^4$ into $\Sigma_7$ in a way explained earlier, but we also need to specify where the extra curves $q_i$ are mapped, which is described in Figure~\ref{fig:Genus7_Q}.
With these curves, we have the relation
\begin{equation} \label{eq:relation_QBZDX_genus7}
    t_{q_1} t_{q_2} t_{q_3} t_{q_4} t_{q_5} t_{q_6} t_{q_7} t_{q_8} t_{B_1} t_{Z_1} t_{D_1} t_{X_1} = t_{a_3} t_{b_3} t_{a_2} t_{b_2}
\end{equation}
in $\M(\Sigma_7)$.
\begin{figure}[htbp]
\centering
\includegraphics[height=180pt]{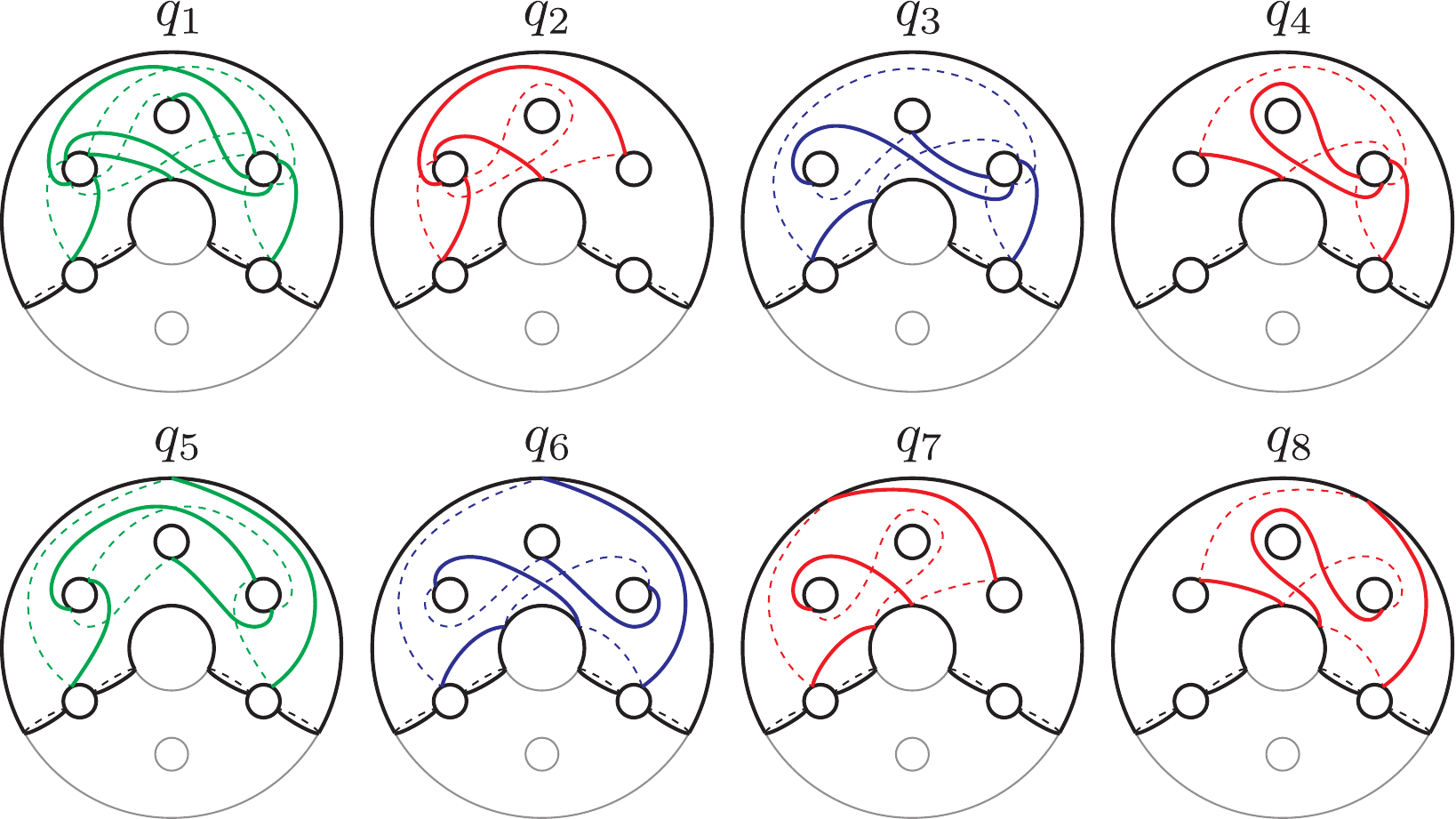}
\caption{Explicit Dehn twist curves for the factorization $Q$ in $\Sigma_7$.}
\label{fig:Genus7_Q}
\end{figure}
We choose $\xi_2$ and $\eta_3$ as $\xi_2=t_{c_2}^{-1} t_{d_2}^{-1} t_{e_2}^{-1}$ and $\eta_3=t_{e_3}^{-1}$ with the Dehn twist curves given in Figure~\ref{fig:Genus7_AuxiliaryCurves} and take the conjugate of the relation~\eqref{eq:relation_QBZDX_genus7} by $\eta_3\xi_2$ to obtain
\begin{equation}
    t_{r_1} t_{r_2} t_{r_3} t_{r_4} t_{r_5} t_{r_6} t_{r_7} t_{r_8} t_{B_1} t_{Z_1} t_{D_1} t_{X_1} = t_{B_3} t_{D_3} t_{X_2} t_{Z_2},
\end{equation}
where $r_i$ is the image of $q_i$ under $\eta_3\xi_2$ ($i=1,\dots,8$); the resulting curves are illustrated in Figure~\ref{fig:Genus7_R}.
\begin{figure}[htbp]
\centering
\includegraphics[height=110pt]{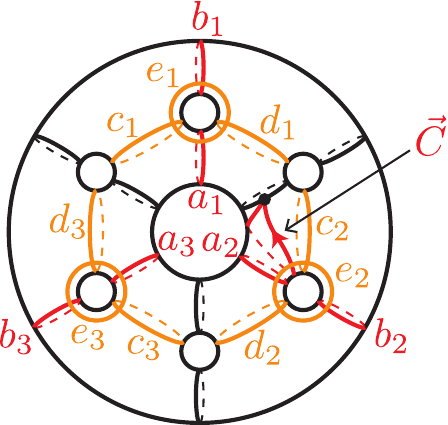}
\caption{The auxiliary curves. The base point and the based loop $\vec{C}$ are used later in Subsection~\ref{subsec:genus7-pi1}.}
\label{fig:Genus7_AuxiliaryCurves}
\end{figure}
\begin{figure}[htbp]
\centering
\includegraphics[height=180pt]{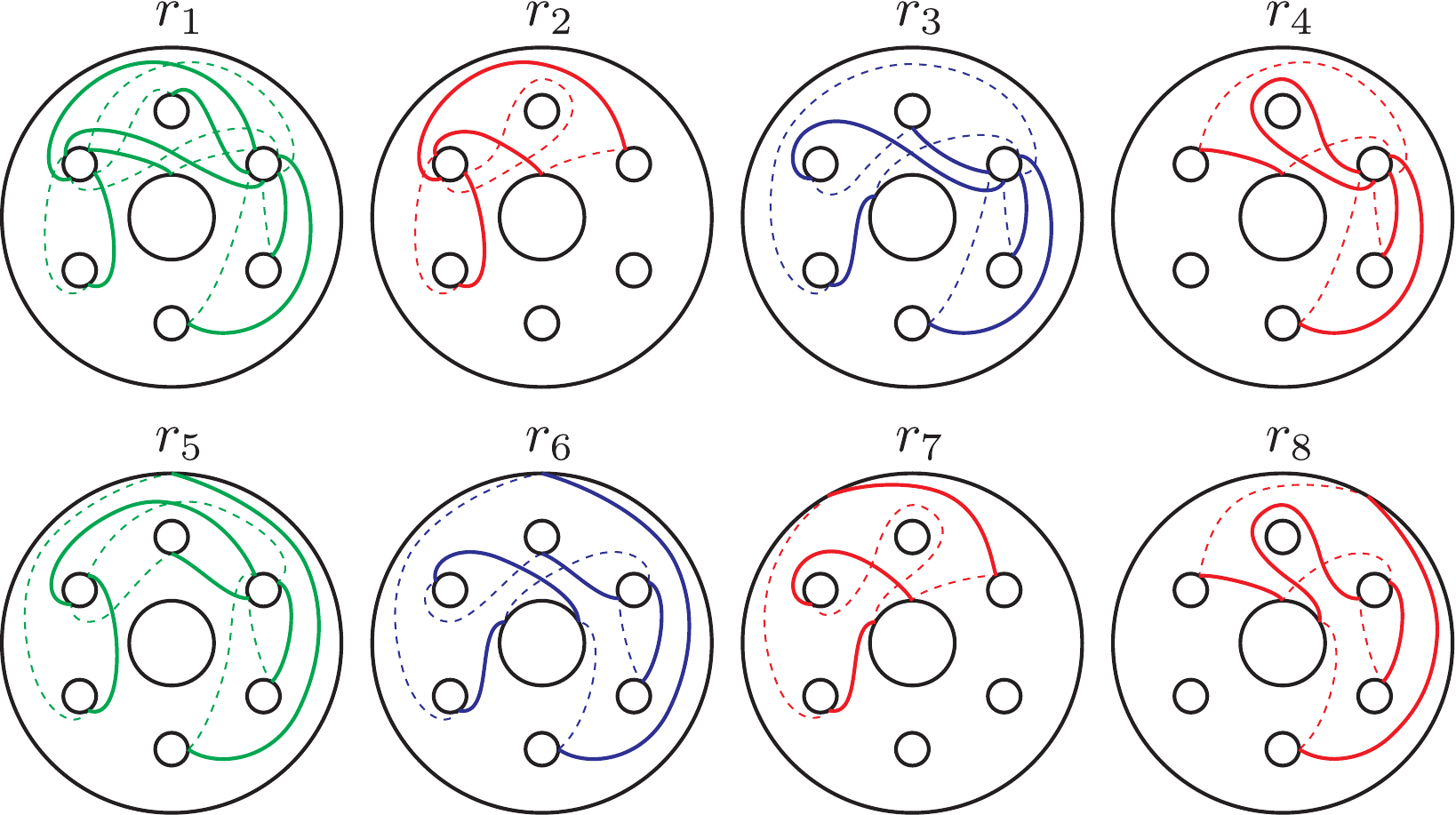}
\caption{Explicit Dehn twist curves for the factorization $R$ in $\Sigma_7$.}
\label{fig:Genus7_R}
\end{figure}
Thus, with $R=t_{r_1} t_{r_2} t_{r_3} t_{r_4} t_{r_5} t_{r_6} t_{r_7} t_{r_8}$, we have the relations~\eqref{eq:R}, \eqref{eq:Rrho}, and \eqref{eq:Rrho^2}.
Furthermore, notice that we can conjugate these relations freely by $t_{u_i}$ and $t_{v_i}$ without changing the right-hand sides as $u_i$ and $v_i$ have no intersections with $X_j$, $D_j$, $Z_j$, $B_j$.
In particular, we choose the following conjugations:
\begin{align}
    R &= t_{X_1}^{-1} t_{D_1}^{-1} t_{Z_1}^{-1} t_{B_1}^{-1} t_{X_2} t_{Z_2} t_{B_3} t_{D_3}, \label{eq:R-0} \\ 
    (R^{\rho})^{t_{u_3}} &= t_{B_1} t_{D_1} t_{X_2}^{-1} t_{D_2}^{-1} t_{Z_2}^{-1} t_{B_2}^{-1} t_{X_3} t_{Z_3}, \label{eq:Rrho-1} \\ 
    (R^{\rho^2})^{t_{u_1} t_{u_3}} &= t_{X_1} t_{Z_1} t_{B_2} t_{D_2} t_{X_3}^{-1} t_{D_3}^{-1} t_{Z_3}^{-1} t_{B_3}^{-1}, \label{eq:Rrho^2-1} \\
    (R^{\rho})^{t_{v_3} t_{u_3}} &= t_{B_1} t_{D_1} t_{X_2}^{-1} t_{D_2}^{-1} t_{Z_2}^{-1} t_{B_2}^{-1} t_{X_3} t_{Z_3}, \label{eq:Rrho-2} \\ 
    (R^{\rho^2})^{t_{v_1} t_{u_1} t_{v_3} t_{u_3}} &= t_{X_1} t_{Z_1} t_{B_2} t_{D_2} t_{X_3}^{-1} t_{D_3}^{-1} t_{Z_3}^{-1} t_{B_3}^{-1}. \label{eq:Rrho^2-2} 
\end{align}
We take two combinations of these relations to produce the following.
\begin{align}
    R (R^{\rho})^{t_{u_3}} (R^{\rho^2})^{t_{u_1} t_{u_3}} &= t_{X_1}^{-1} t_{D_1}^{-1} t_{Z_1}^{-1} t_{D_1} t_{Z_1} t_{X_1} \cdot t_{Z_2} t_{D_2}^{-1} t_{Z_2}^{-1} t_{D_2} \cdot t_{B_3} t_{D_3} t_{Z_3} t_{D_3}^{-1} t_{Z_3}^{-1} t_{B_3}^{-1}, \\
    R (R^{\rho})^{t_{v_3} t_{u_3}} (R^{\rho^2})^{t_{v_1} t_{u_1} t_{v_3} t_{u_3}} &= t_{X_1}^{-1} t_{D_1}^{-1} t_{Z_1}^{-1} t_{D_1} t_{Z_1} t_{X_1} \cdot t_{Z_2} t_{D_2}^{-1} t_{Z_2}^{-1} t_{D_2} \cdot t_{B_3} t_{D_3} t_{Z_3} t_{D_3}^{-1} t_{Z_3}^{-1} t_{B_3}^{-1}.
\end{align}
Let $\mu_i=t_{D_i}t_{c_i}t_{d_i}$, which is supported on $S_i$ and interchanges $Z_i$ and $D_i$, as observed in Properties~\ref{prop:XZBD-configuration}.
We set
\begin{align*}
    \varphi_1 &= t_{X_1}, & \varphi_2 &= t_{a_2}^{-1} t_{b_2}^{-1} t_{c_2} t_{d_2} t_{Z_2}^{-1}, & \varphi_3 &= t_{B_3}^{-1}, \\
    \psi_1 &= \mu_1 t_{X_1}, & \psi_2 &= t_{a_2}^{-1} t_{b_2}^{-1} t_{c_2} t_{d_2} \mu_2 t_{Z_2}^{-1}, & \psi_3 &= \mu_3^{-1} t_{B_3}^{-1},
\end{align*}
and $\varphi = \varphi_1 \varphi_2 \varphi_3$ and $\psi = \psi_1 \psi_2 \psi_3$.
Then, it is straightforward to see
\begin{align}
    (R (R^{\rho})^{t_{u_3}} (R^{\rho^2})^{t_{u_1} t_{u_3}})^{\varphi} &= [t_{D_1}^{-1}, t_{Z_1}^{-1}] \cdot t_{a_2}^{-1} t_{b_2}^{-1} t_{c_2} t_{d_2} [t_{D_2}^{-1},t_{Z_2}^{-1}] t_{d_2}^{-1} t_{c_2}^{-1} t_{b_2} t_{a_2} \cdot [t_{D_3}, t_{Z_3}], \label{eq:[D,Z]-variant} \\
    (R (R^{\rho})^{t_{v_3} t_{u_3}} (R^{\rho^2})^{t_{v_1} t_{u_1} t_{v_3} t_{u_3}})^{\psi} &= [t_{Z_1}^{-1}, t_{D_1}^{-1}] \cdot t_{a_2}^{-1} t_{b_2}^{-1} t_{c_2} t_{d_2} [t_{Z_2}^{-1},t_{D_2}^{-1}] t_{d_2}^{-1} t_{c_2}^{-1} t_{b_2} t_{a_2} \cdot [t_{Z_3}, t_{D_3}]. \label{eq:[Z,D]-variant} 
\end{align}
We combine the two relations and see that the right-hand side cancels out similarly as before.
Therefore, we obtain
\begin{equation} \label{eq:Genus7_Signature0LF_Subwords}
    (R (R^{\rho})^{t_{u_3}} (R^{\rho^2})^{t_{u_1} t_{u_3}})^{\varphi} (R (R^{\rho})^{t_{v_3} t_{u_3}} (R^{\rho^2})^{t_{v_1} t_{u_1} t_{v_3} t_{u_3}})^{\psi} =1.
\end{equation}
Reading off the factorization, we set, for $i=1,\dots,8$,
\begin{align*}
    x_{i} &= \varphi(r_i), & x_{8+i} &= \varphi \, t_{u_3} \, \rho (r_i), & x_{16+i} &= \varphi \, t_{u_1} t_{u_3} \, \rho^2(r_i), \\
    x_{24+i} &= \psi (r_i), & x_{32+i} &= \psi \, t_{v_3} t_{u_3} \, \rho(r_i), & x_{40+i} &= \psi \, t_{v_1} t_{u_1} t_{v_3} t_{u_3} \, \rho^2(r_i).
\end{align*}
Namely, we have the explicit monodromy factorization
\begin{align} \label{eq:Genus7_Signature0LF}
t_{x_1} t_{x_2} \cdots t_{x_{48}} =1,
\end{align}
where the resulting vanishing cycles are illustrated in Figure~\ref{fig:Genus7_SignatureZeroLF_VanishingCycles}.
All these curves are essential: the original vanishing cycles are
nonseparating in $\Sigma_3^4$, and this remains true under the
embeddings and subsequent diffeomorphisms. As above, the signature
count uses six copies of the same genus--$3$ relation. The extra
conjugations do not change this count, so this explicit fibration also
has $48$ critical points and signature $0$.
\begin{figure}[htbp]
\centering
\includegraphics[height=575pt]{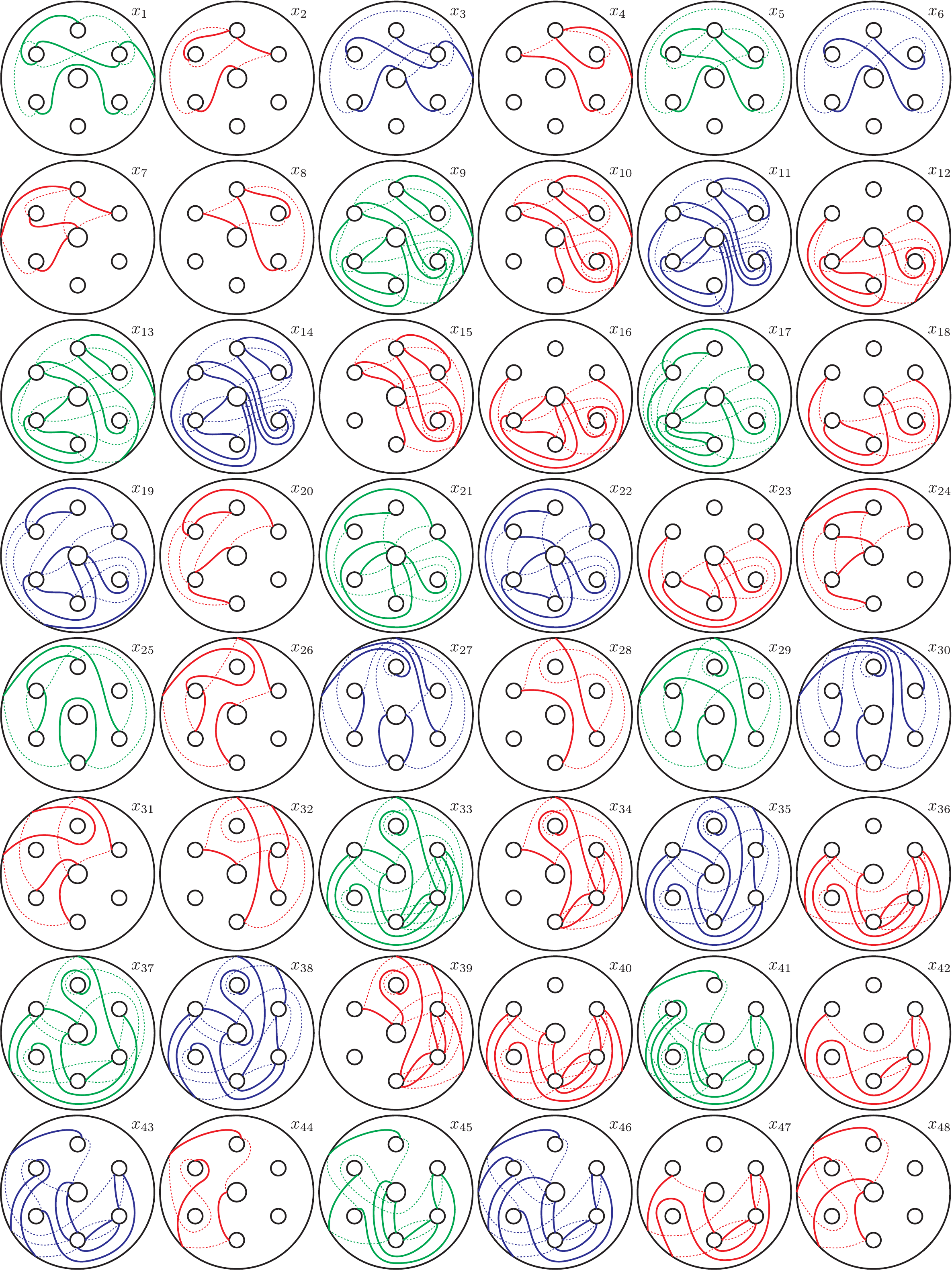}
\caption{Vanishing cycles of the genus--$7$ Lefschetz fibration with signature zero.}
\label{fig:Genus7_SignatureZeroLF_VanishingCycles}
\end{figure}
Note that, to simplify the computation of $x_i$, one can use expressions for $\varphi$ and $\psi$ in terms of $a_i, b_i, c_i, d_i, e_i$ as follows.
Observe first that $t_{B_i} = t_{a_i} t_{e_i} t_{a_i}^{-1}$, $t_{D_i} = t_{b_i} t_{e_i} t_{b_i}^{-1}$, $t_{X_i} = t_{c_i}^{-1} t_{d_i}^{-1} t_{a_i} t_{e_i} t_{a_i}^{-1} t_{d_i} t_{c_i}$, and $t_{Z_i} = t_{c_i}^{-1} t_{d_i}^{-1} t_{b_i} t_{e_i} t_{b_i}^{-1} t_{d_i} t_{c_i}$.
Using these identities and the disjointness of $a_i,b_i,c_i,d_i$, we obtain
\begin{align*}
    \varphi &= t_{c_1}^{-1} t_{d_1}^{-1} t_{a_1} t_{e_1} t_{a_1}^{-1} t_{d_1} t_{c_1} \cdot t_{a_2}^{-1} t_{e_2}^{-1} t_{c_2} t_{d_2} t_{b_2}^{-1} \cdot t_{a_3} t_{e_3}^{-1} t_{a_3}^{-1}, \\
    \psi &= t_{e_1}^{-1} t_{a_1} t_{b_1} t_{e_1} t_{c_1} t_{d_1} \cdot t_{a_2}^{-1} t_{b_2}^{-1} t_{c_2}^2 t_{d_2}^2 \cdot t_{c_3}^{-1} t_{d_3}^{-1} t_{e_3}^{-1} t_{a_3}^{-1} t_{b_3}^{-1} t_{e_3}.
\end{align*}

\subsection{Higher genera examples}\label{subsec:higher-genera}  
%======================================================

Here, we outline the schematic construction of a genus--$g=2h+1$ Lefschetz fibration with signature $0$ for $h \geq 3$, generalizing the genus--$7$ case.
Consider, in Figure~\ref{fig:EmbeddingGenus3IntoGenus2h+1}, the subsurface $S_1$ of $\Sigma_{2h+1}$ bounded by $\{ u_1, v_1, u_2, v_2 \}$, and the clockwise $\frac{2\pi}{h}$-rotation $\rho$.
We set pairs of curves $(X_1,Z_1)$ and $(B_1,D_1)$ on $S_1$ so that the configuration of $(u_1,v_1,u_2,v_2,X_1,Z_1,B_1,D_1)$ is topologically equivalent to that of $(u_1,v_1,u_2,v_2,X,Z,B,D)$ on $S$ in $\Sigma_3^4$ in Figure~\ref{fig:Genus3_BoundingPairs_Configuration}.
We also include the corresponding curves $a_1,b_1$ on $S_1$. We define $S_i$ and $u_i, v_i, a_i, b_i, X_i, Z_i, B_i, D_i$ as the circular images of $S_1$ and $u_1, v_1, a_1, b_1, X_1, Z_1, B_1, D_1$ under the rotation $\rho^{i-1}$ for $i=1,\dots,h$.

\begin{figure}[htbp]
\centering
\includegraphics[height=130pt]{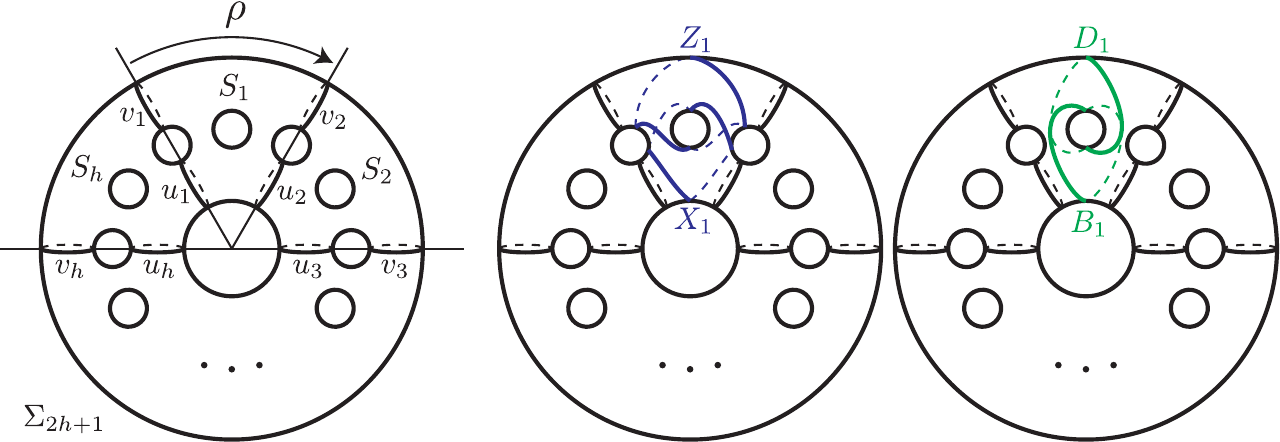}
\caption{The configuration for circular copies of the genus--$3$ relation.}
\label{fig:EmbeddingGenus3IntoGenus2h+1}
\end{figure}
The $S_i$ have disjoint interiors, and the supported maps are again
chosen to be the identity on boundary collars. For each $i$, embed
$\Sigma_3^4$ so that $S$ maps to $S_i$ and its four boundary curves
map to $a_{i-1},b_{i-1},a_{i+1},b_{i+1}$, with indices modulo $h$.
The three sectors involved are distinct because $h\geq3$.
Conjugating by $\eta_{i-1}\xi_{i+1}$, as in the genus--$7$ construction,
realizes copies of the genus--$3$ relation as follows.
\begin{align}
    R \, t_{B_1} t_{Z_1} t_{D_1} t_{X_1} &= t_{B_h} t_{D_h} t_{X_2} t_{Z_2}, \\
    R^{\rho} \, t_{B_2} t_{Z_2} t_{D_2} t_{X_2} &= t_{B_1} t_{D_1} t_{X_3} t_{Z_3}, \\
    &\vdots \notag \\
    R^{\rho^{h-1}} \, t_{B_h} t_{Z_h} t_{D_h} t_{X_h} &= t_{B_{h-1}} t_{D_{h-1}} t_{X_1} t_{Z_1}.
\end{align}
These can collectively be rewritten as
\begin{equation}
    R^{\rho^{i-1}} = t_{B_{i-1}} t_{D_{i-1}} t_{X_i}^{-1} t_{D_i}^{-1} t_{Z_i}^{-1} t_{B_i}^{-1} t_{Z_{i+1}} t_{X_{i+1}},
\end{equation}
for $i=1, \dots, h$, where the subscripts are taken modulo $h$.
Next, we take their product:
\begin{align*}
    R R^{\rho} \cdots R^{\rho^{h-1}} &= \underline{t_{B_{h}} t_{D_{h}}} t_{X_1}^{-1} t_{D_1}^{-1} t_{Z_1}^{-1} t_{B_1}^{-1} \underline{t_{Z_{2}} t_{X_{2}}} \\
    & \quad \underline{\cdot t_{B_{1}} t_{D_{1}}} t_{X_2}^{-1} t_{D_2}^{-1} t_{Z_2}^{-1} t_{B_2}^{-1} \underline{t_{Z_{3}} t_{X_{3}}} \\
    & \qquad \vdots \\
    & \quad \underline{\cdot t_{B_{h-1}} t_{D_{h-1}}} t_{X_h}^{-1} t_{D_h}^{-1} t_{Z_h}^{-1} t_{B_h}^{-1} \underline{t_{Z_{1}} t_{X_{1}}}.
\end{align*}
Conjugating both sides by $f_1 := t_{B_h}^{-1}t_{D_h}^{-1} t_{Z_1}t_{X_1}$ and using commutativity, we get
\begin{align*}
    (R R^{\rho} \cdots R^{\rho^{h-1}})^{f_1}
    &=  t_{Z_{1}} t_{X_{1}} t_{X_1}^{-1} t_{D_1}^{-1} t_{Z_1}^{-1} t_{B_1}^{-1} t_{B_{1}} t_{D_{1}} \\
    & \quad \cdot t_{Z_{2}} t_{X_{2}} t_{X_2}^{-1} t_{D_2}^{-1} t_{Z_2}^{-1} t_{B_2}^{-1} t_{B_{2}} t_{D_{2}} \\
    & \qquad \vdots \\
    & \quad \cdot t_{Z_{h}} t_{X_{h}} t_{X_h}^{-1} t_{D_h}^{-1} t_{Z_h}^{-1} t_{B_h}^{-1} t_{B_{h}} t_{D_{h}} \\
    &= t_{Z_{1}} t_{D_1}^{-1} t_{Z_1}^{-1} t_{D_{1}} \cdot t_{Z_{2}} t_{D_2}^{-1} t_{Z_2}^{-1} t_{D_{2}} \cdots t_{Z_{h}} t_{D_h}^{-1} t_{Z_h}^{-1} t_{D_{h}}.
\end{align*}
We further conjugate both sides by $f_2 := t_{D_1} t_{D_2} \cdots t_{D_h}$ to obtain 
\begin{align} \label{eq:[D,Z]_genus2h+1}
    ((R R^{\rho} \cdots R^{\rho^{h-1}})^{f_1})^{f_2} &= t_{D_{1}} t_{Z_{1}} t_{D_1}^{-1} t_{Z_1}^{-1} \cdot t_{D_{2}} t_{Z_{2}} t_{D_2}^{-1} t_{Z_2}^{-1} \cdots t_{D_{h}} t_{Z_{h}} t_{D_h}^{-1} t_{Z_h}^{-1} \\
    &= [t_{D_1},t_{Z_1}] \cdot [t_{D_2},t_{Z_2}] \cdots [t_{D_h},t_{Z_h}]. \notag
\end{align}
Now, pick a diffeomorphism $\mu_i$ supported on $S_i$ that interchanges $Z_i$ and $D_i$, as in Properties~\ref{prop:XZBD-configuration}(3), and set $\mu=\mu_1 \mu_2 \cdots \mu_h$.
We take, for the last time, the conjugate of \eqref{eq:[D,Z]_genus2h+1} by $\mu$, resulting in
\begin{equation} \label{eq:[Z,D]_genus2h+1}
    (((R R^{\rho} \cdots R^{\rho^{h-1}})^{f_1})^{f_2})^{\mu} = [t_{Z_1},t_{D_1}] \cdot [t_{Z_2},t_{D_2}] \cdots [t_{Z_h},t_{D_h}]. 
\end{equation}
Finally, we combine the relations~\eqref{eq:[D,Z]_genus2h+1} and \eqref{eq:[Z,D]_genus2h+1} to cancel the right hand sides out.
\begin{align*}
    & ((R R^{\rho} \cdots R^{\rho^{h-1}})^{f_1})^{f_2} (((R R^{\rho} \cdots R^{\rho^{h-1}})^{f_1})^{f_2})^{\mu} \\ &= [t_{D_1},t_{Z_1}] \cdot [t_{D_2},t_{Z_2}] \cdots [t_{D_h},t_{Z_h}] \cdot [t_{Z_1},t_{D_1}] \cdot [t_{Z_2},t_{D_2}] \cdots [t_{Z_h},t_{D_h}] \notag \\
    &= [t_{D_1},t_{Z_1}] [t_{Z_1},t_{D_1}] \cdot [t_{D_2},t_{Z_2}] [t_{Z_2},t_{D_2}] \cdots [t_{D_h},t_{Z_h}] [t_{Z_h},t_{D_h}] \notag \\
    &=1. \notag
\end{align*}
The resulting positive factorization
\begin{equation}
    ((R R^{\rho} \cdots R^{\rho^{h-1}})^{f_1})^{f_2} (((R R^{\rho} \cdots R^{\rho^{h-1}})^{f_1})^{f_2})^{\mu} = 1
\end{equation}
consists of $16h$ Dehn twists. These remain nonseparating under the
embeddings, so the resulting fibration is relatively minimal. There
are $2h$ copies of the genus--$3$ relation~\eqref{eq:SmithLP_BZDX};
the same signature calculation now gives $4h(-7)+28h=0$.

\begin{remark}
The novelty of the above constructions lies in expressing the mapping class $[t_{D_1},t_{Z_1}]\allowbreak \cdots [t_{D_h},t_{Z_h}]$ as a product of positive Dehn twists, together with the observation that $[t_{D_1},t_{Z_1}]\allowbreak \cdots [t_{D_h},t_{Z_h}]$ is conjugate to its inverse $[t_{Z_1},t_{D_1}]\allowbreak \cdots [t_{Z_h},t_{D_h}]$.
This suggests a potentially useful approach to constructing positive factorizations of the identity.
Namely, one seeks a positive factorization representing a mapping class $\varphi$ that is conjugate to its inverse $\varphi^{-1}$.
A positive factorization of $\varphi^{-1}$ is then obtained by conjugation, and the product of the two gives a positive factorization of the identity.
\end{remark}

\subsection{A separate genus--$9$ fibration}  
\label{subsec:genus9}
%======================================================

We begin with a particular genus--$9$ Lefschetz fibration $f\colon X\to S^2$ constructed in \cite{baykur-hamada:Lefschetz}, which will serve as the basis for our obstruction argument. Its vanishing cycles are shown in Figure~\ref{fig:SignatureZeroLF_VanishingCycles}. We also recall a marked-point lift of its monodromy factorization, which determines the associated point-pushing word. The calculation below shows that its abelianization vanishes, and hence that the fibration admits a pseudosection.

\begin{figure}[htbp]
\centering
\includegraphics[height=575pt]{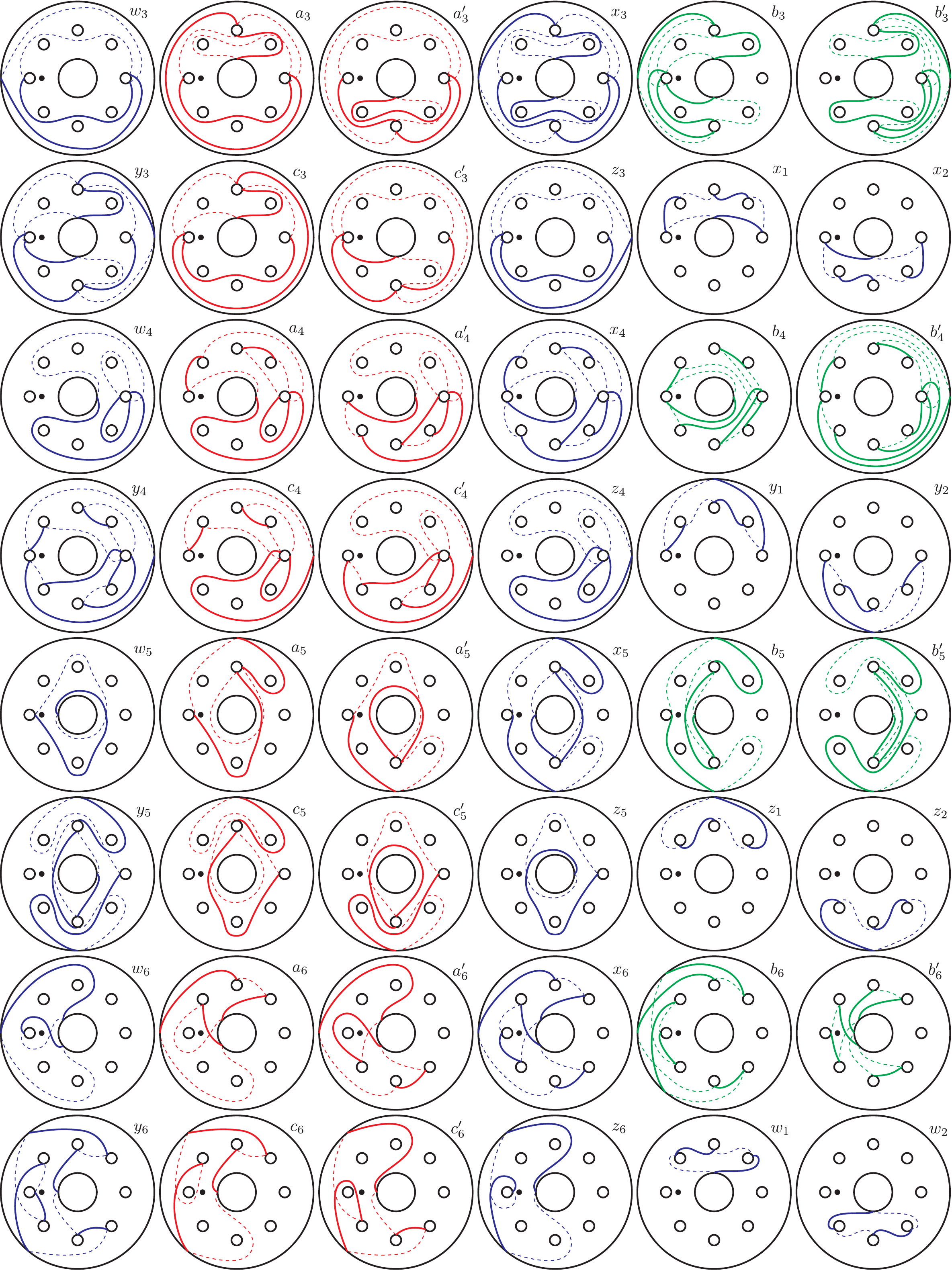}
\caption{Vanishing cycles of the genus--$9$ Lefschetz fibration with signature zero.}
\label{fig:SignatureZeroLF_VanishingCycles}
\end{figure}

\begin{figure}[htbp]
	\centering
	\includegraphics[height=100pt]{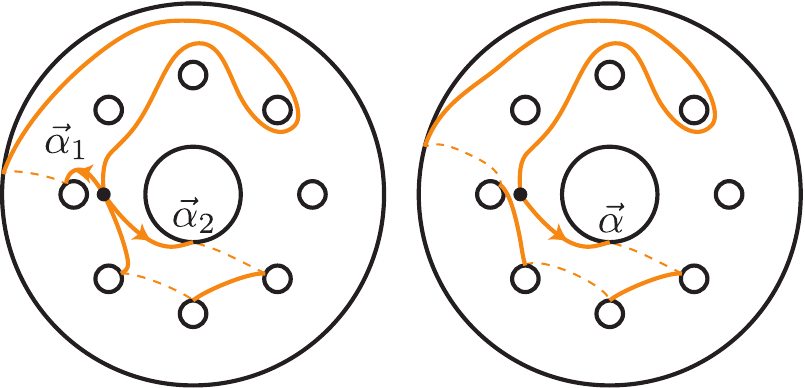}
	\caption{Curves of the point-pushing maps.} 
	\label{fig:Genus9_PushMaps}
\end{figure}

To detect the absence of sections, we consider a marked-point lift of the monodromy factorization of $X$. This lift was constructed in \cite{baykur-hamada:Lefschetz}; we briefly recall the relevant features.

The oriented loops $\vec{\alpha}_1$ and $\vec{\alpha}_2$ for the point-pushing maps are shown in Figure~\ref{fig:Genus9_PushMaps}. Following the construction in \cite{baykur-hamada:Lefschetz}, the marked-point lift contains the subword:
\[
\mathcal{P}_{\vec{\alpha}_1}\cdot \DT{z_1}\DT{z_2}\cdot \mathcal{P}_{\vec{\alpha}_2}.
\]
Here $z_1$ and $z_2$ are disjoint, $\vec{\alpha}_1$ and $z_2$ are disjoint, and $z_1$ and $\vec{\alpha}_2$ are disjoint, so we get
\begin{align*}
\mathcal{P}_{\vec{\alpha}_1} \cdot \DT{z_1} \DT{z_2} \cdot \mathcal{P}_{\vec{\alpha}_2} 
= \mathcal{P}_{\vec{\alpha}_1} \cdot \DT{z_2} \DT{z_1} \cdot \mathcal{P}_{\vec{\alpha}_2} 
= \DT{z_2} \cdot \mathcal{P}_{\vec{\alpha}_1} \mathcal{P}_{\vec{\alpha}_2} \cdot \DT{z_1}.
\end{align*}
The product $\mathcal{P}_{\vec{\alpha}_1}\mathcal{P}_{\vec{\alpha}_2}$ is itself a point-pushing map, namely the point-pushing map along a loop $\vec{\alpha}$ homotopic to the concatenation $\vec{\alpha}_2\cdot\vec{\alpha}_1$.

Replacing the subword $\mathcal{P}_{\vec{\alpha}_1}\DT{z_1}\DT{z_2}\mathcal{P}_{\vec{\alpha}_2}$ by $\DT{z_2}\mathcal{P}_{\vec{\alpha}}\DT{z_1}$, we obtain the factorization 
\begin{align} \label{eq:Signature0LF_Explicit}
 \DT{w_3} \DT{a_3} \DT{a_3^\prime} \DT{x_3} \DT{b_3} \DT{b_3^\prime} \DT{y_3} \DT{c_3} \DT{c_3^\prime} \DT{z_3} \DT{x_1} \DT{x_2} \cdot
\DT{w_4} \DT{a_4} \DT{a_4^\prime} \DT{x_4} \DT{b_4} \DT{b_4^\prime} \DT{y_4} \DT{c_4} \DT{c_4^\prime} \DT{z_4} \DT{y_1} \DT{y_2}  & \\ 
\notag
\cdot 
\DT{w_5} \DT{a_5} \DT{a_5^\prime} \DT{x_5} \DT{b_5} \DT{b_5^\prime} \DT{y_5} \DT{c_5} \DT{c_5^\prime} \DT{z_5} \cdot \DT{z_2} \cdot \mathcal{P}_{\vec{\alpha}} \cdot \DT{z_1} \cdot
\DT{w_6} \DT{a_6} \DT{a_6^\prime} \DT{x_6} \DT{b_6} \DT{b_6^\prime} \DT{y_6} \DT{c_6} \DT{c_6^\prime} \DT{z_6} \DT{w_1} \DT{w_2}
&= 1,
\end{align}
where all the Dehn twist curves are as in Figure~\ref{fig:SignatureZeroLF_VanishingCycles}.
Forgetting the marked point, and therefore the point-pushing map, recovers the monodromy factorization of the genus--$9$ Lefschetz fibration $X$.

Let $w$ denote the image of $\vec\alpha$ in $G=F/N$.
There is an inversion to keep track of here. Writing the marked
relation~\eqref{eq:Signature0LF_Explicit} as
$U\mathcal P_{\vec\alpha}V=1$, its product of positive pointed
Dehn twists is
\[
UV=U\mathcal P_{\vec\alpha}^{-1}U^{-1}
  =\mathcal P_{U_*(\vec\alpha^{-1})}.
\]
Each twist in $U$ acts trivially on $G$, so
\[
\omega(f)=w^{-1}.
\]
This inversion does not affect the normal closure, order, or vanishing
of the abelianized obstruction. We will see in
Section~\ref{sec:calculations} that the abelianization of $w$ vanishes;
thus $X$ admits a pseudosection.

\subsection{Proof of Theorem~\ref{thm:c}}
\label{subsec:proof-thm3}
%======================================================

We follow the same realization strategy as in
\cite{baykur-hamada:Lefschetz}. The new input here is that the preceding constructions
give signature-zero Lefschetz fibrations in every odd genus $g\geq7$.

Fix an odd integer $g\geq7$. Let
$$
W_g=1
$$
be one of the signature-zero genus--$g$ positive factorizations constructed above, and
let $(Y_g,h_g)$ be the corresponding Lefschetz fibration. Choose a
nonseparating vanishing cycle $a$ of $W_g$. After a cyclic permutation,
we may write
$$
W_g=t_aP_g=1.
$$
Choose diffeomorphisms of $\Sigma_g$ carrying $a$ to nonseparating curves
$a_1,a_2,a_3,a_4$ which bound an embedded four-holed sphere with connected complement. Taking the product of
the corresponding conjugates of $W_g$ and applying Hurwitz moves, we obtain a positive
factorization of the form
$$
t_{a_1}t_{a_2}t_{a_3}t_{a_4}\,Q_g=1.
$$
The four boundary curves $a_1,a_2,a_3,a_4$ support a lantern relation. Replacing
$t_{a_1}t_{a_2}t_{a_3}t_{a_4}$ by the product of the three interior twists gives another
positive factorization
$$
t_xt_yt_z\,Q_g=1.
$$
This is a genus--$g$ Lefschetz fibration whose signature is $1$, because the original
factorization was a fiber sum of signature-zero fibrations and a lantern substitution
changes the signature by $+1$ \cite{endo-nagami:signature}.

Let $(L_g,\ell_g)$ be any standard genus--$g$ Lefschetz fibration with negative
signature, for example the standard hyperelliptic fibration on $\CP\#(4g+5)\CPb$, of signature $-4(g+1)$ \cite{baykur-hamada:Lefschetz}, and write
$$
\sigma(L_g)=-s_g
$$
with $s_g>0$. Given any integer $k$, choose $r\geq1$ so that
$$
n=k+rs_g\geq0.
$$
The fiber sum of $n$ copies of the signature--$1$ fibration above and $r$ copies of
$(L_g,\ell_g)$ has signature
$$
n-rs_g=k.
$$
Thus every integer occurs as the signature of a genus--$g$ Lefschetz fibration over
$S^2$.

Finally, for a fixed $k$, we obtain infinitely many such examples by further fiber
summing with arbitrary numbers of copies of the signature-zero fibration $(Y_g,h_g)$.
This does not change the signature, but it increases the number of critical points, and
therefore the Euler characteristic. Hence the resulting total spaces are pairwise
non-diffeomorphic. All factorizations involved are positive, so the resulting fibrations
are symplectic. This proves Theorem~\ref{thm:c}.\QED

\begin{remark}
One may also obtain spin examples; however, this refinement would additionally require compatible quadratic forms for these new factorizations and compatible fiber-sum gluings, as in \cite[Section~4.2]{baykur-hamada:Lefschetz}. To avoid a lengthy addition, we do not pursue this refinement here.
\end{remark}

\section{Calculations}
\label{sec:calculations}
%======================================================

\subsection{Fundamental group of the genus--$9$ example} 
\label{subsec:genus9-pi1}
%======================================================

We now apply the method developed in Section~\ref{sec:prelim} to obtain an explicit presentation of the fundamental group of $X$. We also record the point-pushing word described in Section~\ref{subsec:genus9}. We use the standard generators $a_i,b_i$ shown in Figure~\ref{fig:Genus9_pi1generators} and write $A_i=a_i^{-1}$ and $B_i=b_i^{-1}$. The point-pushing loop $\vec{\alpha}$ corresponds to the word
\[
w=
A_9B_9a_1b_1A_1B_1a_2b_2A_2B_2a_3b_3A_3B_3a_4b_4A_4
B_5a_6b_6A_6B_7a_8a_9a_1a_2A_1.
\]

\begin{prop}\label{prop:pi1X}
Let $X$ be the genus--$9$ Lefschetz fibration described in Section~\ref{subsec:genus9}. Then $\pi_1(X)$ admits the presentation
\[
\pi_1(X)=
\left\langle a_1,b_1,\ldots,a_9,b_9
\ \middle|
R_0,r_{w_3},r_{a_3},r_{a'_3},\ldots,r_{w_2}, w
\right\rangle .
\]
Here $w$ is the relation imposed by the last $2$--handle.
\end{prop}

The generators are shown in Figure~\ref{fig:Genus9_pi1generators}.
The dual chart on the right is used to read off words as explained in
Appendix~\ref{app:word-reading}.

\begin{figure}[htbp]
	\centering
	\includegraphics[height=180pt]{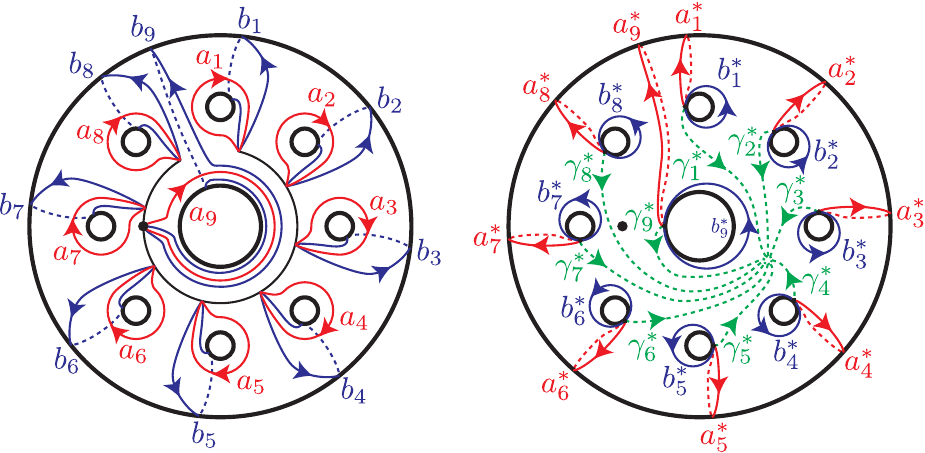}
	\caption{The standard generators of $\pi_1(\Sigma_9)$ (left) and its associated chart (right). Here the arc $\gamma_i^*$ represents $[a_i,b_i]^*$.} 
	\label{fig:Genus9_pi1generators}
\end{figure}

\begingroup
\footnotesize
\setlength{\parindent}{0pt}
\setlength{\parskip}{0.35ex}
\newcommand{\relator}[2]{\noindent $#1 = #2$\par}

\relator{R_0}{a_1\allowbreak b_1\allowbreak A_1\allowbreak B_1\allowbreak a_2\allowbreak b_2\allowbreak A_2\allowbreak B_2\allowbreak a_3\allowbreak b_3\allowbreak A_3\allowbreak B_3\allowbreak a_4\allowbreak b_4\allowbreak A_4\allowbreak B_4\allowbreak a_5\allowbreak b_5\allowbreak A_5\allowbreak B_5\allowbreak a_6\allowbreak b_6\allowbreak A_6\allowbreak B_6\allowbreak a_7\allowbreak b_7\allowbreak A_7\allowbreak B_7\allowbreak a_8\allowbreak b_8\allowbreak A_8\allowbreak B_8\allowbreak a_9\allowbreak b_9\allowbreak A_9\allowbreak B_9}
\relator{r_{w_3}}{a_1\allowbreak a_2\allowbreak a_3\allowbreak b_3\allowbreak a_4\allowbreak a_6\allowbreak a_7\allowbreak B_7\allowbreak a_8\allowbreak a_9\allowbreak a_1\allowbreak b_1\allowbreak A_1\allowbreak B_1\allowbreak a_2\allowbreak a_3\allowbreak b_3\allowbreak a_4\allowbreak a_5\allowbreak a_6\allowbreak a_7\allowbreak a_8\allowbreak a_9}
\relator{r_{a_3}}{a_2\allowbreak B_9\allowbreak A_9\allowbreak b_8\allowbreak a_8\allowbreak B_8\allowbreak a_9\allowbreak a_1\allowbreak b_1\allowbreak A_1\allowbreak B_1\allowbreak a_2\allowbreak a_3\allowbreak b_3\allowbreak a_4\allowbreak a_5\allowbreak a_6\allowbreak a_7\allowbreak a_8\allowbreak a_9\allowbreak a_1\allowbreak B_1\allowbreak a_2\allowbreak a_3\allowbreak b_3\allowbreak a_4\allowbreak a_6\allowbreak a_7\allowbreak B_7\allowbreak a_8\allowbreak a_9\allowbreak a_1\allowbreak b_1\allowbreak A_1}
\relator{r_{a'_3}}{a_1\allowbreak a_2\allowbreak a_3\allowbreak b_3\allowbreak a_4\allowbreak a_6\allowbreak B_9\allowbreak a_1\allowbreak b_1\allowbreak A_1\allowbreak B_1\allowbreak a_2\allowbreak b_2\allowbreak A_2\allowbreak B_2\allowbreak a_3\allowbreak b_3\allowbreak A_3\allowbreak B_3\allowbreak a_4\allowbreak a_5\allowbreak b_5\allowbreak A_5\allowbreak a_6\allowbreak a_7\allowbreak B_7\allowbreak a_8\allowbreak a_9\allowbreak a_1\allowbreak b_1\allowbreak A_1\allowbreak B_1\allowbreak a_2\allowbreak a_3\allowbreak b_3\allowbreak a_4\allowbreak a_5\allowbreak B_5\allowbreak a_6\allowbreak a_7\allowbreak a_8\allowbreak a_9}
\relator{r_{x_3}}{a_2\allowbreak B_9\allowbreak A_9\allowbreak b_8\allowbreak a_8\allowbreak B_8\allowbreak a_9\allowbreak a_1\allowbreak b_1\allowbreak A_1\allowbreak B_1\allowbreak a_2\allowbreak a_3\allowbreak b_3\allowbreak a_4\allowbreak a_5\allowbreak B_5\allowbreak a_6\allowbreak a_7\allowbreak a_8\allowbreak a_9\allowbreak a_1\allowbreak B_1\allowbreak a_2\allowbreak a_3\allowbreak b_3\allowbreak a_4\allowbreak a_6\allowbreak B_9\allowbreak a_1\allowbreak b_1\allowbreak A_1\allowbreak B_1\allowbreak a_2\allowbreak b_2\allowbreak A_2\allowbreak B_2\allowbreak a_3\allowbreak b_3\allowbreak A_3\allowbreak B_3\allowbreak a_4\allowbreak a_5\allowbreak b_5\allowbreak A_5\allowbreak a_6\allowbreak a_7\allowbreak B_7\allowbreak a_8\allowbreak a_9\allowbreak a_1\allowbreak b_1\allowbreak A_1}
\relator{r_{b_3}}{a_2\allowbreak B_9\allowbreak A_9\allowbreak b_8\allowbreak a_8\allowbreak B_8\allowbreak A_8\allowbreak b_7\allowbreak A_7\allowbreak A_6\allowbreak B_5\allowbreak a_6\allowbreak a_7\allowbreak a_8\allowbreak a_9\allowbreak a_1\allowbreak B_1\allowbreak A_1\allowbreak A_9\allowbreak A_8\allowbreak b_7\allowbreak A_7\allowbreak B_9\allowbreak a_1\allowbreak b_1\allowbreak A_1\allowbreak B_1\allowbreak a_2\allowbreak b_2\allowbreak A_2\allowbreak B_2\allowbreak a_3\allowbreak b_3\allowbreak A_3\allowbreak B_3\allowbreak a_4\allowbreak a_5\allowbreak b_5\allowbreak A_5\allowbreak a_6\allowbreak a_7\allowbreak B_7\allowbreak a_8\allowbreak a_9\allowbreak a_1\allowbreak b_1\allowbreak A_1}
\relator{r_{b'_3}}{a_2\allowbreak B_9\allowbreak A_9\allowbreak b_8\allowbreak a_8\allowbreak B_8\allowbreak a_9\allowbreak a_1\allowbreak b_1\allowbreak A_1\allowbreak B_1\allowbreak a_2\allowbreak a_3\allowbreak b_3\allowbreak a_4\allowbreak a_5\allowbreak B_5\allowbreak A_5\allowbreak A_4\allowbreak A_3\allowbreak A_2\allowbreak B_1\allowbreak a_2\allowbreak a_3\allowbreak b_3\allowbreak a_4\allowbreak a_6\allowbreak B_9\allowbreak a_1\allowbreak b_1\allowbreak A_1\allowbreak B_1\allowbreak a_2\allowbreak b_2\allowbreak A_2\allowbreak B_2\allowbreak a_3\allowbreak b_3\allowbreak A_3\allowbreak B_3\allowbreak a_4\allowbreak a_5\allowbreak b_5\allowbreak A_5\allowbreak A_4\allowbreak B_3\allowbreak A_3\allowbreak A_2\allowbreak b_1}
\relator{r_{y_3}}{A_2\allowbreak B_1\allowbreak a_2\allowbreak a_3\allowbreak b_3\allowbreak a_4\allowbreak a_5\allowbreak B_5\allowbreak A_5\allowbreak A_4\allowbreak b_3\allowbreak a_3\allowbreak B_3\allowbreak A_3\allowbreak b_2\allowbreak a_2\allowbreak B_2\allowbreak A_2\allowbreak b_1\allowbreak a_1\allowbreak B_1\allowbreak A_1\allowbreak b_9\allowbreak a_7\allowbreak B_7\allowbreak a_8\allowbreak a_9\allowbreak a_1\allowbreak b_1\allowbreak a_2\allowbreak a_3\allowbreak a_4\allowbreak a_5\allowbreak b_5\allowbreak a_6\allowbreak a_7\allowbreak B_7\allowbreak a_8\allowbreak b_8\allowbreak A_8\allowbreak B_8\allowbreak a_9\allowbreak b_9}
\relator{r_{c_3}}{A_2\allowbreak B_1\allowbreak a_2\allowbreak a_3\allowbreak b_3\allowbreak a_4\allowbreak a_6\allowbreak a_7\allowbreak B_7\allowbreak a_8\allowbreak a_9\allowbreak a_1\allowbreak b_1\allowbreak a_2\allowbreak a_3\allowbreak a_4\allowbreak a_5\allowbreak a_6\allowbreak a_7\allowbreak B_7\allowbreak a_8\allowbreak b_8\allowbreak A_8\allowbreak B_8\allowbreak a_9\allowbreak b_9}
\relator{r_{c'_3}}{a_1\allowbreak a_2\allowbreak a_3\allowbreak a_4\allowbreak a_5\allowbreak b_5\allowbreak a_6\allowbreak a_7\allowbreak B_7\allowbreak a_8\allowbreak a_9\allowbreak a_1\allowbreak b_1\allowbreak A_1\allowbreak B_1\allowbreak a_2\allowbreak a_3\allowbreak b_3\allowbreak a_4\allowbreak a_5\allowbreak B_5\allowbreak A_5\allowbreak A_4\allowbreak b_3\allowbreak a_3\allowbreak B_3\allowbreak A_3\allowbreak b_2\allowbreak a_2\allowbreak B_2\allowbreak A_2\allowbreak b_1\allowbreak a_1\allowbreak B_1\allowbreak A_1\allowbreak b_9\allowbreak a_7\allowbreak B_7\allowbreak a_8\allowbreak a_9}
\relator{r_{z_3}}{a_1\allowbreak a_2\allowbreak a_3\allowbreak a_4\allowbreak a_5\allowbreak a_6\allowbreak a_7\allowbreak B_7\allowbreak a_8\allowbreak a_9\allowbreak a_1\allowbreak b_1\allowbreak A_1\allowbreak B_1\allowbreak a_2\allowbreak a_3\allowbreak b_3\allowbreak a_4\allowbreak a_6\allowbreak a_7\allowbreak B_7\allowbreak a_8\allowbreak a_9}
\relator{r_{x_1}}{a_2\allowbreak B_2\allowbreak a_3\allowbreak b_3\allowbreak A_3\allowbreak B_9\allowbreak A_9\allowbreak b_8\allowbreak a_8\allowbreak B_8\allowbreak A_8\allowbreak b_7\allowbreak a_8\allowbreak B_8\allowbreak a_9\allowbreak a_1\allowbreak b_1\allowbreak A_1}
\relator{r_{x_2}}{B_9\allowbreak a_1\allowbreak b_1\allowbreak A_1\allowbreak B_1\allowbreak a_2\allowbreak b_2\allowbreak A_2\allowbreak B_2\allowbreak a_3\allowbreak b_3\allowbreak A_3\allowbreak a_4\allowbreak B_4\allowbreak a_5\allowbreak b_5\allowbreak A_5\allowbreak a_6\allowbreak B_6\allowbreak a_7\allowbreak b_7\allowbreak A_7}
\relator{r_{w_4}}{B_9\allowbreak A_9\allowbreak b_8\allowbreak a_8\allowbreak B_8\allowbreak a_9\allowbreak a_1\allowbreak a_2\allowbreak a_3\allowbreak b_3\allowbreak a_4\allowbreak a_3\allowbreak B_3\allowbreak A_3\allowbreak b_2\allowbreak a_2\allowbreak B_2\allowbreak a_3\allowbreak b_3\allowbreak a_4\allowbreak a_5\allowbreak a_6}
\relator{r_{a_4}}{B_9\allowbreak A_9\allowbreak b_8\allowbreak a_8\allowbreak B_8\allowbreak A_8\allowbreak b_7\allowbreak a_8\allowbreak B_8\allowbreak a_9\allowbreak a_1\allowbreak a_2\allowbreak a_3\allowbreak b_3\allowbreak a_4\allowbreak a_3\allowbreak B_3\allowbreak A_3\allowbreak b_2\allowbreak a_2\allowbreak B_2\allowbreak A_2\allowbreak b_1\allowbreak a_2\allowbreak B_2\allowbreak a_3\allowbreak b_3\allowbreak a_4\allowbreak a_5\allowbreak a_6}
\relator{r_{a'_4}}{B_9\allowbreak A_9\allowbreak b_8\allowbreak a_8\allowbreak B_8\allowbreak a_9\allowbreak a_1\allowbreak a_2\allowbreak a_3\allowbreak b_3\allowbreak a_4\allowbreak B_4\allowbreak a_5\allowbreak b_5\allowbreak A_5\allowbreak a_3\allowbreak B_3\allowbreak A_3\allowbreak b_2\allowbreak a_2\allowbreak B_2\allowbreak a_3\allowbreak b_3\allowbreak a_4\allowbreak a_5\allowbreak a_6\allowbreak B_6\allowbreak a_7\allowbreak b_7\allowbreak A_7}
\relator{r_{x_4}}{B_9\allowbreak A_9\allowbreak b_8\allowbreak a_8\allowbreak B_8\allowbreak A_8\allowbreak b_7\allowbreak a_8\allowbreak B_8\allowbreak a_9\allowbreak a_1\allowbreak a_2\allowbreak a_3\allowbreak b_3\allowbreak a_4\allowbreak B_4\allowbreak a_5\allowbreak b_5\allowbreak A_5\allowbreak a_3\allowbreak B_3\allowbreak A_3\allowbreak b_2\allowbreak a_2\allowbreak B_2\allowbreak A_2\allowbreak b_1\allowbreak a_2\allowbreak B_2\allowbreak a_3\allowbreak b_3\allowbreak a_4\allowbreak a_5\allowbreak a_6\allowbreak B_6\allowbreak a_7\allowbreak b_7\allowbreak A_7}
\relator{r_{b_4}}{B_9\allowbreak A_9\allowbreak b_8\allowbreak a_8\allowbreak B_8\allowbreak A_8\allowbreak b_7\allowbreak B_8\allowbreak a_9\allowbreak a_1\allowbreak b_1\allowbreak A_1\allowbreak B_1\allowbreak a_2\allowbreak b_2\allowbreak A_2\allowbreak B_2\allowbreak a_3\allowbreak b_3\allowbreak A_3\allowbreak B_4\allowbreak a_5\allowbreak b_5\allowbreak A_5\allowbreak a_3\allowbreak B_3\allowbreak A_3\allowbreak b_2\allowbreak a_2\allowbreak B_2\allowbreak A_2\allowbreak b_1\allowbreak B_2\allowbreak a_3\allowbreak b_3\allowbreak A_3\allowbreak B_6\allowbreak a_7\allowbreak b_7\allowbreak A_7}
\relator{r_{b'_4}}{a_2\allowbreak B_2\allowbreak a_3\allowbreak b_3\allowbreak a_4\allowbreak a_5\allowbreak a_6\allowbreak B_6\allowbreak a_7\allowbreak b_7\allowbreak A_7\allowbreak A_6\allowbreak A_5\allowbreak A_4\allowbreak A_3\allowbreak A_2\allowbreak A_1\allowbreak A_9\allowbreak A_8\allowbreak b_7\allowbreak a_8\allowbreak B_8\allowbreak a_9\allowbreak a_1\allowbreak a_2\allowbreak a_3\allowbreak b_3\allowbreak a_4\allowbreak B_4\allowbreak a_5\allowbreak b_5\allowbreak A_5\allowbreak A_4\allowbreak B_3\allowbreak A_3\allowbreak A_2\allowbreak b_1}
\relator{r_{y_4}}{B_1\allowbreak a_2\allowbreak a_3\allowbreak b_3\allowbreak a_4\allowbreak a_5\allowbreak B_5\allowbreak A_5\allowbreak b_4\allowbreak a_3\allowbreak B_3\allowbreak A_3\allowbreak b_2\allowbreak a_2\allowbreak B_2\allowbreak A_2\allowbreak b_1\allowbreak a_1\allowbreak B_1\allowbreak A_1\allowbreak A_9\allowbreak b_8\allowbreak B_7\allowbreak a_8\allowbreak a_9\allowbreak a_1\allowbreak a_2\allowbreak a_3\allowbreak a_4\allowbreak a_5\allowbreak a_6\allowbreak a_7\allowbreak B_7\allowbreak A_7\allowbreak b_6\allowbreak a_3\allowbreak B_3\allowbreak A_3\allowbreak b_2}
\relator{r_{c_4}}{B_1\allowbreak a_2\allowbreak a_3\allowbreak b_3\allowbreak a_4\allowbreak a_3\allowbreak B_3\allowbreak A_3\allowbreak b_2\allowbreak a_2\allowbreak B_2\allowbreak A_2\allowbreak b_1\allowbreak a_1\allowbreak B_1\allowbreak A_1\allowbreak A_9\allowbreak b_8\allowbreak B_7\allowbreak a_8\allowbreak a_9\allowbreak a_1\allowbreak a_2\allowbreak a_3\allowbreak a_4\allowbreak a_5\allowbreak a_6\allowbreak a_3\allowbreak B_3\allowbreak A_3\allowbreak b_2}
\relator{r_{c'_4}}{a_1\allowbreak a_2\allowbreak a_3\allowbreak a_4\allowbreak a_5\allowbreak a_6\allowbreak a_7\allowbreak B_7\allowbreak A_7\allowbreak b_6\allowbreak a_3\allowbreak B_3\allowbreak A_3\allowbreak b_2\allowbreak a_2\allowbreak B_2\allowbreak a_3\allowbreak b_3\allowbreak a_4\allowbreak a_5\allowbreak B_5\allowbreak A_5\allowbreak b_4\allowbreak a_3\allowbreak B_3\allowbreak A_3\allowbreak b_2\allowbreak a_2\allowbreak B_2\allowbreak A_2\allowbreak b_1\allowbreak a_1\allowbreak B_1\allowbreak A_1\allowbreak A_9\allowbreak b_8\allowbreak a_8\allowbreak B_8\allowbreak a_9}
\relator{r_{z_4}}{a_1\allowbreak a_2\allowbreak a_3\allowbreak a_4\allowbreak a_5\allowbreak a_6\allowbreak a_3\allowbreak B_3\allowbreak A_3\allowbreak b_2\allowbreak a_2\allowbreak B_2\allowbreak a_3\allowbreak b_3\allowbreak a_4\allowbreak a_3\allowbreak B_3\allowbreak A_3\allowbreak b_2\allowbreak a_2\allowbreak B_2\allowbreak A_2\allowbreak b_1\allowbreak a_1\allowbreak B_1\allowbreak A_1\allowbreak A_9\allowbreak b_8\allowbreak a_8\allowbreak B_8\allowbreak a_9}
\relator{r_{y_1}}{B_1\allowbreak A_1\allowbreak A_9\allowbreak b_8\allowbreak B_7\allowbreak a_8\allowbreak a_9\allowbreak a_1\allowbreak a_2\allowbreak a_3\allowbreak B_3\allowbreak A_3\allowbreak b_2}
\relator{r_{y_2}}{B_3\allowbreak a_4\allowbreak a_5\allowbreak a_6\allowbreak a_7\allowbreak B_7\allowbreak A_7\allowbreak b_6\allowbreak B_5\allowbreak A_5\allowbreak b_4}
\relator{r_{w_5}}{a_5\allowbreak B_7\allowbreak a_8\allowbreak b_8\allowbreak A_8\allowbreak B_8\allowbreak a_9\allowbreak a_1\allowbreak b_1\allowbreak a_1\allowbreak B_1\allowbreak A_1\allowbreak b_9\allowbreak a_9}
\relator{r_{a_5}}{a_5\allowbreak B_7\allowbreak a_8\allowbreak a_9\allowbreak a_1\allowbreak a_2\allowbreak a_1\allowbreak B_1\allowbreak A_1\allowbreak b_9\allowbreak a_9\allowbreak B_9\allowbreak a_1\allowbreak b_1\allowbreak A_1}
\relator{r_{a'_5}}{a_5\allowbreak B_5\allowbreak A_5\allowbreak b_4\allowbreak a_4\allowbreak B_4\allowbreak a_5\allowbreak a_6\allowbreak B_7\allowbreak a_8\allowbreak b_8\allowbreak A_8\allowbreak B_8\allowbreak a_9\allowbreak a_1\allowbreak a_2\allowbreak b_2\allowbreak A_2\allowbreak B_2\allowbreak a_3\allowbreak b_3\allowbreak A_3\allowbreak B_3\allowbreak a_4\allowbreak b_4\allowbreak A_4\allowbreak B_4\allowbreak a_5\allowbreak b_5\allowbreak A_5\allowbreak a_9}
\relator{r_{x_5}}{a_5\allowbreak B_5\allowbreak A_5\allowbreak b_4\allowbreak a_4\allowbreak B_4\allowbreak a_5\allowbreak a_6\allowbreak B_7\allowbreak a_8\allowbreak a_9\allowbreak a_1\allowbreak a_2\allowbreak B_1\allowbreak a_2\allowbreak b_2\allowbreak A_2\allowbreak B_2\allowbreak a_3\allowbreak b_3\allowbreak A_3\allowbreak B_3\allowbreak a_4\allowbreak b_4\allowbreak A_4\allowbreak B_4\allowbreak a_5\allowbreak b_5\allowbreak A_5\allowbreak a_9\allowbreak B_9\allowbreak a_1\allowbreak b_1\allowbreak A_1}
\relator{r_{b_5}}{B_1\allowbreak A_1\allowbreak A_9\allowbreak b_8\allowbreak a_8\allowbreak B_8\allowbreak A_8\allowbreak b_7\allowbreak a_7\allowbreak B_7\allowbreak A_7\allowbreak b_6\allowbreak a_6\allowbreak B_6\allowbreak A_6\allowbreak b_5\allowbreak a_9\allowbreak B_9\allowbreak a_1\allowbreak b_1\allowbreak A_1\allowbreak A_9\allowbreak B_5\allowbreak A_5\allowbreak b_4\allowbreak a_4\allowbreak B_4\allowbreak a_5\allowbreak a_6\allowbreak B_7\allowbreak a_8\allowbreak a_9\allowbreak a_1\allowbreak a_2}
\relator{r_{b'_5}}{B_9\allowbreak a_1\allowbreak b_1\allowbreak A_1\allowbreak a_5\allowbreak B_5\allowbreak A_5\allowbreak b_4\allowbreak a_4\allowbreak B_4\allowbreak a_5\allowbreak a_6\allowbreak A_5\allowbreak a_3\allowbreak B_3\allowbreak A_3\allowbreak b_2\allowbreak a_2\allowbreak B_2\allowbreak A_2\allowbreak A_1\allowbreak A_9\allowbreak b_8\allowbreak a_8\allowbreak B_8\allowbreak a_9\allowbreak a_1\allowbreak a_2\allowbreak B_1\allowbreak a_2\allowbreak b_2\allowbreak A_2\allowbreak B_2\allowbreak a_3\allowbreak b_3\allowbreak A_3\allowbreak B_3\allowbreak a_4\allowbreak b_4\allowbreak A_4\allowbreak B_4\allowbreak a_5\allowbreak b_5\allowbreak A_5}
\relator{r_{y_5}}{a_1\allowbreak B_1\allowbreak A_1\allowbreak b_9\allowbreak B_5\allowbreak a_6\allowbreak b_6\allowbreak A_6\allowbreak B_6\allowbreak a_7\allowbreak b_7\allowbreak A_7\allowbreak B_7\allowbreak a_8\allowbreak b_8\allowbreak A_8\allowbreak B_8\allowbreak a_9\allowbreak a_1\allowbreak b_1\allowbreak A_2\allowbreak A_1\allowbreak A_9\allowbreak b_8\allowbreak A_8\allowbreak B_8\allowbreak a_9\allowbreak a_1\allowbreak a_2\allowbreak b_2\allowbreak A_2\allowbreak B_2\allowbreak a_3\allowbreak b_3\allowbreak A_3\allowbreak a_5\allowbreak A_6\allowbreak A_5\allowbreak b_4\allowbreak A_4\allowbreak B_4\allowbreak a_5\allowbreak b_5\allowbreak a_9}
\relator{r_{c_5}}{a_1\allowbreak B_1\allowbreak A_1\allowbreak b_9\allowbreak a_9\allowbreak B_9\allowbreak a_1\allowbreak b_1\allowbreak A_2\allowbreak A_1\allowbreak A_9\allowbreak b_8\allowbreak A_8\allowbreak B_8\allowbreak a_9\allowbreak a_1\allowbreak a_2\allowbreak b_2\allowbreak A_2\allowbreak B_2\allowbreak a_3\allowbreak b_3\allowbreak A_3\allowbreak a_5\allowbreak a_9}
\relator{r_{c'_5}}{B_5\allowbreak a_6\allowbreak b_6\allowbreak A_6\allowbreak B_6\allowbreak a_7\allowbreak b_7\allowbreak A_7\allowbreak B_7\allowbreak a_8\allowbreak b_8\allowbreak A_8\allowbreak B_8\allowbreak a_9\allowbreak a_1\allowbreak a_2\allowbreak b_2\allowbreak A_2\allowbreak B_2\allowbreak a_3\allowbreak b_3\allowbreak A_3\allowbreak a_5\allowbreak A_6\allowbreak A_5\allowbreak b_4\allowbreak A_4\allowbreak B_4\allowbreak a_5\allowbreak b_5\allowbreak a_9}
\relator{r_{z_5}}{a_9\allowbreak B_9\allowbreak a_1\allowbreak a_2\allowbreak b_2\allowbreak A_2\allowbreak B_2\allowbreak a_3\allowbreak b_3\allowbreak A_3\allowbreak a_5\allowbreak a_9}
\relator{r_{z_1}}{A_9\allowbreak a_8\allowbreak B_8\allowbreak a_9\allowbreak a_1\allowbreak a_2\allowbreak a_2\allowbreak B_2\allowbreak A_2\allowbreak b_1\allowbreak A_1}
\relator{r_{z_2}}{a_4\allowbreak B_4\allowbreak a_5\allowbreak a_6\allowbreak a_6\allowbreak B_6\allowbreak A_6\allowbreak b_5\allowbreak A_5}
\relator{r_{w_6}}{B_9\allowbreak A_9\allowbreak b_8\allowbreak a_8\allowbreak B_8\allowbreak A_8\allowbreak b_7\allowbreak a_7\allowbreak B_9\allowbreak a_1\allowbreak b_1\allowbreak A_1\allowbreak B_1\allowbreak a_2\allowbreak b_2\allowbreak A_2\allowbreak B_2\allowbreak a_3\allowbreak b_3\allowbreak A_3\allowbreak B_3\allowbreak a_4\allowbreak b_4\allowbreak A_4\allowbreak B_4\allowbreak a_5\allowbreak a_6\allowbreak a_7\allowbreak a_8\allowbreak a_9\allowbreak a_1}
\relator{r_{a_6}}{B_9\allowbreak A_9\allowbreak b_8\allowbreak a_8\allowbreak B_8\allowbreak A_8\allowbreak b_7\allowbreak a_7\allowbreak B_7\allowbreak a_8\allowbreak b_8\allowbreak A_8\allowbreak B_9\allowbreak a_1\allowbreak b_1\allowbreak A_1\allowbreak B_1\allowbreak a_2\allowbreak b_2\allowbreak A_2\allowbreak B_2\allowbreak a_3\allowbreak b_3\allowbreak A_3\allowbreak B_3\allowbreak a_4\allowbreak b_4\allowbreak A_4\allowbreak B_4\allowbreak a_5\allowbreak a_6\allowbreak a_7\allowbreak a_8\allowbreak a_9\allowbreak a_1\allowbreak B_1\allowbreak a_2\allowbreak b_2\allowbreak A_2}
\relator{r_{a'_6}}{B_9\allowbreak A_9\allowbreak b_8\allowbreak a_8\allowbreak B_8\allowbreak A_8\allowbreak b_7\allowbreak a_7\allowbreak B_7\allowbreak A_7\allowbreak b_6\allowbreak a_7\allowbreak B_9\allowbreak a_1\allowbreak b_1\allowbreak A_1\allowbreak B_1\allowbreak a_2\allowbreak b_2\allowbreak A_2\allowbreak B_2\allowbreak a_3\allowbreak b_3\allowbreak A_3\allowbreak B_3\allowbreak a_4\allowbreak b_4\allowbreak A_4\allowbreak a_5\allowbreak B_5\allowbreak a_6\allowbreak a_7\allowbreak a_8\allowbreak a_9\allowbreak a_1}
\relator{r_{x_6}}{B_9\allowbreak A_9\allowbreak b_8\allowbreak a_8\allowbreak B_8\allowbreak A_8\allowbreak b_7\allowbreak a_7\allowbreak B_7\allowbreak A_7\allowbreak b_6\allowbreak a_7\allowbreak B_7\allowbreak a_8\allowbreak b_8\allowbreak A_8\allowbreak B_9\allowbreak a_1\allowbreak b_1\allowbreak A_1\allowbreak B_1\allowbreak a_2\allowbreak b_2\allowbreak A_2\allowbreak B_2\allowbreak a_3\allowbreak b_3\allowbreak A_3\allowbreak B_3\allowbreak a_4\allowbreak b_4\allowbreak A_4\allowbreak a_5\allowbreak B_5\allowbreak a_6\allowbreak a_7\allowbreak a_8\allowbreak a_9\allowbreak a_1\allowbreak B_1\allowbreak a_2\allowbreak b_2\allowbreak A_2}
\relator{r_{b_6}}{A_9\allowbreak A_8\allowbreak A_7\allowbreak A_6\allowbreak b_5\allowbreak A_5\allowbreak B_4\allowbreak a_5\allowbreak a_6\allowbreak a_7\allowbreak a_8\allowbreak B_8\allowbreak A_8\allowbreak b_7\allowbreak A_7\allowbreak B_6\allowbreak a_7\allowbreak a_8\allowbreak a_9\allowbreak a_1\allowbreak a_2\allowbreak B_2\allowbreak A_2\allowbreak b_1\allowbreak A_1}
\relator{r_{b'_6}}{a_2\allowbreak B_2\allowbreak A_2\allowbreak b_1\allowbreak B_9\allowbreak A_9\allowbreak b_8\allowbreak a_8\allowbreak B_8\allowbreak A_8\allowbreak b_7\allowbreak a_7\allowbreak B_7\allowbreak A_7\allowbreak b_6\allowbreak a_6\allowbreak B_6\allowbreak A_6\allowbreak b_5\allowbreak a_4\allowbreak B_4\allowbreak A_4\allowbreak b_3\allowbreak a_3\allowbreak B_3\allowbreak A_3\allowbreak b_2\allowbreak a_2\allowbreak B_2\allowbreak A_2\allowbreak b_1\allowbreak a_1\allowbreak B_1\allowbreak A_1\allowbreak b_9\allowbreak a_8\allowbreak B_8\allowbreak A_8\allowbreak b_7\allowbreak B_6\allowbreak a_7\allowbreak b_7\allowbreak A_7\allowbreak B_7\allowbreak a_8\allowbreak b_8\allowbreak A_8\allowbreak B_8\allowbreak a_9\allowbreak b_9}
\relator{r_{y_6}}{B_9\allowbreak A_9\allowbreak b_8\allowbreak a_8\allowbreak B_8\allowbreak A_8\allowbreak b_7\allowbreak a_7\allowbreak B_7\allowbreak A_7\allowbreak b_6\allowbreak a_6\allowbreak B_6\allowbreak A_6\allowbreak b_5\allowbreak B_4\allowbreak a_5\allowbreak a_6\allowbreak a_7\allowbreak a_8\allowbreak B_8\allowbreak A_8\allowbreak b_7\allowbreak B_6\allowbreak a_7\allowbreak a_8\allowbreak a_9\allowbreak a_1\allowbreak a_2\allowbreak B_2\allowbreak A_2\allowbreak b_1}
\relator{r_{c_6}}{B_9\allowbreak A_9\allowbreak b_8\allowbreak a_8\allowbreak B_8\allowbreak A_8\allowbreak b_7\allowbreak a_7\allowbreak B_7\allowbreak A_7\allowbreak b_6\allowbreak a_6\allowbreak B_6\allowbreak A_6\allowbreak b_5\allowbreak a_5\allowbreak B_5\allowbreak a_6\allowbreak a_7\allowbreak a_8\allowbreak B_8\allowbreak A_8\allowbreak b_7\allowbreak a_7\allowbreak B_7\allowbreak a_8\allowbreak a_9\allowbreak a_1\allowbreak a_2\allowbreak B_2\allowbreak A_2\allowbreak b_1}
\relator{r_{c'_6}}{B_9\allowbreak A_9\allowbreak b_8\allowbreak a_8\allowbreak B_8\allowbreak A_8\allowbreak b_7\allowbreak a_7\allowbreak B_7\allowbreak A_7\allowbreak b_6\allowbreak a_6\allowbreak B_6\allowbreak A_6\allowbreak b_5\allowbreak B_4\allowbreak a_5\allowbreak a_6\allowbreak a_7\allowbreak B_6\allowbreak a_7\allowbreak a_8\allowbreak a_9\allowbreak a_1}
\relator{r_{z_6}}{B_9\allowbreak A_9\allowbreak b_8\allowbreak a_8\allowbreak B_8\allowbreak A_8\allowbreak b_7\allowbreak a_7\allowbreak B_7\allowbreak A_7\allowbreak b_6\allowbreak a_6\allowbreak B_6\allowbreak A_6\allowbreak b_5\allowbreak a_5\allowbreak B_5\allowbreak a_6\allowbreak a_7\allowbreak a_7\allowbreak B_7\allowbreak a_8\allowbreak a_9\allowbreak a_1}
\relator{r_{w_1}}{B_1\allowbreak a_2\allowbreak b_2\allowbreak B_9\allowbreak A_9\allowbreak b_8\allowbreak a_8\allowbreak a_9}
\relator{r_{w_2}}{B_9\allowbreak a_1\allowbreak b_1\allowbreak A_1\allowbreak B_1\allowbreak a_2\allowbreak b_2\allowbreak A_2\allowbreak B_2\allowbreak a_3\allowbreak b_3\allowbreak A_3\allowbreak B_3\allowbreak a_4\allowbreak b_4\allowbreak B_5\allowbreak a_6\allowbreak b_6}

\endgroup

\begin{proof}
The relator $R_0$ is the surface relation. The remaining relators are obtained from the vanishing cycles in the lifted monodromy factorization~\eqref{eq:Signature0LF_Explicit} using the procedure of Section~\ref{sec:prelim}. They are listed in the same order as the corresponding curves in Figure~\ref{fig:SignatureZeroLF_VanishingCycles}.
\end{proof}

Writing $\operatorname{ab}$ for the exponent vector in
$H_1(\Sigma_9;\Z)$, we have
\[
\operatorname{ab}(w)=\operatorname{ab}(r_{b'_3})-
\operatorname{ab}(r_{x_2}).
\]
Thus $w$ maps to the identity in $H_1(G;\Z)$.
%\marginpar{\footnotesize \blue{Maybe it's better to define $G$ for clarity?}} 
% I think we used $G$ earlier, so maybe no need.

\begin{remark}\label{rmk:nothyperbolic}
With this explicit presentation of $\pi_1(X)$ in hand, we can show that $X$ doesn't admit a strictly negative sectional curvature, and in particular, a hyperbolic structure. To see this, let
\[
h=a_2a_3a_4a_5,\qquad u=b_5,\qquad v=h^{-1}b_1h \, .
\]
In each of the relator pairs $(r_{c_3},r_{z_3})$ and
$(r_{y_3},r_{c'_3})$, the last subword after $a_1b_1A_1B_1$
in the latter relator also occurs in the former.  Solving the second relation
for this word and substituting into the first gives
\[
v=a_6a_7B_7a_8a_9a_1b_1A_1a_2B_9A_9b_8a_8B_8A_8b_7A_7A_6
 =u^{-1}vu.
\]
Hence $u$ and $v$ commute.
Mapping
\[
b_1,b_2\longmapsto(1,0),\qquad
b_4,b_5\longmapsto(0,1),
\]
and all other generators sent to zero, kills every defining relator,
including $w$, and sends $u, v$ to the standard basis of $\mathbb Z^2$.
Thus $\langle u,v\rangle\cong\mathbb Z^2$, and the conclusion follows from
Preissmann's theorem~\cite{preissmann}, as in
\cite[Addendum~13]{baykur:justfakeS2xS2}.
\end{remark}

\begin{comment}
\begin{remark}\label{rmk:second-pseudosection} 
The alternative marked lift depicted in Figure~\ref{fig:Genus9_PushMaps_alternative} gives another point-pushing word. With respect to the generators used above, this word is
$$
w_2=
a_2B_2A_2b_1A_9a_8B_8A_8b_7A_6A_5b_4A_4B_4a_5a_6b_6
A_6B_6a_7b_7A_7B_7a_8b_8A_8B_8a_9b_9 .
$$
Repeating the GAP calculation in Appendix~\ref{app:gap} with $w_2$
in place of $w$ again gives an image of infinite order in the
class--$2$ nilpotent quotient of $G$. Their integral abelianizations
vanish already by the identities
\[
\operatorname{ab}(w)=\operatorname{ab}(r_{b'_3})-
\operatorname{ab}(r_{x_2}),\qquad
\operatorname{ab}(w_2)=\operatorname{ab}(r_{x_1})-
\operatorname{ab}(r_{b'_3}),
\]
where $\operatorname{ab}$ denotes the exponent vector in
$H_1(\Sigma_9;\Z)$. Thus the same obstruction calculation applies to
both words.
\begin{figure}[htbp]
	\centering
	\includegraphics[height=100pt]{Genus9_PushMaps_alternative.eps}
	\caption{Curves of the point-pushing maps for the alternative marked lift.} 
	\label{fig:Genus9_PushMaps_alternative}
\end{figure}
\end{remark}
\end{comment}

\subsection{Fundamental group of the genus--$7$ example}
\label{subsec:genus7-pi1}
%======================================================

We now carry out the same calculation for the genus--$7$ Lefschetz
fibration $h\colon Y\to S^2$ with monodromy
\eqref{eq:Genus7_Signature0LF}. We use the standard generators
$a_1,b_1,\ldots,a_7,b_7$ shown in Figure~\ref{fig:Genus7_pi1generators},
and again write $A_i=a_i^{-1}$ and $B_i=b_i^{-1}$.

\begin{figure}[htbp]
\centering
\includegraphics[height=165pt]{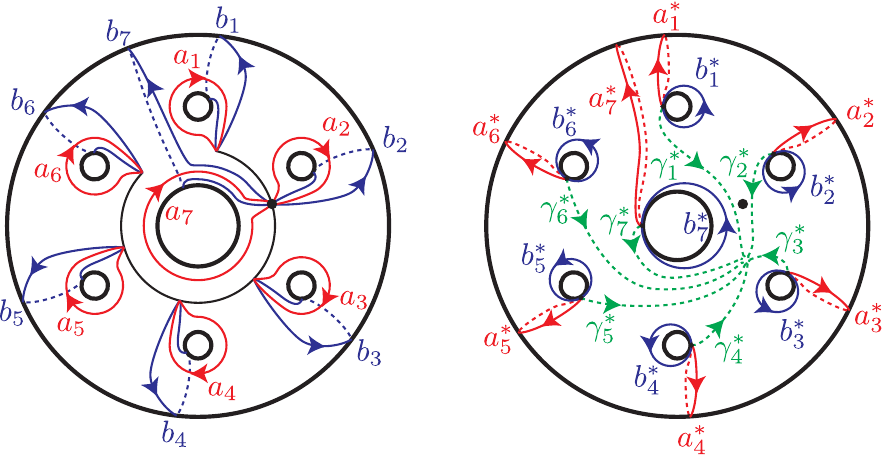}
\caption{The standard generators of $\pi_1(\Sigma_7)$ (left) and its associated chart (right). Here the arc $\gamma_i^*$ represents $[a_i,b_i]^*$.
}
\label{fig:Genus7_pi1generators}
\end{figure}

%%%%%%%%%%%%%%%%%%%%%%% pseudo-section construction
First, we outline how to track the marked point through the breeding construction to give a word  that represents the point-pushing obstruction $\omega(h)$.
We revisit the copies \eqref{eq:R-0}--\eqref{eq:Rrho^2-2} of the genus--$3$ relation each embedded in $\Sigma_7$, this time with a marked point specified in Figure~\ref{fig:Genus7_AuxiliaryCurves}.
Following the construction in~\cite[Section 3.2]{baykur-hamada:Lefschetz}, one can find the lifts of these relations to $\M(\Sigma_{7,1})$ given below:
\begin{align}
    R &= \mathcal{P}_{\vec{\gamma}_2} \cdot t_{X_1}^{-1} t_{D_1}^{-1} t_{Z_1}^{-1} t_{B_1}^{-1} t_{X_2} t_{Z_2} t_{B_3} t_{D_3}, \tag{\ref*{eq:R-0}$^\prime$} \label{eq:R-0'} \\ 
    (R^{\rho})^{t_{u_3}} &= \mathcal{P}_{\vec{\gamma}_1} \cdot t_{B_1} t_{D_1} t_{X_2}^{-1} t_{D_2}^{-1} t_{Z_2}^{-1} t_{B_2}^{-1} t_{X_3} t_{Z_3}, \tag{\ref*{eq:Rrho-1}$^\prime$} \label{eq:Rrho-1'} \\ 
    (R^{\rho^2})^{t_{u_1} t_{u_3}} &= t_{X_1} t_{Z_1} t_{B_2} t_{D_2} t_{X_3}^{-1} t_{D_3}^{-1} t_{Z_3}^{-1} t_{B_3}^{-1}, \tag{\ref*{eq:Rrho^2-1}$^\prime$} \label{eq:Rrho^2-1'} \\
    (R^{\rho})^{t_{v_3} t_{u_3}} &= t_{B_1} t_{D_1} t_{X_2}^{-1} t_{D_2}^{-1} t_{Z_2}^{-1} t_{B_2}^{-1} t_{X_3} t_{Z_3} \cdot \mathcal{P}_{\vec{\gamma}_1}, \tag{\ref*{eq:Rrho-2}$^\prime$} \label{eq:Rrho-2'} \\ 
    (R^{\rho^2})^{t_{v_1} t_{u_1} t_{v_3} t_{u_3}} &= t_{X_1} t_{Z_1} t_{B_2} t_{D_2} t_{X_3}^{-1} t_{D_3}^{-1} t_{Z_3}^{-1} t_{B_3}^{-1}, \tag{\ref*{eq:Rrho^2-2}$^\prime$} \label{eq:Rrho^2-2'} 
\end{align}
where $\vec{\gamma}_1$ and $\vec{\gamma}_2$ are the based loops shown in Figure~\ref{fig:Genus7_PushMaps}, which are parallel to $B_1$ and $X_2$, respectively.
\begin{figure}[htbp]
	\centering
	\includegraphics[height=100pt]{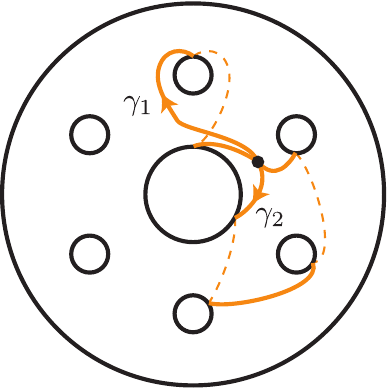}
	\caption{Curves of the point-pushing maps.} 
	\label{fig:Genus7_PushMaps}
\end{figure}
Combining the relations~\eqref{eq:R-0'}, \eqref{eq:Rrho-1'}, \eqref{eq:Rrho^2-1'} gives
\begin{align*}
    R (R^{\rho})^{t_{u_3}} (R^{\rho^2})^{t_{u_1} t_{u_3}} 
    &= \mathcal{P}_{\vec{\gamma}_2} \cdot t_{X_1}^{-1} t_{D_1}^{-1} t_{Z_1}^{-1} t_{B_1}^{-1} t_{X_2} t_{Z_2} t_{B_3} t_{D_3} \cdot \mathcal{P}_{\vec{\gamma}_1} \cdot t_{B_1} t_{D_1} t_{X_2}^{-1} t_{D_2}^{-1} t_{Z_2}^{-1} t_{B_2}^{-1} t_{X_3} t_{Z_3} \\
    & \qquad \cdot t_{X_1} t_{Z_1} t_{B_2} t_{D_2} t_{X_3}^{-1} t_{D_3}^{-1} t_{Z_3}^{-1} t_{B_3}^{-1} \\
    &= \mathcal{P}_{\vec{\gamma}_2} \cdot \mathcal{P}_{\vec{\gamma}_1^*} \cdot t_{X_1}^{-1} t_{D_1}^{-1} t_{Z_1}^{-1} t_{D_1} t_{Z_1} t_{X_1} \cdot  t_{Z_2} t_{D_2}^{-1} t_{Z_2}^{-1} t_{D_2} \cdot t_{B_3} t_{D_3} t_{Z_3} t_{D_3}^{-1} t_{Z_3}^{-1} t_{B_3}^{-1} \\
    &= \mathcal{P}_{\vec{\gamma}_1^* \vec{\gamma}_2} \cdot t_{X_1}^{-1} t_{D_1}^{-1} t_{Z_1}^{-1} t_{D_1} t_{Z_1} t_{X_1} \cdot  t_{Z_2} t_{D_2}^{-1} t_{Z_2}^{-1} t_{D_2} \cdot t_{B_3} t_{D_3} t_{Z_3} t_{D_3}^{-1} t_{Z_3}^{-1} t_{B_3}^{-1}, 
\end{align*}
where $\vec{\gamma}_1^* = t_{X_1}^{-1} t_{D_1}^{-1} t_{Z_1}^{-1} t_{B_1}^{-1}(\vec{\gamma}_1)$ (Here, the twists act from right to left).
We further conjugate the resulting relation by $\varphi = \varphi_1 \varphi_2 \varphi_3 = t_{X_1} t_{a_2}^{-1} t_{b_2}^{-1} t_{c_2} t_{d_2} t_{Z_2}^{-1} t_{B_3}^{-1}$ to obtain
\begin{equation} \tag{\ref*{eq:[D,Z]-variant}$^\prime$} \label{eq:[D,Z]-variant'} 
    (R (R^{\rho})^{t_{u_3}} (R^{\rho^2})^{t_{u_1} t_{u_3}})^{\varphi} 
    = \mathcal{P}_{\vec{\gamma}_3} \cdot [t_{D_1}^{-1}, t_{Z_1}^{-1}] \cdot t_{a_2}^{-1} t_{b_2}^{-1} t_{c_2} t_{d_2} [t_{D_2}^{-1},t_{Z_2}^{-1}] t_{d_2}^{-1} t_{c_2}^{-1} t_{b_2} t_{a_2} \cdot [t_{D_3}, t_{Z_3}]
\end{equation}
with $\vec{\gamma}_3 = \varphi(\vec{\gamma}_1^* \vec{\gamma}_2)$.
In a similar manner, combining the relations~\eqref{eq:R-0'}, \eqref{eq:Rrho-2'}, \eqref{eq:Rrho^2-2'} gives
\begin{align*}
    R (R^{\rho})^{t_{v_3} t_{u_3}} (R^{\rho^2})^{t_{v_1} t_{u_1} t_{v_3} t_{u_3}} 
    &= \mathcal{P}_{\vec{\gamma}_2} \cdot t_{X_1}^{-1} t_{D_1}^{-1} t_{Z_1}^{-1} t_{B_1}^{-1} t_{X_2} t_{Z_2} t_{B_3} t_{D_3} \cdot t_{B_1} t_{D_1} t_{X_2}^{-1} t_{D_2}^{-1} t_{Z_2}^{-1} t_{B_2}^{-1} t_{X_3} t_{Z_3} \cdot \mathcal{P}_{\vec{\gamma}_1} \\
    & \qquad t_{X_1} t_{Z_1} t_{B_2} t_{D_2} t_{X_3}^{-1} t_{D_3}^{-1} t_{Z_3}^{-1} t_{B_3}^{-1} \\
    &= t_{X_1}^{-1} t_{D_1}^{-1} t_{Z_1}^{-1} t_{D_1} t_{Z_1} t_{X_1} \cdot t_{Z_2} t_{D_2}^{-1} t_{Z_2}^{-1} t_{D_2} \cdot t_{B_3} t_{D_3} t_{Z_3} t_{D_3}^{-1} t_{Z_3}^{-1} t_{B_3}^{-1} \cdot \mathcal{P}_{\vec{\gamma}_2^\#} \cdot \mathcal{P}_{\vec{\gamma}_1^\#} \\
    &= t_{X_1}^{-1} t_{D_1}^{-1} t_{Z_1}^{-1} t_{D_1} t_{Z_1} t_{X_1} \cdot  t_{Z_2} t_{D_2}^{-1} t_{Z_2}^{-1} t_{D_2} \cdot t_{B_3} t_{D_3} t_{Z_3} t_{D_3}^{-1} t_{Z_3}^{-1} t_{B_3}^{-1} \cdot \mathcal{P}_{\vec{\gamma}_1^\# \vec{\gamma}_2^\#}, 
\end{align*}
where $\vec{\gamma}_1^\# = t_{Z_1}^{-1} t_{X_1}^{-1} (\vec{\gamma}_1)$ and $\vec{\gamma}_2^\# =  t_{D_2}^{-1} t_{Z_2} t_{D_2} t_{Z_2}^{-1}(\vec{\gamma}_2)$.
Conjugating the resulting relation by $\psi = \psi_1 \psi_2 \psi_3 = \mu_1 t_{X_1} t_{a_2}^{-1} t_{b_2}^{-1} t_{c_2} t_{d_2} \mu_2 t_{Z_2}^{-1} \mu_3^{-1} t_{B_3}^{-1}$, we obtain
\begin{equation} \tag{\ref*{eq:[Z,D]-variant}$^\prime$} \label{eq:[Z,D]-variant'} 
    (R (R^{\rho})^{t_{v_3} t_{u_3}} (R^{\rho^2})^{t_{v_1} t_{u_1} t_{v_3} t_{u_3}})^{\psi} = [t_{Z_1}^{-1}, t_{D_1}^{-1}] \cdot t_{a_2}^{-1} t_{b_2}^{-1} t_{c_2} t_{d_2} [t_{Z_2}^{-1},t_{D_2}^{-1}] t_{d_2}^{-1} t_{c_2}^{-1} t_{b_2} t_{a_2} \cdot [t_{Z_3}, t_{D_3}] \cdot \mathcal{P}_{\vec{\gamma}_4}
\end{equation}
with $\vec{\gamma}_4 = \psi(\vec{\gamma}_1^\# \vec{\gamma}_2^\#)$.
Finally, we combine the relations~\eqref{eq:[D,Z]-variant'} and \eqref{eq:[Z,D]-variant'} to cancel out the right-hand side except for the point-pushing maps:
\begin{equation} \label{eq:Genus7_Signature0LF_Subwords'}
    (R (R^{\rho})^{t_{u_3}} (R^{\rho^2})^{t_{u_1} t_{u_3}})^{\varphi} (R (R^{\rho})^{t_{v_3} t_{u_3}} (R^{\rho^2})^{t_{v_1} t_{u_1} t_{v_3} t_{u_3}})^{\psi} = \mathcal{P}_{\vec{\gamma}_3} \mathcal{P}_{\vec{\gamma}_4}, \tag{\ref*{eq:Genus7_Signature0LF_Subwords}$^\prime$}
\end{equation}
Setting $\vec{\gamma} = \vec{\gamma}_4 \vec{\gamma}_3$, the right-hand side is equal to the single point-pushing map $\mathcal{P}_{\vec{\gamma}}$.
Ultimately, this reconstruction produces a lift of the monodromy factorization~\eqref{eq:Genus7_Signature0LF} to $\M(\Sigma_{7,1})$ in the form
\begin{equation} \tag{\ref*{eq:Genus7_Signature0LF}$^\prime$} \label{eq:Genus7_Signature0LF'} 
    t_{\tilde{x}_1} t_{\tilde{x}_2} \dots t_{\tilde{x}_{48}} = \mathcal{P}_{\vec{\gamma}}
\end{equation}
where $\tilde{x}_i$ is a certain representative of $x_i$ disjoint from the marked point.
We do not present the actual curves $\tilde{x}_i$ here, since the free loops $x_i$ together with the point-pushing loop $\vec\gamma$ are enough for our calculation of the point-pushing obstruction. 

We have thus shown that the loop $\vec{\gamma}$ represents the point-pushing obstruction $\omega(h)$.
To determine $\omega(h)$ explicitly, it remains to express $\vec{\gamma}$ as a word in terms of $a_1, b_1, \dots, a_7, b_7$.
Let $w_7$ denote this word.
From the construction, $\vec\gamma$ naturally factorizes as
\[
\vec\gamma = \vec\gamma_4^L \vec\gamma_4^R \vec\gamma_3^L \vec\gamma_3^R,
\]
where $\vec\gamma_4^L = \psi(\vec\gamma_1^\#)$, $\vec\gamma_4^R = \psi(\vec\gamma_2^\#)$, $\vec\gamma_3^L = \varphi(\vec\gamma_1^*)$, and $\vec\gamma_3^R = \varphi(\vec\gamma_2)$.
To obtain $\vec\gamma_4^R$, consider the based loop $\vec{C}$ in Figure~\ref{fig:Genus7_AuxiliaryCurves}, which can be expressed as
\[
\vec{C}=B_7[a_1,b_1][a_2,b_2]a_3b_3A_3.
\]
Then, it is not hard to verify
\[
\vec\gamma_4^R=
 t_{a_2}^{-1}t_{c_2}^{3}t_{d_2}^{3}t_{e_2}
 t_{c_2}^{-2}t_{d_2}^{-2}t_{b_2}^{-1}t_{e_2}^{-1}(\vec{C}).
\]
Here the subscripts refer to curves in Figure~\ref{fig:Genus7_AuxiliaryCurves}, not to the fundamental-group generators. 
The other three loops are easier to work out and the resulting words are
\begingroup
\small
\begin{align*}
\vec\gamma_4^L &= a_1B_1A_1A_7b_6a_7a_1B_1a_2b_2A_2a_1B_1A_1b_7, \\*
\vec\gamma_4^R &= a_4B_4A_4b_3a_4B_4A_4b_3A_3B_2a_3b_3A_3B_2a_3b_3A_3a_2B_2A_2b_1a_1B_1A_1b_7a_4B_4A_4b_3a_4B_4A_4b_3A_3B_2 \\*
&\quad a_3b_3A_3B_2a_3b_3A_3a_2B_2A_2b_1a_1B_1A_1b_7a_4B_4A_4b_3a_4B_4A_4b_3a_4B_4A_4A_3B_2a_3b_3A_3B_2a_3b_3A_3 \\*
&\quad a_2B_2A_2b_1a_1B_1A_1b_7a_4B_4A_4b_3a_4B_4A_4b_3a_4B_4A_4b_3A_3B_2a_3b_3A_3a_2B_2A_2b_1a_1B_1A_1b_7a_4B_4 \\*
&\quad A_4b_3a_4B_4A_4b_3a_4B_4A_4b_3A_3B_2a_3b_3A_3B_2a_3b_3A_3, \\*
\vec\gamma_3^L &= A_7B_6a_7a_1b_1A_1a_2B_2A_2b_1b_1b_1A_1A_7B_6a_7a_1b_1A_1a_2B_2A_2b_1a_1B_1A_1b_7, \\*
\vec\gamma_3^R &= B_3A_3.
\end{align*}
\endgroup
Thus, the word $w_7$ is the product of these four words.

We can now calculate the full fundamental group.

\begin{prop}\label{prop:pi1Y}
Let $Y$ be the genus--$7$ Lefschetz fibration described above. Then
\[
\pi_1(Y)=
\left\langle a_1,b_1,\ldots,a_7,b_7
\ \middle|\ R_0,r_{x_1},\ldots,r_{x_{48}},w_7\right\rangle.
\]
The relators $R_0,r_{x_1},\ldots,r_{x_{48}}$ are listed below, and
$w_7$ is the relation imposed by the last $2$--handle.
\end{prop}

\begingroup
\footnotesize
\raggedright
\setlength{\parindent}{0pt}
\setlength{\parskip}{0.35ex}
\newcommand{\relator}[2]{\noindent $#1 = #2$\par}

\relator{R_0}{a_1\allowbreak b_1\allowbreak A_1\allowbreak B_1\allowbreak a_2\allowbreak b_2\allowbreak A_2\allowbreak B_2\allowbreak a_3\allowbreak b_3\allowbreak A_3\allowbreak B_3\allowbreak a_4\allowbreak b_4\allowbreak A_4\allowbreak B_4\allowbreak a_5\allowbreak b_5\allowbreak A_5\allowbreak B_5\allowbreak a_6\allowbreak b_6\allowbreak A_6\allowbreak B_6\allowbreak a_7\allowbreak b_7\allowbreak A_7\allowbreak B_7}
\relator{r_{x_{1}}}{B_3\allowbreak A_3\allowbreak B_2\allowbreak A_2\allowbreak b_1\allowbreak a_1\allowbreak B_1\allowbreak A_1\allowbreak A_7\allowbreak b_6\allowbreak a_6\allowbreak a_7\allowbreak a_1\allowbreak B_1\allowbreak a_2\allowbreak b_2\allowbreak A_7\allowbreak B_6\allowbreak A_6\allowbreak b_5\allowbreak A_5\allowbreak a_7}
\relator{r_{x_{2}}}{A_7\allowbreak a_5\allowbreak B_5\allowbreak a_6\allowbreak a_7\allowbreak a_1\allowbreak b_1\allowbreak A_1\allowbreak A_7\allowbreak A_6\allowbreak B_6\allowbreak a_7\allowbreak a_1\allowbreak b_1\allowbreak A_1\allowbreak a_2\allowbreak B_2\allowbreak A_2\allowbreak b_1\allowbreak a_1\allowbreak B_1\allowbreak A_1\allowbreak b_7}
\relator{r_{x_{3}}}{B_3\allowbreak A_3\allowbreak B_2\allowbreak A_2\allowbreak b_1\allowbreak a_1\allowbreak B_1\allowbreak A_1\allowbreak b_7\allowbreak A_7\allowbreak a_5\allowbreak B_5\allowbreak a_6\allowbreak a_7\allowbreak a_1\allowbreak a_2\allowbreak b_2\allowbreak A_7\allowbreak A_6\allowbreak B_6\allowbreak a_7\allowbreak a_1\allowbreak b_1\allowbreak A_1\allowbreak a_2\allowbreak B_2\allowbreak A_2\allowbreak b_1\allowbreak a_1\allowbreak B_1\allowbreak A_1\allowbreak b_7}
\relator{r_{x_{4}}}{A_7\allowbreak B_6\allowbreak a_7\allowbreak a_1\allowbreak b_1\allowbreak A_1\allowbreak B_2\allowbreak A_2\allowbreak b_1\allowbreak a_2\allowbreak a_3\allowbreak b_3}
\relator{r_{x_{5}}}{A_7\allowbreak B_6\allowbreak a_7\allowbreak a_1\allowbreak b_1\allowbreak A_1\allowbreak A_7\allowbreak A_6\allowbreak B_6\allowbreak a_7\allowbreak a_1\allowbreak b_1\allowbreak A_1\allowbreak B_2\allowbreak A_2\allowbreak b_1\allowbreak a_2\allowbreak B_2\allowbreak A_2\allowbreak b_1\allowbreak a_1\allowbreak B_1\allowbreak A_1\allowbreak b_7\allowbreak A_7\allowbreak a_5\allowbreak B_5\allowbreak a_6\allowbreak a_7\allowbreak a_1\allowbreak a_2\allowbreak a_3\allowbreak b_3}
\relator{r_{x_{6}}}{A_7\allowbreak A_6\allowbreak B_6\allowbreak a_7\allowbreak a_1\allowbreak b_1\allowbreak A_1\allowbreak B_2\allowbreak A_2\allowbreak b_1\allowbreak a_1\allowbreak B_1\allowbreak A_1\allowbreak b_7\allowbreak A_7\allowbreak a_5\allowbreak B_5\allowbreak a_6\allowbreak a_7\allowbreak a_1\allowbreak a_2\allowbreak a_3\allowbreak b_3}
\relator{r_{x_{7}}}{B_7\allowbreak a_1\allowbreak b_1\allowbreak A_1\allowbreak B_1\allowbreak a_2\allowbreak b_2\allowbreak A_2\allowbreak a_1\allowbreak B_1\allowbreak A_1\allowbreak b_7\allowbreak A_7\allowbreak a_5\allowbreak B_5\allowbreak a_6\allowbreak a_7\allowbreak a_1\allowbreak B_1\allowbreak A_1\allowbreak A_7\allowbreak b_6\allowbreak A_6\allowbreak a_7}
\relator{r_{x_{8}}}{a_1\allowbreak B_1\allowbreak A_1\allowbreak A_7\allowbreak b_6\allowbreak a_7\allowbreak B_7\allowbreak a_1\allowbreak b_1\allowbreak A_1\allowbreak B_1\allowbreak a_2\allowbreak b_2\allowbreak A_2\allowbreak A_2\allowbreak B_1\allowbreak a_2\allowbreak a_3\allowbreak b_3}
\relator{r_{x_{9}}}{A_7\allowbreak B_6\allowbreak a_7\allowbreak a_1\allowbreak b_1\allowbreak a_2\allowbreak a_3\allowbreak b_3\allowbreak a_3\allowbreak a_4\allowbreak b_4\allowbreak a_5\allowbreak B_5\allowbreak a_6\allowbreak b_6\allowbreak A_6\allowbreak B_6\allowbreak a_7\allowbreak b_7\allowbreak A_7\allowbreak a_5\allowbreak B_5\allowbreak A_5\allowbreak b_4\allowbreak a_4\allowbreak B_4\allowbreak A_4\allowbreak b_3\allowbreak a_3\allowbreak B_7\allowbreak a_1\allowbreak b_1\allowbreak A_1\allowbreak B_1\allowbreak a_2\allowbreak b_2\allowbreak A_2\allowbreak B_2\allowbreak A_2\allowbreak b_1\allowbreak a_2\allowbreak B_2\allowbreak A_2\allowbreak b_1\allowbreak a_1\allowbreak B_1\allowbreak A_1\allowbreak b_7\allowbreak A_7\allowbreak B_6\allowbreak a_7\allowbreak a_1\allowbreak b_1\allowbreak A_1\allowbreak a_3\allowbreak a_4\allowbreak b_4\allowbreak A_4\allowbreak a_5\allowbreak B_5\allowbreak a_6\allowbreak b_6\allowbreak A_6\allowbreak A_5\allowbreak A_4\allowbreak B_7\allowbreak a_1\allowbreak b_1\allowbreak A_1\allowbreak B_1\allowbreak a_2\allowbreak b_2\allowbreak A_2\allowbreak B_2\allowbreak a_3\allowbreak b_3}
\relator{r_{x_{10}}}{A_7\allowbreak B_6\allowbreak a_7\allowbreak a_1\allowbreak b_1\allowbreak a_2\allowbreak a_3\allowbreak b_3\allowbreak a_3\allowbreak B_7\allowbreak a_1\allowbreak b_1\allowbreak A_1\allowbreak B_1\allowbreak a_2\allowbreak b_2\allowbreak A_2\allowbreak B_2\allowbreak A_2\allowbreak b_1\allowbreak a_2\allowbreak B_2\allowbreak A_2\allowbreak b_1\allowbreak a_1\allowbreak B_1\allowbreak A_1\allowbreak b_7\allowbreak A_7\allowbreak B_6\allowbreak a_7\allowbreak a_1\allowbreak b_1\allowbreak A_1\allowbreak a_3\allowbreak a_4\allowbreak b_4\allowbreak A_4\allowbreak B_7\allowbreak a_1\allowbreak b_1\allowbreak A_1\allowbreak B_1\allowbreak a_2\allowbreak b_2\allowbreak A_2\allowbreak B_2\allowbreak a_3\allowbreak b_3}
\relator{r_{x_{11}}}{A_3\allowbreak B_3\allowbreak a_4\allowbreak b_4\allowbreak A_4\allowbreak B_4\allowbreak a_5\allowbreak b_5\allowbreak A_5\allowbreak a_7\allowbreak B_7\allowbreak A_7\allowbreak b_6\allowbreak a_6\allowbreak B_6\allowbreak A_6\allowbreak b_5\allowbreak A_5\allowbreak B_4\allowbreak A_4\allowbreak B_7\allowbreak a_1\allowbreak b_1\allowbreak A_1\allowbreak B_1\allowbreak a_2\allowbreak b_2\allowbreak A_2\allowbreak B_2\allowbreak A_2\allowbreak b_1\allowbreak a_2\allowbreak B_2\allowbreak A_2\allowbreak b_1\allowbreak a_1\allowbreak B_1\allowbreak A_1\allowbreak b_7\allowbreak A_7\allowbreak B_6\allowbreak a_7\allowbreak a_1\allowbreak b_1\allowbreak A_1\allowbreak B_3\allowbreak A_3\allowbreak b_2\allowbreak a_2\allowbreak B_2\allowbreak A_2\allowbreak b_1\allowbreak a_1\allowbreak B_1\allowbreak A_1\allowbreak b_7\allowbreak a_4\allowbreak a_5\allowbreak a_6\allowbreak B_6\allowbreak A_6\allowbreak b_5\allowbreak A_5\allowbreak B_7\allowbreak a_1\allowbreak b_1\allowbreak A_1\allowbreak B_1\allowbreak a_2\allowbreak b_2\allowbreak A_2\allowbreak B_2\allowbreak a_3\allowbreak b_3\allowbreak A_7\allowbreak B_6\allowbreak a_7\allowbreak a_1\allowbreak b_1\allowbreak a_2\allowbreak a_2\allowbreak B_2\allowbreak A_2\allowbreak b_1\allowbreak a_1\allowbreak B_1\allowbreak A_1\allowbreak b_7}
\relator{r_{x_{12}}}{B_7\allowbreak a_1\allowbreak b_1\allowbreak A_1\allowbreak B_1\allowbreak a_2\allowbreak b_2\allowbreak A_2\allowbreak a_3\allowbreak a_4\allowbreak b_4\allowbreak A_4\allowbreak a_5\allowbreak B_5\allowbreak a_6\allowbreak b_6\allowbreak A_6\allowbreak A_5\allowbreak A_4\allowbreak B_7\allowbreak a_1\allowbreak b_1\allowbreak A_1\allowbreak B_1\allowbreak a_2\allowbreak b_2\allowbreak A_2\allowbreak B_2\allowbreak a_3\allowbreak b_3\allowbreak a_3\allowbreak a_4\allowbreak b_4\allowbreak a_5\allowbreak B_5\allowbreak a_6\allowbreak b_6\allowbreak A_6\allowbreak B_6\allowbreak a_7\allowbreak b_7\allowbreak A_7\allowbreak a_5\allowbreak B_5\allowbreak A_5\allowbreak b_4\allowbreak a_4\allowbreak B_4\allowbreak A_4\allowbreak b_3\allowbreak a_3}
\relator{r_{x_{13}}}{A_7\allowbreak B_6\allowbreak a_7\allowbreak a_1\allowbreak b_1\allowbreak a_2\allowbreak a_3\allowbreak a_4\allowbreak b_4\allowbreak A_4\allowbreak a_5\allowbreak B_5\allowbreak a_6\allowbreak b_6\allowbreak A_6\allowbreak A_5\allowbreak A_4\allowbreak B_7\allowbreak a_1\allowbreak b_1\allowbreak A_1\allowbreak B_1\allowbreak a_2\allowbreak b_2\allowbreak A_2\allowbreak B_2\allowbreak A_2\allowbreak b_1\allowbreak a_2\allowbreak B_2\allowbreak A_2\allowbreak b_1\allowbreak a_1\allowbreak B_1\allowbreak A_1\allowbreak b_7\allowbreak A_7\allowbreak B_6\allowbreak a_7\allowbreak a_1\allowbreak b_1\allowbreak A_1\allowbreak a_3\allowbreak a_4\allowbreak b_4\allowbreak a_5\allowbreak B_5\allowbreak a_6\allowbreak b_6\allowbreak A_6\allowbreak B_6\allowbreak a_7\allowbreak b_7\allowbreak A_7\allowbreak a_5\allowbreak B_5\allowbreak A_5\allowbreak b_4\allowbreak a_4\allowbreak B_4\allowbreak A_4\allowbreak b_3\allowbreak a_3}
\relator{r_{x_{14}}}{A_7\allowbreak B_6\allowbreak a_7\allowbreak a_1\allowbreak b_1\allowbreak a_2\allowbreak a_2\allowbreak B_2\allowbreak A_2\allowbreak b_1\allowbreak a_1\allowbreak B_1\allowbreak A_1\allowbreak b_7\allowbreak A_3\allowbreak B_3\allowbreak a_4\allowbreak b_4\allowbreak A_4\allowbreak a_5\allowbreak B_5\allowbreak a_6\allowbreak b_6\allowbreak A_6\allowbreak A_5\allowbreak A_4\allowbreak B_7\allowbreak a_1\allowbreak b_1\allowbreak A_1\allowbreak B_1\allowbreak a_2\allowbreak b_2\allowbreak A_2\allowbreak B_2\allowbreak A_2\allowbreak b_1\allowbreak a_2\allowbreak B_2\allowbreak A_2\allowbreak b_1\allowbreak a_1\allowbreak B_1\allowbreak A_1\allowbreak b_7\allowbreak A_7\allowbreak B_6\allowbreak a_7\allowbreak a_1\allowbreak b_1\allowbreak A_1\allowbreak B_3\allowbreak A_3\allowbreak b_2\allowbreak a_2\allowbreak B_2\allowbreak A_2\allowbreak b_1\allowbreak a_1\allowbreak B_1\allowbreak A_1\allowbreak b_7\allowbreak a_4\allowbreak a_5\allowbreak B_5\allowbreak a_6\allowbreak b_6\allowbreak A_6\allowbreak B_6\allowbreak a_7\allowbreak b_7\allowbreak A_7\allowbreak a_5\allowbreak B_5\allowbreak A_5\allowbreak b_4\allowbreak a_4\allowbreak B_4\allowbreak A_4\allowbreak b_3\allowbreak a_3}
\relator{r_{x_{15}}}{B_3\allowbreak A_3\allowbreak b_2\allowbreak a_2\allowbreak B_2\allowbreak A_2\allowbreak b_1\allowbreak a_1\allowbreak B_1\allowbreak A_1\allowbreak b_7\allowbreak a_4\allowbreak B_4\allowbreak A_4\allowbreak A_3\allowbreak a_2\allowbreak B_2\allowbreak A_2\allowbreak b_1\allowbreak a_1\allowbreak B_1\allowbreak A_1\allowbreak b_7\allowbreak A_7\allowbreak B_6\allowbreak a_7\allowbreak a_1\allowbreak b_1\allowbreak a_2\allowbreak a_2\allowbreak B_2\allowbreak A_2\allowbreak b_1\allowbreak a_1\allowbreak B_1\allowbreak A_1\allowbreak b_7\allowbreak A_3\allowbreak B_3\allowbreak A_3\allowbreak A_2\allowbreak b_1\allowbreak a_2\allowbreak B_2\allowbreak A_2\allowbreak b_1\allowbreak a_1\allowbreak B_1\allowbreak A_1\allowbreak b_7\allowbreak A_7\allowbreak B_6\allowbreak a_7\allowbreak a_1\allowbreak b_1\allowbreak A_1}
\relator{r_{x_{16}}}{B_7\allowbreak a_1\allowbreak b_1\allowbreak A_1\allowbreak B_1\allowbreak a_2\allowbreak b_2\allowbreak A_2\allowbreak a_3\allowbreak a_4\allowbreak b_4\allowbreak A_4\allowbreak a_7\allowbreak B_7\allowbreak A_7\allowbreak b_6\allowbreak a_6\allowbreak B_6\allowbreak A_6\allowbreak b_5\allowbreak a_7\allowbreak B_7\allowbreak A_7\allowbreak b_6\allowbreak a_6\allowbreak B_6\allowbreak A_6\allowbreak b_5\allowbreak A_5\allowbreak A_4\allowbreak B_7\allowbreak a_1\allowbreak b_1\allowbreak A_1\allowbreak B_1\allowbreak a_2\allowbreak b_2\allowbreak A_2\allowbreak B_2\allowbreak a_3\allowbreak b_3\allowbreak a_3\allowbreak a_4\allowbreak a_5\allowbreak a_6\allowbreak B_6\allowbreak A_6\allowbreak b_5\allowbreak A_5\allowbreak a_4\allowbreak B_4\allowbreak A_4\allowbreak b_3\allowbreak a_3}
\relator{r_{x_{17}}}{B_7\allowbreak a_1\allowbreak b_1\allowbreak A_1\allowbreak B_1\allowbreak a_2\allowbreak b_2\allowbreak A_2\allowbreak a_3\allowbreak a_4\allowbreak b_4\allowbreak a_5\allowbreak B_5\allowbreak a_6\allowbreak b_6\allowbreak A_6\allowbreak B_6\allowbreak a_7\allowbreak a_1\allowbreak b_1\allowbreak A_1\allowbreak A_7\allowbreak B_6\allowbreak A_6\allowbreak A_5\allowbreak b_4\allowbreak a_5\allowbreak B_5\allowbreak a_6\allowbreak b_6\allowbreak A_6\allowbreak A_5\allowbreak A_4\allowbreak B_7\allowbreak a_1\allowbreak b_1\allowbreak A_1\allowbreak B_1\allowbreak a_2\allowbreak b_2\allowbreak A_2\allowbreak B_2\allowbreak a_3\allowbreak a_4\allowbreak b_4\allowbreak A_4\allowbreak a_5\allowbreak B_5\allowbreak a_6\allowbreak b_6\allowbreak a_7\allowbreak a_1\allowbreak a_2\allowbreak B_2\allowbreak A_2\allowbreak b_1\allowbreak a_1\allowbreak B_1\allowbreak A_1\allowbreak b_7\allowbreak A_7\allowbreak B_5\allowbreak a_6\allowbreak b_6\allowbreak A_6\allowbreak B_6\allowbreak a_7\allowbreak b_7\allowbreak A_7\allowbreak a_5\allowbreak B_5\allowbreak A_5\allowbreak b_4\allowbreak a_4\allowbreak B_4\allowbreak A_4\allowbreak b_3\allowbreak a_3}
\relator{r_{x_{18}}}{B_7\allowbreak a_1\allowbreak b_1\allowbreak A_1\allowbreak B_1\allowbreak a_2\allowbreak b_2\allowbreak A_2\allowbreak a_3\allowbreak a_4\allowbreak b_4\allowbreak a_5\allowbreak B_5\allowbreak a_6\allowbreak b_6\allowbreak A_6\allowbreak A_5\allowbreak A_4\allowbreak B_7\allowbreak a_1\allowbreak b_1\allowbreak A_1\allowbreak B_1\allowbreak a_2\allowbreak b_2\allowbreak A_2\allowbreak B_2\allowbreak a_3\allowbreak a_4\allowbreak b_4\allowbreak A_4\allowbreak a_5\allowbreak B_5\allowbreak A_5\allowbreak b_4\allowbreak a_4\allowbreak B_4\allowbreak A_4\allowbreak b_3\allowbreak a_3}
\relator{r_{x_{19}}}{A_3\allowbreak B_3\allowbreak a_4\allowbreak b_4\allowbreak A_4\allowbreak a_5\allowbreak B_5\allowbreak a_6\allowbreak b_6\allowbreak a_7\allowbreak a_1\allowbreak a_2\allowbreak B_2\allowbreak A_2\allowbreak b_1\allowbreak a_1\allowbreak B_1\allowbreak A_1\allowbreak b_7\allowbreak A_7\allowbreak B_5\allowbreak a_6\allowbreak b_6\allowbreak A_6\allowbreak B_6\allowbreak a_7\allowbreak b_7\allowbreak A_7\allowbreak a_4\allowbreak B_4\allowbreak A_4\allowbreak A_3\allowbreak b_2\allowbreak a_2\allowbreak B_2\allowbreak A_2\allowbreak b_1\allowbreak a_1\allowbreak B_1\allowbreak A_1\allowbreak b_7\allowbreak a_4\allowbreak a_5\allowbreak B_5\allowbreak a_6\allowbreak b_6\allowbreak A_6\allowbreak B_6\allowbreak a_7\allowbreak a_1\allowbreak b_1\allowbreak A_1\allowbreak A_7\allowbreak B_6\allowbreak A_6\allowbreak A_5\allowbreak A_4\allowbreak A_3\allowbreak a_2\allowbreak B_2\allowbreak A_2\allowbreak b_1\allowbreak a_1\allowbreak B_1\allowbreak A_1\allowbreak b_7}
\relator{r_{x_{20}}}{A_7\allowbreak B_5\allowbreak a_6\allowbreak b_6\allowbreak A_6\allowbreak B_6\allowbreak a_7\allowbreak a_1\allowbreak b_1\allowbreak A_1\allowbreak A_7\allowbreak B_6\allowbreak A_6\allowbreak A_5\allowbreak b_4\allowbreak a_5\allowbreak B_5\allowbreak a_6\allowbreak b_6\allowbreak a_7\allowbreak a_1\allowbreak a_2\allowbreak B_2\allowbreak A_2\allowbreak b_1\allowbreak a_1\allowbreak B_1\allowbreak A_1\allowbreak b_7}
\relator{r_{x_{21}}}{A_7\allowbreak A_5\allowbreak A_4\allowbreak B_7\allowbreak a_1\allowbreak b_1\allowbreak A_1\allowbreak B_1\allowbreak a_2\allowbreak b_2\allowbreak A_2\allowbreak B_2\allowbreak a_3\allowbreak a_4\allowbreak b_4\allowbreak A_4\allowbreak a_7\allowbreak B_7\allowbreak A_7\allowbreak b_6\allowbreak a_6\allowbreak B_6\allowbreak A_6\allowbreak b_5\allowbreak a_7\allowbreak B_7\allowbreak a_1\allowbreak b_1\allowbreak A_1\allowbreak A_7\allowbreak B_6\allowbreak A_6\allowbreak A_5\allowbreak b_4\allowbreak a_4\allowbreak B_4\allowbreak A_4\allowbreak b_3\allowbreak a_3\allowbreak B_7\allowbreak a_1\allowbreak b_1\allowbreak A_1\allowbreak B_1\allowbreak a_2\allowbreak b_2\allowbreak A_2\allowbreak a_3\allowbreak a_4\allowbreak a_5\allowbreak a_6\allowbreak a_7\allowbreak a_1\allowbreak a_2\allowbreak B_2\allowbreak A_2\allowbreak b_1\allowbreak a_1\allowbreak B_1\allowbreak A_1\allowbreak b_7}
\relator{r_{x_{22}}}{A_7\allowbreak a_6\allowbreak B_6\allowbreak A_6\allowbreak b_5\allowbreak A_5\allowbreak A_4\allowbreak B_7\allowbreak a_1\allowbreak b_1\allowbreak A_1\allowbreak B_1\allowbreak a_2\allowbreak b_2\allowbreak A_2\allowbreak B_2\allowbreak a_3\allowbreak a_4\allowbreak b_4\allowbreak A_4\allowbreak a_7\allowbreak B_7\allowbreak A_7\allowbreak b_6\allowbreak a_6\allowbreak B_6\allowbreak A_6\allowbreak b_5\allowbreak a_7\allowbreak B_7\allowbreak a_1\allowbreak b_1\allowbreak A_1\allowbreak A_7\allowbreak B_6\allowbreak A_6\allowbreak b_5\allowbreak A_5\allowbreak a_4\allowbreak B_4\allowbreak A_4\allowbreak b_3\allowbreak a_3\allowbreak B_7\allowbreak a_1\allowbreak b_1\allowbreak A_1\allowbreak B_1\allowbreak a_2\allowbreak b_2\allowbreak A_2\allowbreak a_3\allowbreak a_4\allowbreak a_5\allowbreak a_6\allowbreak a_7\allowbreak a_1\allowbreak a_2\allowbreak B_2\allowbreak A_2\allowbreak b_1\allowbreak a_1\allowbreak B_1\allowbreak A_1\allowbreak b_7}
\relator{r_{x_{23}}}{B_7\allowbreak a_1\allowbreak b_1\allowbreak A_1\allowbreak B_1\allowbreak a_2\allowbreak b_2\allowbreak A_2\allowbreak a_3\allowbreak a_4\allowbreak a_5\allowbreak a_6\allowbreak B_6\allowbreak A_6\allowbreak b_5\allowbreak A_5\allowbreak A_4\allowbreak B_7\allowbreak a_1\allowbreak b_1\allowbreak A_1\allowbreak B_1\allowbreak a_2\allowbreak b_2\allowbreak A_2\allowbreak B_2\allowbreak a_3\allowbreak a_4\allowbreak b_4\allowbreak A_4\allowbreak a_7\allowbreak B_7\allowbreak A_7\allowbreak b_6\allowbreak a_6\allowbreak B_6\allowbreak A_6\allowbreak b_5\allowbreak a_4\allowbreak B_4\allowbreak A_4\allowbreak b_3\allowbreak a_3}
\relator{r_{x_{24}}}{A_7\allowbreak a_6\allowbreak B_6\allowbreak A_6\allowbreak b_5\allowbreak a_7\allowbreak B_7\allowbreak a_1\allowbreak b_1\allowbreak A_1\allowbreak A_7\allowbreak B_6\allowbreak A_6\allowbreak b_5\allowbreak A_5\allowbreak B_4\allowbreak a_5\allowbreak a_6\allowbreak a_7\allowbreak a_1\allowbreak a_2\allowbreak B_2\allowbreak A_2\allowbreak b_1\allowbreak a_1\allowbreak B_1\allowbreak A_1\allowbreak b_7}
\relator{r_{x_{25}}}{A_7\allowbreak B_4\allowbreak a_5\allowbreak a_6\allowbreak a_7\allowbreak a_1\allowbreak a_3\allowbreak B_3\allowbreak A_3\allowbreak A_2\allowbreak b_1\allowbreak A_1\allowbreak A_7\allowbreak B_6\allowbreak A_6\allowbreak b_5\allowbreak B_6\allowbreak a_7\allowbreak a_1\allowbreak a_2\allowbreak a_3\allowbreak a_4\allowbreak b_4\allowbreak A_4}
\relator{r_{x_{26}}}{A_7\allowbreak B_6\allowbreak A_6\allowbreak b_5\allowbreak B_6\allowbreak a_7\allowbreak a_1\allowbreak a_2\allowbreak B_2\allowbreak A_2\allowbreak b_1\allowbreak a_1\allowbreak B_1\allowbreak A_1\allowbreak b_7\allowbreak A_7\allowbreak B_4\allowbreak a_5\allowbreak a_6\allowbreak a_7\allowbreak a_1}
\relator{r_{x_{27}}}{a_4\allowbreak B_4\allowbreak A_4\allowbreak A_3\allowbreak A_2\allowbreak A_1\allowbreak b_7\allowbreak A_7\allowbreak B_4\allowbreak a_5\allowbreak a_6\allowbreak a_7\allowbreak a_1\allowbreak a_1\allowbreak B_1\allowbreak a_2\allowbreak a_3\allowbreak b_3\allowbreak A_3\allowbreak A_1\allowbreak A_7\allowbreak A_6\allowbreak b_5\allowbreak B_6\allowbreak a_7\allowbreak a_1\allowbreak a_2\allowbreak B_2\allowbreak A_2\allowbreak b_1\allowbreak a_1\allowbreak B_1\allowbreak A_1\allowbreak b_7}
\relator{r_{x_{28}}}{A_7\allowbreak B_6\allowbreak a_7\allowbreak a_1\allowbreak a_3\allowbreak B_3\allowbreak A_3\allowbreak A_2\allowbreak b_1\allowbreak a_1\allowbreak B_1\allowbreak a_2\allowbreak a_3\allowbreak a_4\allowbreak b_4\allowbreak A_4}
\relator{r_{x_{29}}}{A_7\allowbreak B_6\allowbreak A_6\allowbreak b_5\allowbreak B_6\allowbreak a_7\allowbreak a_1\allowbreak a_3\allowbreak B_3\allowbreak A_3\allowbreak A_2\allowbreak b_1\allowbreak a_1\allowbreak B_1\allowbreak A_1\allowbreak b_7\allowbreak A_7\allowbreak B_4\allowbreak a_5\allowbreak a_6\allowbreak a_7\allowbreak a_1\allowbreak a_1\allowbreak B_1\allowbreak a_2\allowbreak a_3\allowbreak a_4\allowbreak b_4\allowbreak A_4}
\relator{r_{x_{30}}}{A_1\allowbreak A_7\allowbreak A_6\allowbreak b_5\allowbreak B_6\allowbreak a_7\allowbreak a_1\allowbreak a_3\allowbreak B_3\allowbreak A_3\allowbreak b_2\allowbreak A_2\allowbreak A_1\allowbreak b_7\allowbreak A_7\allowbreak B_4\allowbreak a_5\allowbreak a_6\allowbreak a_7\allowbreak a_1\allowbreak a_1\allowbreak B_1\allowbreak a_2\allowbreak a_3\allowbreak a_4\allowbreak b_4\allowbreak A_4}
\relator{r_{x_{31}}}{A_7\allowbreak B_5\allowbreak a_6\allowbreak a_7\allowbreak a_1\allowbreak A_7\allowbreak A_6\allowbreak A_5\allowbreak b_4\allowbreak a_7\allowbreak B_7\allowbreak a_1\allowbreak a_2\allowbreak B_2\allowbreak A_2\allowbreak b_1\allowbreak a_1\allowbreak B_1\allowbreak A_1\allowbreak b_7}
\relator{r_{x_{32}}}{a_4\allowbreak B_4\allowbreak A_4\allowbreak A_3\allowbreak A_2\allowbreak b_1\allowbreak a_1\allowbreak B_1\allowbreak a_2\allowbreak B_2\allowbreak a_3\allowbreak b_3\allowbreak A_3\allowbreak a_2\allowbreak B_2\allowbreak A_2\allowbreak b_1\allowbreak a_1\allowbreak B_1\allowbreak A_1\allowbreak b_7\allowbreak A_7\allowbreak B_6\allowbreak a_7\allowbreak a_1}
\relator{r_{x_{33}}}{a_1\allowbreak a_1\allowbreak B_1\allowbreak a_2\allowbreak a_3\allowbreak a_4\allowbreak b_4\allowbreak A_4\allowbreak a_3\allowbreak B_3\allowbreak A_3\allowbreak b_2\allowbreak a_3\allowbreak a_4\allowbreak a_5\allowbreak B_6\allowbreak a_7\allowbreak b_7\allowbreak A_7\allowbreak B_4\allowbreak a_5\allowbreak B_6\allowbreak a_7\allowbreak b_7\allowbreak A_7\allowbreak B_7\allowbreak a_1\allowbreak b_1\allowbreak A_1\allowbreak B_1\allowbreak a_2\allowbreak b_2\allowbreak A_2\allowbreak a_3\allowbreak B_3\allowbreak A_3\allowbreak A_2\allowbreak b_1\allowbreak a_1\allowbreak B_1\allowbreak A_1\allowbreak b_7\allowbreak A_7\allowbreak B_6\allowbreak a_7\allowbreak a_1\allowbreak a_3\allowbreak B_3\allowbreak A_3\allowbreak b_2\allowbreak a_3\allowbreak a_4\allowbreak b_4\allowbreak A_4\allowbreak B_3\allowbreak a_4\allowbreak b_4\allowbreak A_4\allowbreak B_4\allowbreak a_5\allowbreak b_5\allowbreak A_4\allowbreak B_3\allowbreak a_4\allowbreak b_4\allowbreak A_4}
\relator{r_{x_{34}}}{a_1\allowbreak a_1\allowbreak B_1\allowbreak a_2\allowbreak a_3\allowbreak a_4\allowbreak b_4\allowbreak A_4\allowbreak a_3\allowbreak B_3\allowbreak A_3\allowbreak b_2\allowbreak a_3\allowbreak B_3\allowbreak A_3\allowbreak A_2\allowbreak b_1\allowbreak a_1\allowbreak B_1\allowbreak A_1\allowbreak b_7\allowbreak A_7\allowbreak B_6\allowbreak a_7\allowbreak a_1\allowbreak a_3\allowbreak B_3\allowbreak A_3\allowbreak b_2\allowbreak a_3\allowbreak a_4\allowbreak b_4\allowbreak A_4\allowbreak B_3\allowbreak a_4\allowbreak b_4\allowbreak A_4}
\relator{r_{x_{35}}}{a_1\allowbreak a_1\allowbreak B_1\allowbreak a_2\allowbreak B_2\allowbreak a_3\allowbreak b_3\allowbreak A_3\allowbreak a_2\allowbreak B_2\allowbreak A_2\allowbreak b_1\allowbreak a_1\allowbreak B_1\allowbreak A_1\allowbreak b_7\allowbreak a_7\allowbreak B_7\allowbreak A_7\allowbreak b_6\allowbreak A_5\allowbreak b_4\allowbreak a_7\allowbreak B_7\allowbreak A_7\allowbreak b_6\allowbreak A_5\allowbreak A_4\allowbreak B_3\allowbreak A_3\allowbreak A_2\allowbreak b_1\allowbreak a_1\allowbreak B_1\allowbreak A_1\allowbreak b_7\allowbreak A_7\allowbreak B_6\allowbreak a_7\allowbreak a_1\allowbreak a_4\allowbreak B_4\allowbreak A_4\allowbreak b_3\allowbreak a_4\allowbreak B_5\allowbreak A_5\allowbreak b_4}
\relator{r_{x_{36}}}{a_3\allowbreak B_3\allowbreak A_3\allowbreak b_2\allowbreak a_3\allowbreak a_4\allowbreak b_4\allowbreak A_4\allowbreak B_3\allowbreak a_4\allowbreak b_4\allowbreak A_4\allowbreak B_4\allowbreak a_5\allowbreak b_5\allowbreak A_4\allowbreak B_3\allowbreak a_4\allowbreak b_4\allowbreak A_4\allowbreak a_3\allowbreak B_3\allowbreak A_3\allowbreak b_2\allowbreak a_3\allowbreak a_4\allowbreak a_5\allowbreak B_6\allowbreak a_7\allowbreak b_7\allowbreak A_7\allowbreak B_4\allowbreak a_5\allowbreak B_6\allowbreak a_7\allowbreak b_7\allowbreak A_7\allowbreak B_7\allowbreak a_1\allowbreak b_1\allowbreak A_1\allowbreak B_1\allowbreak a_2\allowbreak b_2\allowbreak A_2}
\relator{r_{x_{37}}}{a_1\allowbreak a_1\allowbreak B_1\allowbreak a_2\allowbreak a_3\allowbreak a_4\allowbreak b_4\allowbreak A_4\allowbreak B_3\allowbreak a_4\allowbreak b_4\allowbreak A_4\allowbreak B_4\allowbreak a_5\allowbreak b_5\allowbreak A_4\allowbreak B_3\allowbreak A_3\allowbreak A_2\allowbreak b_1\allowbreak a_1\allowbreak B_1\allowbreak A_1\allowbreak b_7\allowbreak A_7\allowbreak B_6\allowbreak a_7\allowbreak a_1\allowbreak a_3\allowbreak B_3\allowbreak A_3\allowbreak b_2\allowbreak a_3\allowbreak a_4\allowbreak a_5\allowbreak B_6\allowbreak a_7\allowbreak b_7\allowbreak A_7\allowbreak B_4\allowbreak a_5\allowbreak B_6\allowbreak a_7\allowbreak b_7\allowbreak A_7}
\relator{r_{x_{38}}}{a_1\allowbreak a_1\allowbreak B_1\allowbreak a_2\allowbreak B_2\allowbreak a_3\allowbreak b_3\allowbreak A_3\allowbreak B_2\allowbreak a_3\allowbreak b_3\allowbreak A_3\allowbreak B_3\allowbreak a_4\allowbreak b_4\allowbreak A_4\allowbreak B_4\allowbreak a_5\allowbreak b_5\allowbreak A_4\allowbreak B_3\allowbreak A_3\allowbreak A_2\allowbreak b_1\allowbreak a_1\allowbreak B_1\allowbreak A_1\allowbreak b_7\allowbreak A_7\allowbreak B_6\allowbreak a_7\allowbreak a_1\allowbreak a_4\allowbreak B_4\allowbreak A_4\allowbreak b_3\allowbreak a_4\allowbreak B_4\allowbreak a_5\allowbreak B_6\allowbreak a_7\allowbreak b_7\allowbreak A_7\allowbreak B_4\allowbreak a_5\allowbreak B_6\allowbreak a_7\allowbreak b_7\allowbreak A_7}
\relator{r_{x_{39}}}{a_4\allowbreak B_4\allowbreak A_4\allowbreak b_3\allowbreak a_4\allowbreak B_4\allowbreak A_4\allowbreak A_3\allowbreak B_2\allowbreak a_3\allowbreak b_3\allowbreak A_3\allowbreak a_2\allowbreak B_2\allowbreak A_2\allowbreak b_1\allowbreak a_1\allowbreak B_1\allowbreak A_1\allowbreak b_7\allowbreak a_1\allowbreak a_1\allowbreak B_1\allowbreak a_2\allowbreak B_2\allowbreak a_3\allowbreak b_3\allowbreak A_3\allowbreak B_2\allowbreak a_3\allowbreak b_3\allowbreak A_3\allowbreak a_4\allowbreak B_4\allowbreak A_4\allowbreak A_3\allowbreak A_2\allowbreak b_1\allowbreak a_1\allowbreak B_1\allowbreak A_1\allowbreak b_7\allowbreak A_7\allowbreak B_6\allowbreak a_7\allowbreak a_1}
\relator{r_{x_{40}}}{a_7\allowbreak B_7\allowbreak A_7\allowbreak b_6\allowbreak A_5\allowbreak b_4\allowbreak a_7\allowbreak B_7\allowbreak A_7\allowbreak b_6\allowbreak A_5\allowbreak b_4\allowbreak A_4\allowbreak B_3\allowbreak a_4\allowbreak b_4\allowbreak A_4\allowbreak a_3\allowbreak B_3\allowbreak A_3\allowbreak b_2\allowbreak a_3\allowbreak a_4\allowbreak B_5\allowbreak A_5\allowbreak b_4\allowbreak a_4\allowbreak B_4\allowbreak A_4\allowbreak b_3\allowbreak a_3\allowbreak B_3\allowbreak A_3\allowbreak b_2\allowbreak a_3\allowbreak B_3\allowbreak A_3\allowbreak b_2\allowbreak a_3\allowbreak a_4\allowbreak b_4\allowbreak A_4\allowbreak B_3\allowbreak a_4\allowbreak b_4\allowbreak A_4}
\relator{r_{x_{41}}}{a_3\allowbreak B_3\allowbreak A_3\allowbreak b_2\allowbreak a_3\allowbreak a_4\allowbreak a_5\allowbreak B_6\allowbreak A_6\allowbreak b_5\allowbreak a_5\allowbreak A_4\allowbreak B_3\allowbreak a_4\allowbreak b_4\allowbreak A_4\allowbreak B_3\allowbreak a_4\allowbreak b_4\allowbreak A_4\allowbreak B_4\allowbreak a_5\allowbreak a_6\allowbreak a_7\allowbreak a_1\allowbreak a_1\allowbreak B_1\allowbreak A_1\allowbreak b_7\allowbreak A_7\allowbreak B_4\allowbreak a_5\allowbreak B_6\allowbreak a_7\allowbreak b_7\allowbreak A_7\allowbreak B_4\allowbreak a_5\allowbreak B_6\allowbreak a_7\allowbreak b_7\allowbreak A_7\allowbreak B_7\allowbreak a_1\allowbreak b_1\allowbreak A_1\allowbreak B_1\allowbreak a_2\allowbreak b_2\allowbreak A_2}
\relator{r_{x_{42}}}{a_3\allowbreak B_3\allowbreak A_3\allowbreak b_2\allowbreak a_3\allowbreak a_4\allowbreak a_5\allowbreak A_4\allowbreak B_3\allowbreak a_4\allowbreak b_4\allowbreak A_4\allowbreak B_3\allowbreak a_4\allowbreak b_4\allowbreak A_4\allowbreak B_4\allowbreak a_5\allowbreak B_6\allowbreak a_7\allowbreak b_7\allowbreak A_7\allowbreak B_7\allowbreak a_1\allowbreak b_1\allowbreak A_1\allowbreak B_1\allowbreak a_2\allowbreak b_2\allowbreak A_2}
\relator{r_{x_{43}}}{a_3\allowbreak B_3\allowbreak A_3\allowbreak b_2\allowbreak a_3\allowbreak a_4\allowbreak B_5\allowbreak a_6\allowbreak b_6\allowbreak A_5\allowbreak b_4\allowbreak A_4\allowbreak B_3\allowbreak a_4\allowbreak b_4\allowbreak A_4\allowbreak B_3\allowbreak a_4\allowbreak b_4\allowbreak A_4\allowbreak a_7\allowbreak B_7\allowbreak A_7\allowbreak b_6\allowbreak A_5\allowbreak b_4\allowbreak a_7\allowbreak B_7\allowbreak a_1\allowbreak b_1\allowbreak A_1\allowbreak A_1\allowbreak A_7\allowbreak A_6\allowbreak A_5\allowbreak b_4\allowbreak a_4\allowbreak B_4\allowbreak A_4\allowbreak b_3\allowbreak a_3\allowbreak B_3\allowbreak A_3\allowbreak b_2}
\relator{r_{x_{44}}}{A_7\allowbreak B_4\allowbreak a_5\allowbreak B_6\allowbreak A_6\allowbreak b_5\allowbreak a_5\allowbreak B_5\allowbreak a_6\allowbreak a_7\allowbreak a_1\allowbreak a_1\allowbreak B_1\allowbreak A_1\allowbreak b_7}
\relator{r_{x_{45}}}{a_3\allowbreak B_3\allowbreak A_3\allowbreak b_2\allowbreak a_3\allowbreak a_4\allowbreak B_5\allowbreak a_6\allowbreak a_7\allowbreak a_1\allowbreak a_1\allowbreak B_1\allowbreak A_1\allowbreak b_7\allowbreak A_7\allowbreak A_4\allowbreak B_3\allowbreak a_4\allowbreak b_4\allowbreak A_4\allowbreak B_3\allowbreak a_4\allowbreak b_4\allowbreak A_4\allowbreak a_7\allowbreak B_7\allowbreak A_7\allowbreak b_6\allowbreak A_5\allowbreak b_4\allowbreak a_7\allowbreak B_7\allowbreak A_7\allowbreak b_6\allowbreak A_6\allowbreak B_6\allowbreak a_7\allowbreak b_7\allowbreak A_7\allowbreak B_7\allowbreak a_1\allowbreak b_1\allowbreak A_1\allowbreak B_1\allowbreak a_2\allowbreak b_2\allowbreak A_2}
\relator{r_{x_{46}}}{a_3\allowbreak B_3\allowbreak A_3\allowbreak b_2\allowbreak a_3\allowbreak a_4\allowbreak B_5\allowbreak a_6\allowbreak a_7\allowbreak a_1\allowbreak a_1\allowbreak B_1\allowbreak A_1\allowbreak b_7\allowbreak A_7\allowbreak A_5\allowbreak b_4\allowbreak A_4\allowbreak B_3\allowbreak a_4\allowbreak b_4\allowbreak A_4\allowbreak B_3\allowbreak a_4\allowbreak b_4\allowbreak A_4\allowbreak a_7\allowbreak B_7\allowbreak A_7\allowbreak b_6\allowbreak A_5\allowbreak b_4\allowbreak a_7\allowbreak B_7\allowbreak A_7\allowbreak b_6\allowbreak A_6\allowbreak A_5\allowbreak b_4\allowbreak a_4\allowbreak B_4\allowbreak A_4\allowbreak b_3\allowbreak a_3\allowbreak B_3\allowbreak A_3\allowbreak b_2}
\relator{r_{x_{47}}}{a_3\allowbreak B_3\allowbreak A_3\allowbreak b_2\allowbreak a_3\allowbreak a_4\allowbreak A_5\allowbreak b_4\allowbreak A_4\allowbreak B_3\allowbreak a_4\allowbreak b_4\allowbreak A_4\allowbreak B_3\allowbreak a_4\allowbreak b_4\allowbreak A_4\allowbreak a_7\allowbreak B_7\allowbreak A_7\allowbreak b_6\allowbreak A_5\allowbreak b_4\allowbreak a_4\allowbreak B_4\allowbreak A_4\allowbreak b_3\allowbreak a_3\allowbreak B_3\allowbreak A_3\allowbreak b_2}
\relator{r_{x_{48}}}{A_7\allowbreak A_5\allowbreak b_4\allowbreak a_7\allowbreak B_7\allowbreak A_7\allowbreak b_6\allowbreak A_6\allowbreak b_5\allowbreak A_5\allowbreak B_5\allowbreak a_6\allowbreak a_7\allowbreak a_1\allowbreak a_1\allowbreak B_1\allowbreak A_1\allowbreak b_7}

\endgroup

\begin{proof}
The relator $R_0$ is the surface relation. The remaining relators are
obtained from the vanishing cycles in
\eqref{eq:Genus7_Signature0LF}, using the generators in
Figure~\ref{fig:Genus7_pi1generators} and the procedure in
Appendix~\ref{app:word-reading}. They are listed in the order of the
curves in Figure~\ref{fig:Genus7_SignatureZeroLF_VanishingCycles}.
Proposition~\ref{prop:pointpush-obstruction} gives the additional
relation $w_7=1$.
\end{proof}

Before imposing this last relation, we have
\[
G_7:=\pi_1(Y\setminus\operatorname{int}\nu(\Sigma_7))
=\left\langle a_1,b_1,\ldots,a_7,b_7
\ \middle|\ R_0,r_{x_1},\ldots,r_{x_{48}}\right\rangle.
\]
Abelianizing this presentation gives
\[
H_1(G_7;\Z)\cong\Z^7\oplus\Z_2. 
\]
Moreover, writing $\operatorname{ab}$ for the exponent vector in
$H_1(\Sigma_7;\Z)$, we have
\begin{align*}
\operatorname{ab}(w_7)={}&-12\operatorname{ab}(r_{x_1})
-37\operatorname{ab}(r_{x_2})+5\operatorname{ab}(r_{x_3})\\
&+19\operatorname{ab}(r_{x_4})+20\operatorname{ab}(r_{x_5})
-13\operatorname{ab}(r_{x_{10}}).
\end{align*}
Thus $w_7$ maps to the identity in $H_1(G_7;\Z)$, although, as we will show shortly, it is nontrivial in $G_7$.

\subsection{Proof of Theorem~\ref{thm:a}}
%======================================================

We prove the theorem with two families of examples, who we present first for genus $9$, and then for  $7$, in the chronological order we discovered them, and also in the order the calculations become more involved.

\subsubsection*{Genus $9$.}
Let $f_1=f\colon X_1=X\to S^2$ be the genus--$9$ Lefschetz fibration described in Section~\ref{subsec:genus9}, and for $n\geq 1$ let
$$
f_n\colon X_n\to S^2
$$
be the untwisted fiber sum of $n$ copies of $f$. These are again symplectic Lefschetz fibrations.

For $f=f_1$, the point-pushing obstruction $\omega(f)\in G=F/N$ is represented by the word $w^{-1}$. If $f$ admitted a section, then Proposition~\ref{prop:pointpush-obstruction} would imply that $\omega(f)=1$ in $G$, and hence that $w=1$ in $G$.

To derive a contradiction, we pass to the maximal class--$2$ nilpotent quotient of $G$. The GAP code used to compute this quotient and evaluate the image of $w$ is recorded in Appendix~\ref{app:gap}. The output shows that the image of $w$ has infinite order in this quotient. Hence no positive power of $w$ is trivial in $G$. In particular, $w\neq 1$, and therefore $f$ admits no section.

The same argument applies to the fiber sums $f_n$. Indeed, the vanishing cycles for $f_n$ are obtained by repeating the vanishing cycles for $f$, so the normal closure $N$ and the group $G=F/N$ are unchanged. The marked lift is correspondingly repeated, and the point-pushing obstruction is represented by
$$
\omega(f_n)=\omega(f)^n=w^{-n},
$$
that is, by the word $w^{-n}$ in $G$. Since no positive power of $w$ is trivial in $G$, Proposition~\ref{prop:pointpush-obstruction} shows that $f_n$ admits no section for every $n\geq 1$. Each copy contributes $48$ critical points, so $\eu(X_n)=48n-32$. These Euler characteristics are strictly increasing, and hence the total spaces are pairwise non-diffeomorphic.

On the other hand, Appendix~\ref{app:gap} also shows that $w$ maps to the identity in the abelianization $H_1(G;\Z)$. Hence $w^n$ maps to the identity in $H_1(G;\Z)$ for every $n\geq 1$. By the exact sequence used in Theorem~\ref{thm:multisection-degrees}, each $f_n$ admits homological sections. Moreover, Theorem~\ref{thm:multisection-degrees} implies that each $f_n$ admits $d$--sections for every $d\geq 2$.

Since $\omega(f_n)$ has infinite order,
Proposition~\ref{prop:spherical-degrees} shows that every
component of every multisection of $f_n$ has positive genus.

\subsubsection*{Genus $7$.}
Let $h_1=h\colon Y_1=Y\to S^2$ be the genus--$7$ fibration of
Section~\ref{subsec:genus7-pi1}, and let $h_n\colon Y_n\to S^2$ be
its untwisted $n$--fold self fiber sum. We use the same fiber marking
in every copy. As above, repeating the vanishing cycles leaves the
group $G_7$ unchanged, and repeating the marked lift gives
\[
\omega(h_n)=\omega(h)^n.
\]
Up to inversion and conjugacy, this is represented by $w_7^n$.
The GAP calculation in Appendix~\ref{app:genus7-gap} gives a map
$\rho_5\colon G_7\to Q_5$ to the maximal integral class--$5$
nilpotent quotient for which
\[
q:=\rho_5(w_7)\neq1,\qquad q^2=1.
\]
Thus $q^n=q\neq1$ for every odd $n$, and
Proposition~\ref{prop:pointpush-obstruction} rules out a section of
$h_n$ for every odd $n\geq1$.

These are symplectic Lefschetz fibrations with
\[
\eu(Y_n)=48n-24,\qquad \sigma(Y_n)=0.
\]
Their Euler characteristics are strictly increasing, so their total
spaces are pairwise non-diffeomorphic. Since $w_7$ maps to the identity
in $H_1(G_7;\Z)$, the same is true of every power $w_7^n$. Consequently,
all the $h_n$ admit pseudosections and $d$--sections for every $d\geq2$.
The class--$5$ calculation does not decide whether the even self fiber
sums admit sections.

This completes the proof of Theorem~\ref{thm:a}.\QED

\begin{corollary}\label{cor:genus9-aspherical}
For every $n\geq1$, the genus--$9$ Lefschetz fibration
$f_n\colon X_n\to S^2$ constructed in the proof of
Theorem~\ref{thm:a} admits compatible symplectically aspherical forms.
\end{corollary}

\begin{proof}
As shown above, $\omega(f_n)$ has infinite order for every $n\geq1$.
To obtain symplectically aspherical forms, define characters
$p_n,q_n\colon \pi_1(X_n)\to\mathbb Z$ by
\[
p_n(b_1)=p_n(b_2)=1,\qquad
q_n(a_1)=-1,\qquad q_n(a_5)=1,
\]
with all values on the unspecified generators zero. These assignments respect
the vanishing-cycle relations in Section~4.1: the total
exponents of $b_1$ and $b_2$ sum to zero, and those of $a_1$
and $a_5$ agree. They also respect the surface relation.
Thus they define characters of $G$, and descend to
$\pi_1(X_n)=G/\langle w^n\rangle$ because $w$ is central and vanishes
in the abelianization of $G$. Regard $p_n,q_n$ as classes in $H^1(X_n;\mathbb Z)$ and set
\mbox{$\eta_n=p_n\cup q_n$.} The standard cup-product pairing on the
oriented regular fiber $\Sigma$ gives \cite[Example~3.7]{hatcher:AT}
\[
\langle\eta_n,[\Sigma]\rangle
=
\sum_{i=1}^{9}
\bigl(p_n(a_i)q_n(b_i)-p_n(b_i)q_n(a_i)\bigr)
=1.
\]
Moreover, for every map $u\colon S^2\to X_n$, naturality of the cup
product \cite[Proposition~3.10]{hatcher:AT} gives
\[
u^*\eta_n=u^*p_n\cup u^*q_n=0,
\]
since $H^1(S^2;\mathbb Z)=0$. Thus $\eta_n$ is aspherical.

Abelianizing the vanishing-cycle words in Section~4.1 also
shows that each has nonzero homology class in the fiber.
Thus all vanishing cycles are nonseparating and all fibers
are irreducible. Corollary~\ref{cor:aspherical-form} now
gives compatible symplectically aspherical forms in the real cohomology classes
$
p_n\cup q_n+\kappa\,{\rm PD}[\Sigma]
$
for all sufficiently large $\kappa$.  For these forms there are no symplectic spheres, since any
such sphere would have positive symplectic area.
\end{proof}

\begin{remark}
The proof of Theorem~1.3 in \cite{stipsicz:indecomposability} contains
a step asserting that, after obtaining a section of
$X_i\setminus \nu F$ over a disk, the untwisted self fiber sums
$$
X_i \#_F X_i=(X_i\setminus \nu F)\cup (X_i\setminus \nu F),
\qquad i=1,2,
$$
admit sections. In the language of the present paper, such a doubling
carries a point-pushing obstruction: the relevant obstruction class
is the square of the point-pushing class for the relative section
over the disk. For the fibration $(X,f)$ above, the doubled fibration $f_2$ has
obstruction represented by $w^{-2}$. Since Appendix~\ref{app:gap}
shows that $w$ has infinite order in the maximal class--$2$ nilpotent
quotient of $G$, the word $w^{-2}$ is nontrivial in $G$. Thus, an
untwisted double of a relative section need not give a section of
the doubled Lefschetz fibration. The conclusion of \cite[Theorem~1.3]{stipsicz:indecomposability}
nevertheless remains valid by Smith's independent
argument~\cite{smith:sections}.
\end{remark}

\subsection{Epilogue: existence of sections and $2$--sections}

Lefschetz fibrations without sections appear to be rare to the authors
of this article. The obstruction $\omega(f)\in G$ gives a concrete way
to prove that a section does not exist, but using it requires detecting
nontriviality in the nonabelian group $G=F/N$, and in practice the
calculation can be costly. Thus it is useful to know where to look for
good candidates.

The discussion in Section~\ref{subsec:finding-sections} suggests one
source. In constructions obtained by breeding lower-genus relations,
boundary components often record sections before the relevant capping
and cancellation operations are performed. Surviving boundary
components give honest sections, while cancellations of boundary
twists that hold only after forgetting the marked point may leave
residual point-pushing terms. When no boundary component survives to
give an obvious section, and the abelianized obstruction vanishes,
the resulting fibration is a natural candidate for the nonabelian
obstruction, as in the examples considered above. The genus--$7$
calculation reinforces this picture, and we expect that the
higher-genus examples in Section~\ref{subsec:higher-genera} also
admit no sections.

In light of our examples and Theorem~\ref{thm:b}, we are left with
the following question.

\begin{question}\label{q:two-section}
Does every Lefschetz fibration over $S^2$ admit a $2$--section?
\end{question}

By Theorem~\ref{thm:multisection-degrees}, this question is equivalent
to asking whether the divisibility $m(f)$ of the fiber class always
divides $2$. In particular, it would follow from the stronger assertion
that the fiber class of every Lefschetz fibration over $S^2$ is
primitive. This primitivity assertion was claimed in
\cite{stipsicz:singularfibers}, but the proof contains a gap; see the
appendix of \cite{baykur:minimality}. We do not know whether the
assertion itself is true. For achiral Lefschetz fibrations, however,
the analogous statement is false: Gompf's examples show that one
cannot expect such low-degree multisections without the positivity
assumption on the singularities \cite{gompf:sections}.

It is worth noting that nonexistence of a section is a highly unstable
phenomenon under fiber sums, as illustrated by the following argument
due to Seraphina Lee. Suppose a genus--$g$ Lefschetz fibration, with
$g\geq2$, has a positive pointed word $W=\mathcal P_\gamma$, where
$\gamma$ is a simple nonseparating based loop. A single fiber sum
with a same-genus fibration $h$ admitting a section and having full
monodromy (Dehn twists along its vanishing cycles generate $\M(\Sigma_g)$)
admits a section. To see this, write
$\mathcal P_\gamma=t_at_b^{-1}$, where $a,b$ bound an annulus
containing the marked point.
By Hurwitz moves and conjugation, $h$ admits a pointed factorization
$t_aR=1$ without changing its Hurwitz class
\cite[Lemma~6]{auroux:stable}. Replacing $t_a$ by $t_b$ gives
\[
W\,t_bR=(t_at_b^{-1})t_bR=t_aR=1
\quad\text{in }\M(\Sigma_{g,1}).
\]
Since $a,b$ are isotopic after forgetting the marked point,
this lifts the fiber-sum monodromy and gives a section.
This applies to our original genus--$9$ example $f=f_1$.
Moreover, choosing a pointed mapping class $\psi$ reversing
$\gamma$ gives
$WW^\psi=\mathcal P_\gamma\mathcal P_{\gamma^{-1}}=1$.
Thus one can also conclude that a suitable twisted double admits
a section, even though the untwisted double does not.
%More generally, decomposing an arbitrary point push into pushes along
%simple nonseparating based loops and absorbing them successively gives
%a section after sufficiently many fiber sums with $h$.

We still believe that Lefschetz fibrations over the $2$--sphere
without sections are very rare, and with this in mind, our
counter-examples raise many questions about how special the topology
of their total spaces is. In light of
Corollary~\ref{cor:genus9-aspherical}, we ask:

\begin{question}
Can a closed aspherical $4$--manifold admit a Lefschetz
fibration over $S^2$?
\end{question}

We do not know whether the original genus--$9$ total space
$X=X_1$ is aspherical, even though it does not admit a strictly negative sectional curvature by~\ref{rmk:nothyperbolic}. Its untwisted self fiber sums $X_n$,
however, are not aspherical for $n\geq2$. Indeed, since $w$
is central and has infinite order in $G$, the quotient
$\pi_1(X_n)=G/\langle w^n\rangle$ contains the cyclic subgroup
$\langle w\rangle/\langle w^n\rangle\cong \Z_n$.
This rules out asphericity, since the fundamental group of a closed
aspherical manifold is torsion-free
\cite[Proposition~2.45]{hatcher:AT}. 

\enlargethispage{0.1in}
Specifically, it would be interesting to know if a hyperbolic $4$--manifold could admit a Lefschetz fibration; cf.~\cite[Problem~4.95(a)]{K3}.

\bigskip

\appendix
\section{A method for computing a representation of a loop in $\pi_1(\Sigma_g)$}\label{app:word-reading}
% ==========================================================

Since computing representations of vanishing cycles in $\pi_1(\Sigma_g)$ in terms of the standard generators can be quite cumbersome and prone to errors, we introduce a highly efficient method for this computation, which may be of interest in its own right.
For simplicity, we describe this method only for $\Sigma_3$, since the extension to arbitrary $\Sigma_g$ is straightforward.
Let $a_1, b_1, a_2, b_2, a_3, b_3$ be standard geometric generators of $\pi_1(\Sigma_3)$, represented by simple loops meeting only at the common base point, and satisfying $[a_1,b_1][a_2,b_2][a_3,b_3]=1$.
Then we get a dodecagon by cutting along the generators as in Figure~\ref{fig:Genus3_chart}.
We define the \textit{dual chart}, or simply \textit{chart}, associated with the generators as the system of oriented loops $a_i^*$, $b_i^*$ and arcs $[a_i,b_i]^*$ that are dual to the loops $a_i$, $b_i$, and $[a_i,b_i]$, respectively, as drawn on the right of Figure~\ref{fig:Genus3_chart}.
\begin{figure}[htbp]
	\centering
	\includegraphics[height=170pt]{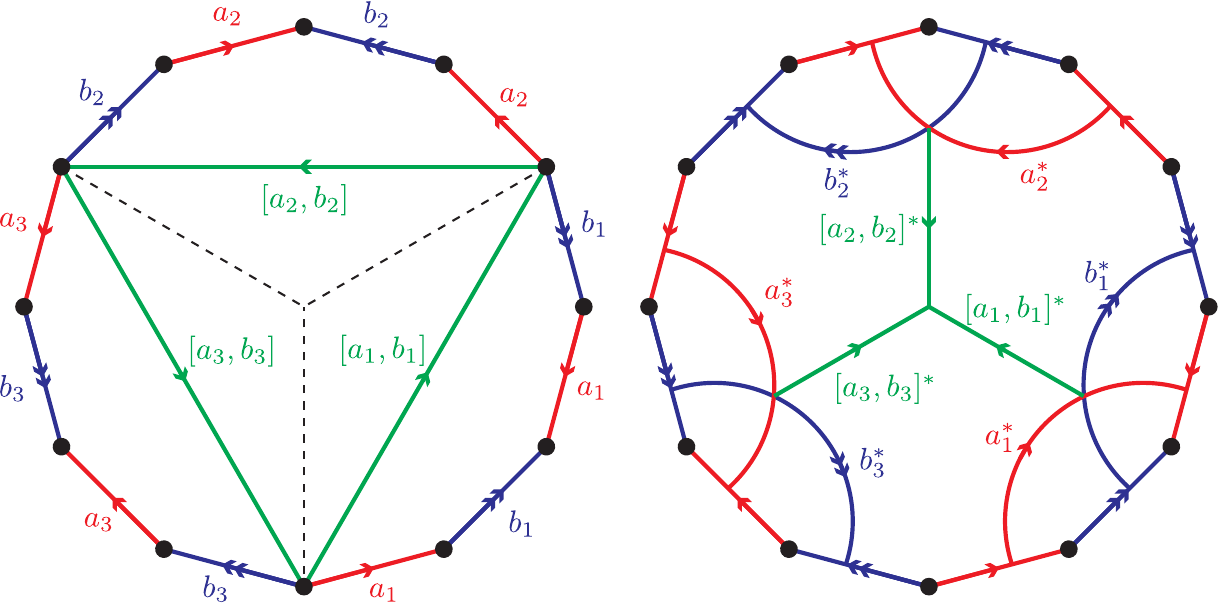}
	\caption{The standard generators of $\pi_1(\Sigma_3)$ (left) and its associated chart (right).} 
	\label{fig:Genus3_chart}
\end{figure}
Put an oriented free loop $\gamma$ in general position, transverse to the chart and disjoint from its vertices and the base point. We wish to express its conjugacy class in $\pi_1(\Sigma_3)$ in terms of $a_i, b_i$.
We demonstrate how to do this with the example in Figure~\ref{fig:Genus3_chart_demo}.
\begin{figure}[htbp]
	\centering
	\includegraphics[height=170pt]{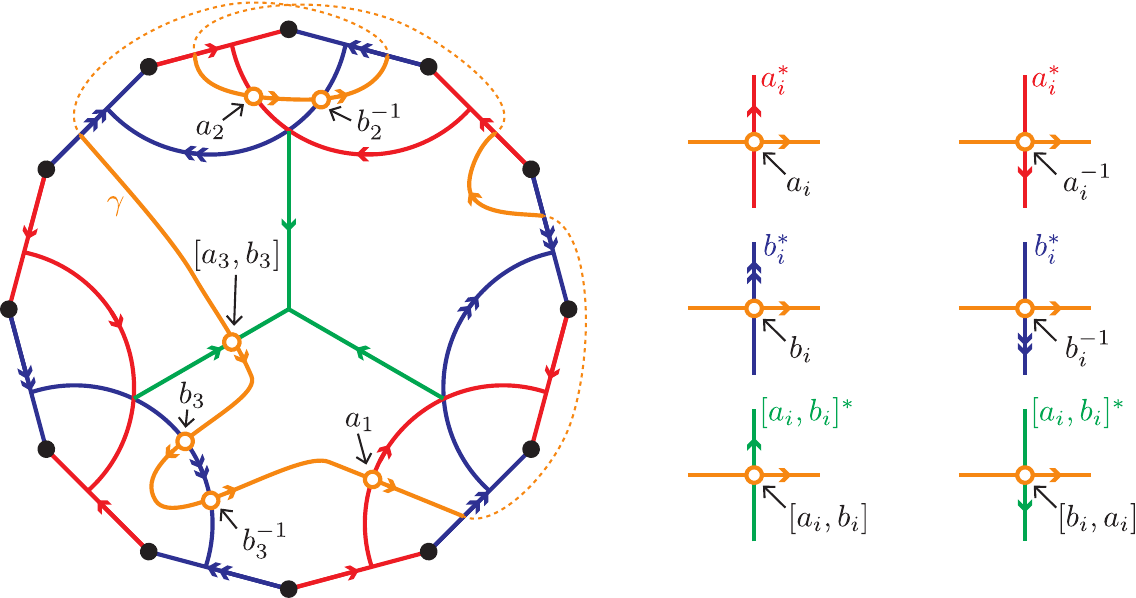}
	\caption{A free loop $\gamma$ on $\Sigma_3$, where the intersections between $\gamma$ and the chart are marked as white vertices. To read off a representation of its conjugacy class in $\pi_1(\Sigma_3)$, traverse $\gamma$ and at each white vertex record the corresponding symbol $a_i^{\pm 1}$, $b_i^{\pm 1}$, or $[a_i,b_i]^{\pm 1}$ according to the intersection patterns shown on the right.
    The resulting representation is $\gamma=a_1 a_2 b_2^{-1} [a_3,b_3] b_3 b_3^{-1}$.} 
	\label{fig:Genus3_chart_demo}
\end{figure}
First, mark all intersection points between $\gamma$ and the chart ($a_i^*$, $b_i^*$, or $[a_i,b_i]^*$) as white vertices.
Next, traverse $\gamma$ according to its orientation, and as we pass over each white vertex record the symbol $a_i^{\pm 1}$, $b_i^{\pm 1}$, or $[a_i,b_i]^{\pm 1}$ according to the intersection patterns on the right.
For this particular example, we obtain $a_1 a_2 b_2^{-1} [a_3,b_3] b_3 b_3^{-1}$ or a cyclic permutation of it depending on where one starts from.
The conjugacy class of this word, or equivalently, its reduction $a_1 a_2 b_2^{-1} [a_3,b_3]$, represents the free homotopy loop $\gamma$.

The disk decomposition proves the reading rule, as illustrated in Figure~\ref{fig:Genus3_chart_demo_proof}.

\begin{figure}[htbp]
	\centering
	\includegraphics[height=130pt]{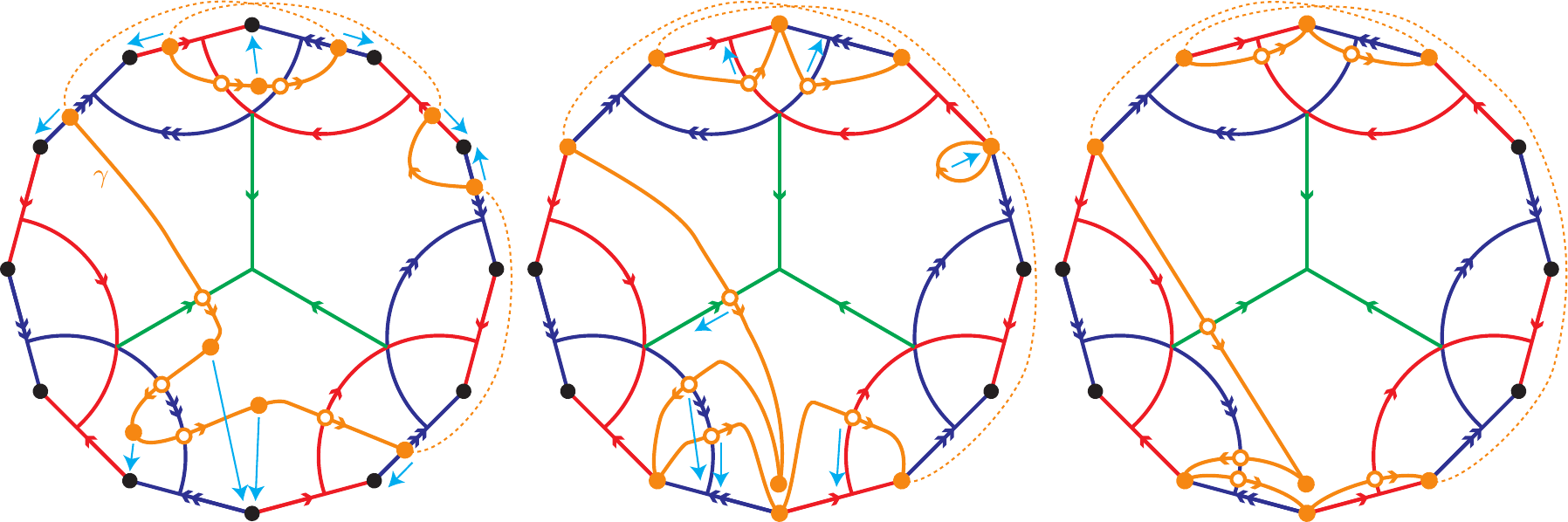}
	\caption{Left to Middle: Mark black vertices on the segments of $\gamma$ and drag the black vertices to the base point.
    Middle to Right: Drag the segments between two adjacent black vertices to $a_i$, $b_i$, $[a_i,b_i]$, or remove it when it is a trivial loop.} 
	\label{fig:Genus3_chart_demo_proof}
\end{figure}
First, we cut the surface along $a_i$, $b_i$, $a_i^*$, $b_i^*$, and $[a_i,b_i]^*$.
The resulting pieces are disks, each with a single distinguished corner mapping to the base point.
Next, we consider the segments of $\gamma$ contained in these disks and mark as black vertices all endpoints of these segments lying on the boundary of the dodecagon.
If endpoints of a segment are both white vertices, then make the middle point of the segment a black vertex too.
In this case, we decompose the segment into two by cutting it at the black vertex.
We then move each black vertex within its disk to this distinguished corner, which maps to the common base point after gluing.
Now we recognize that for each white vertex the adjacent segments make up a based loop in $\Sigma_3$ that is homotopic to either one of $a_i$, $b_i$, $[a_i,b_i]$, or their inverses, which is determined by the intersection patterns described in Figure~\ref{fig:Genus3_chart_demo}.
If there is a segment whose endpoints are both black vertices, then it is a trivial loop within a disk; hence we can remove the segment.
As a result, we obtain the desired concatenated sequence of
$a_i^{\pm1}$, $b_i^{\pm1}$, and $[a_i,b_i]^{\pm1}$, freely homotopic
to $\gamma$. A loop contained in a complementary disk reads the empty
word and is null-homotopic. Changing the starting point cyclically
permutes the word, and choosing a path to the common base point
changes only its conjugacy class. The same disk argument applies in
any genus.

Lastly, we demonstrate how to compute the relators listed in Subsection~\ref{subsec:genus9-pi1} in practice by applying this method to the vanishing cycle $w_5$ as an example.
For the standard generators of $\pi_1(\Sigma_9)$ shown on the left of Figure~\ref{fig:Genus9_pi1generators}, the corresponding dual chart is given on the right.
We put the free loop $w_5$ against the dual chart as in Figure~\ref{fig:Genus9_chart_demo}, together with an orientation.
Reading off the labeled intersections, we obtain the word $a_5 b_7^{-1} [a_8,b_8] a_9 a_1 [b_1,a_1] b_9 a_9$.
This yields the relator $r_{w_5}={a_5\allowbreak B_7\allowbreak a_8\allowbreak b_8\allowbreak A_8\allowbreak B_8\allowbreak a_9\allowbreak a_1\allowbreak b_1\allowbreak a_1\allowbreak B_1\allowbreak A_1\allowbreak b_9\allowbreak a_9}$.
\begin{figure}[htbp]
	\centering
	\includegraphics[height=170pt]{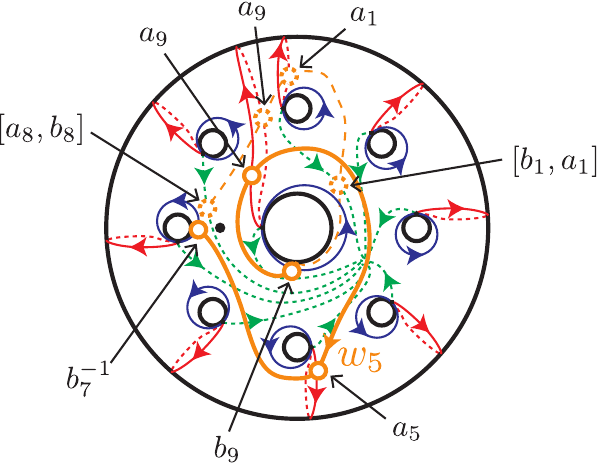}
	\caption{A demonstration of computing the relator corresponding to the vanishing cycle $w_5$ of the genus--$9$ Lefschetz fibration.} 
	\label{fig:Genus9_chart_demo}
\end{figure}

%======================================================

\section{Finding sections of Xiao's $(6,7)$ fibration} \label{app:xiao-section}

Recall that a relatively minimal genus--$2$ Lefschetz fibration is said to be of
\emph{type $(n,s)$} if it has $n$ nonseparating and $s$ separating vanishing cycles.
Here we illustrate the method of Section~\ref{subsec:finding-sections} with
Xiao's holomorphic $(6,7)$ fibration $f\colon X\to S^2$
\cite{xiao:genus2,polizzi:pgq1}, using Szab\'o's explicit monodromy
\cite{szabo:xiao-monodromy}. As observed by Szab\'o (private communication; see also \cite[Section~9.5]{szabo:xiao-monodromy}
and Proposition~\ref{prop:xiao-four-sections} below), it does not have a holomorphic section, and we do not have a section coming from the braid factorization in the hyperelliptic quotient, making it a good example to consider here.
For comparison, Matsumoto's $(6,2)$ fibration has two sections from
Korkmaz's braid lift \cite{korkmaz:noncomplex,korkmaz:invariant} (and four
disjoint sections as shown in \cite{hamada:MCKsections}), while the $(4,3)$
fibration has three sections coming from Weierstrass points \cite{baykur-genus3}.\footnote{The
Lefschetz fibrations of types $(6,2)$, $(4,3)$ and $(6,7)$ are realized by Xiao's holomorphic
fibrations of degrees $2$, $3$ and $4$, respectively 
\cite{xiao:genus2, szabo:xiao-monodromy}.}

\subsubsection*{The based loops.}
Let $a_1,b_1,a_2,b_2$ be the standard geometric generators of
$F=\pi_1(\Sigma_2,p)$. As in Section~\ref{sec:calculations}, write
$A_i=a_i^{-1}$ and $B_i=b_i^{-1}$, and use bars for inverses of other words.
For the monodromy and its hyperelliptic quotient, we follow Szab\'o's
account \cite{szabo:xiao-monodromy}, using his notation $M_i,X_i,Y_i,P_i$
for the vanishing cycles and $r_i,r_i'$ for the six branch points.
The $X_i,Y_i$ twists exchange $r_i,r_i'$; the separating twists fix all
six points. The orbits $\{r_i,r_i'\}$ have size two, so there is no section
consisting of Weierstrass points, that is, fixed points of the
hyperelliptic involution on each smooth fiber. Put
\[
 R_0=a_1b_1A_1B_1a_2b_2A_2B_2,\qquad F=\left\langle a_1,b_1,a_2,b_2\ \middle|\ R_0\right\rangle.
\]
Capping the four boundary components in Figure~\ref{fig:xiao-marking} gives
these generators for $\pi_1(\Sigma_2)$. Following
\cite[Figure~18]{szabo:xiao-monodromy}, choose positive branch meridians
$x_1,\ldots,x_6$ with disjoint access paths from an exterior base point,
ordered $r_1,r'_2,r'_1,r_2,r_3,r'_3$. Let $p$ lie above this base point,
and let $y_j$ be the lift of $x_jx_6$ after filling the ramification points.
Then
\[
 a_1=y_1\bar y_2,\qquad b_1=y_3\bar y_1,\qquad
 a_2=y_3\bar y_2y_1,\qquad b_2=\bar y_4.
\]
Let $t_{c_i}$, $i=1, \ldots, 13$, denote the Dehn twists along the monodromy curves $M_3,\allowbreak X_3,\allowbreak M_0,\allowbreak Y_2,\allowbreak P_2,\allowbreak M_1,\allowbreak X_1,\allowbreak Y_3,\allowbreak P_3,\allowbreak M_2,\allowbreak X_2,\allowbreak Y_1,\allowbreak P_1$, in this order. 
Choose representatives of $c_i$ disjoint from $p$. Their positive Dehn
twists give lifts $\varphi_i=\widetilde t_{c_i}\in\M(\Sigma_{2,1})$ fixing
$p$, as in Section~\ref{subsec:pointpush}. We also write $\varphi_i$ for
their actions on $F$, and put $W=\varphi_{13}\cdots\varphi_1$.
Using the method of Appendix~\ref{app:word-reading}, we read off each
oriented loop as a word in $F$, keeping track of its path from $p$.
The words $v_i=r_{c_i}$, in the same order, are

\medskip
\begingroup
\footnotesize
\setlength{\parindent}{0pt}
\setlength{\parskip}{0.35ex}

\noindent $r_{M_3} = A_1\allowbreak A_2\allowbreak b_1\allowbreak a_1\allowbreak B_1\allowbreak a_2$\par
\noindent $r_{X_3} = A_1\allowbreak b_2\allowbreak b_1\allowbreak A_1\allowbreak B_1\allowbreak a_2$\par
\noindent $r_{M_0} = b_2\allowbreak a_2\allowbreak B_2\allowbreak b_1\allowbreak A_1\allowbreak B_1\allowbreak A_2\allowbreak b_1\allowbreak a_1\allowbreak B_1$\par
\noindent $r_{Y_2} = A_1\allowbreak B_1\allowbreak b_2\allowbreak A_2\allowbreak b_1\allowbreak a_1\allowbreak B_1$\par
\noindent $r_{P_2} = b_1\allowbreak a_1\allowbreak B_1\allowbreak A_1$\par
\noindent $r_{M_1} = a_1\allowbreak a_2\allowbreak B_2\allowbreak A_2\allowbreak A_1\allowbreak b_2$\par
\noindent $r_{X_1} = B_1\allowbreak A_1\allowbreak b_2\allowbreak A_1\allowbreak b_2$\par
\noindent $r_{Y_3} = b_2\allowbreak A_1\allowbreak A_1\allowbreak B_1\allowbreak a_2\allowbreak b_2\allowbreak b_2\allowbreak A_1\allowbreak B_1\allowbreak B_2\allowbreak a_1\allowbreak b_1\allowbreak a_1\allowbreak B_2\allowbreak B_2\allowbreak A_2\allowbreak b_1\allowbreak a_1\allowbreak a_1\allowbreak B_2\allowbreak a_1\allowbreak B_2\allowbreak a_1\allowbreak b_1\allowbreak a_1\allowbreak B_2\allowbreak B_2\allowbreak A_2\allowbreak b_1\allowbreak a_1\allowbreak a_1\allowbreak B_2$\par
\noindent $r_{P_3} = A_1\allowbreak A_1\allowbreak B_1\allowbreak B_2\allowbreak a_1\allowbreak b_1\allowbreak a_1\allowbreak b_2$\par
\noindent $r_{M_2} = A_1\allowbreak A_1\allowbreak B_1\allowbreak A_1\allowbreak b_2\allowbreak b_1\allowbreak a_1\allowbreak a_1\allowbreak a_1\allowbreak B_2$\par
\noindent $r_{X_2} = a_1\allowbreak B_2\allowbreak B_2\allowbreak B_2\allowbreak A_2\allowbreak b_1\allowbreak a_1\allowbreak a_1\allowbreak a_1$\par
\noindent $r_{Y_1} = a_1\allowbreak B_2\allowbreak A_2\allowbreak B_2\allowbreak a_1\allowbreak B_2\allowbreak B_2\allowbreak A_2\allowbreak A_1\allowbreak b_2\allowbreak a_2\allowbreak b_2\allowbreak A_1\allowbreak b_2\allowbreak a_2\allowbreak b_2\allowbreak A_1\allowbreak A_1\allowbreak A_1\allowbreak A_1\allowbreak B_1\allowbreak a_2\allowbreak b_2\allowbreak b_2\allowbreak A_1\allowbreak b_2\allowbreak a_2\allowbreak b_2\allowbreak A_1$\par
\noindent $r_{P_1} = a_1\allowbreak B_2\allowbreak A_2\allowbreak A_1\allowbreak b_2\allowbreak a_2$\par
\endgroup

\begin{figure}[t]
\centering
\includegraphics[width=3in]{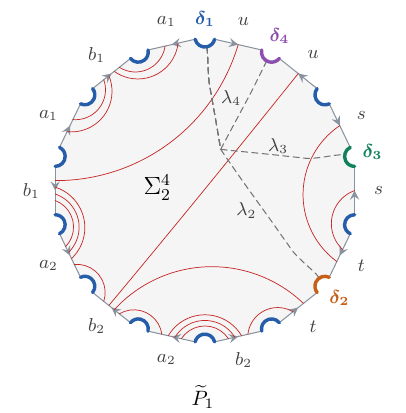}
\caption{The fundamental domain for the surface $\Sigma_2^4$ obtained by gluing the edges as indicated. The blue, orange, green, and purple
boundary arcs form $\delta_1,\delta_2,\delta_3,\delta_4$; the dashed
arcs are $\lambda_2,\lambda_3,\lambda_4$. The red curve is the vanishing cycle $\widetilde P_1$.}
\label{fig:xiao-marking}
\end{figure}

\subsubsection*{The point-pushing obstruction.}
To determine the loop $\gamma$ for which $W=\mathcal P_\gamma$ in
$\M(\Sigma_{2,1})$, we apply the thirteen Dehn twists successively to each
based generator. A segment crossing the annulus supporting a Dehn twist
winds once around its core after the twist; outside the annulus the loop
is unchanged. We read the resulting based loops as words using
Appendix~\ref{app:word-reading}. Put
\[
 e=A_2B_2a_1B_2a_1B_2,\quad r=A_2b_1A_1B_1a_2,\quad H=A_2b_1a_1^2,\quad P=b_2A_1b_2a_2A_1^4B_1a_2b_2.
\]
This gives
\begin{equation}\label{eq:xiao-initial-action}
 W_*(x)=wx\bar w\quad(x=a_1,b_1,a_2,b_2),\qquad
 w=PA_1\bar e a_1e^2\bar r^3Her\bar e A_1.
\end{equation}
Since $(\mathcal P_\gamma)_*(x)=\bar\gamma x\gamma$ and $F$ has trivial
center, $\gamma=\bar w$. Proposition~\ref{prop:pointpush-obstruction} gives
\[
 \pi_1(X)=\left\langle a_1,b_1,a_2,b_2\ \middle|\ R_0,r_{M_3},r_{X_3},r_{M_0},\ldots,r_{P_1},w\right\rangle.
\]
Here $w$ is the relation imposed by the last $2$--handle.

We first calculate in $G=F/\langle\!\langle v_1,\ldots,v_{13}\rangle\!\rangle$.
Put $T=b_1a_1B_1$. The $X_1,X_3$ relations give $b_1=A_1b_2A_1b_2$ and $a_2=TB_2a_1$.
Substituting the latter in the $M_3$ relation $A_2Ta_2=a_1$ gives $T=B_2a_1b_2$
and $a_2=B_2a_1^2$. The $Y_2$ relation then becomes $A_1B_1b_2A_1b_2=1$; together with
$b_2A_1b_2=a_1b_1$, this gives $A_1B_1a_1b_1=1$. Hence $T=a_1=B_2a_1b_2$, and
\begin{equation}\label{eq:xiao-quotient}
 [a_1,b_2]=1,\qquad b_1=A_1^2b_2^2,\qquad a_2=a_1^2B_2.
\end{equation}
Conversely, these substitutions kill every vanishing cycle word and
$R_0$, so $G\cong\Z^2$. They give
$e=B_2^2$, $r=A_1$, $H=A_1^2b_2^3$, $P=a_1B_2$, and $w=1$.
Thus $\omega(f)=1$ and $\pi_1(X)\cong\Z^2$.
Proposition~\ref{prop:spherical-degrees} gives a homotopy section, but
we still need to cancel the loop in $F$, not just its image in $G$.

\subsubsection*{Finding a smooth section.}
We change the lifts fixing $p$, without changing the factorization in
$\M(\Sigma_2)$, to make the point-pushing loop trivial in $F$. As in
Lemma~\ref{lem:pointpush-class}, move a vanishing cycle past $p$ or apply
a point push to it. This changes its isotopy class relative to $p$, and
hence the point-pushing loop, but not its isotopy class on $\Sigma_2$
after forgetting $p$. For based loops $\zeta_i$ and $m_i\in\{0,1\}$, put
\[
 \psi_i=\mathcal P_{\bar\zeta_i}\mathcal P_{v_i^{m_i}}
          \varphi_i\mathcal P_{\zeta_i},
 \qquad \alpha_i=\zeta_i v_i^{m_i}\varphi_i^{-1}(\bar\zeta_i).
\]
Orient $v_i$ so that $\mathcal P_{v_i}\varphi_i$ is the positive Dehn twist
about the other parallel copy, with $p$ between the two copies.
Conjugating by $\mathcal P_{\bar\zeta_i}$ gives another positive Dehn
twist $\psi_i$ fixing $p$ and projecting to $t_{c_i}$. Take $m_i=1$ for
$i=3,6,9,10,13$ and $m_i=0$ otherwise, and choose
\[
 \zeta_4=A_1,\quad \zeta_7=A_1B_1a_2,\quad \zeta_9=\zeta_{12}=a_1,\quad
 \zeta_{10}=A_1^2B_1a_2,\quad \zeta_{11}=A_1^2B_1a_2b_2^2,
\]
with the other $\zeta_i=1$. For example,
$\varphi_4^{-1}(a_1)=B_1b_2A_2b_1a_1B_1$, so $\alpha_4=A_1\varphi_4^{-1}(a_1)=v_4$.
Transport each contribution by the preceding inverse twists before
multiplying. Equivalently, start with $u_0=1$ and calculate
$u_i=u_{i-1}(\varphi_1^{-1}\cdots\varphi_{i-1}^{-1})(\alpha_i)$.
The resulting word is
\begin{equation}\label{eq:xiao-ordered-cancellation}
 \alpha:=\alpha_1\varphi_1^{-1}(\alpha_2)\cdots
 (\varphi_1^{-1}\cdots\varphi_{12}^{-1})(\alpha_{13})=w\quad\text{in }F.
\end{equation}
The reductions use only the surface relation and free cancellation.
Naturality of point pushing gives
\[
 (\psi_{13}\cdots\psi_1)^{-1}=\mathcal P_{\bar\alpha}W^{-1},
 \qquad \psi_{13}\cdots\psi_1=\mathcal P_{\alpha\bar w}=1.
\]
This gives a smooth section by Proposition~\ref{prop:pointpush-obstruction}; cf. \cite[Theorem~3.19]{hillman-pedrotti}.

\begin{figure}[p]
\centering
\makebox[\textwidth][c]{\includegraphics[width=1.05\textwidth]{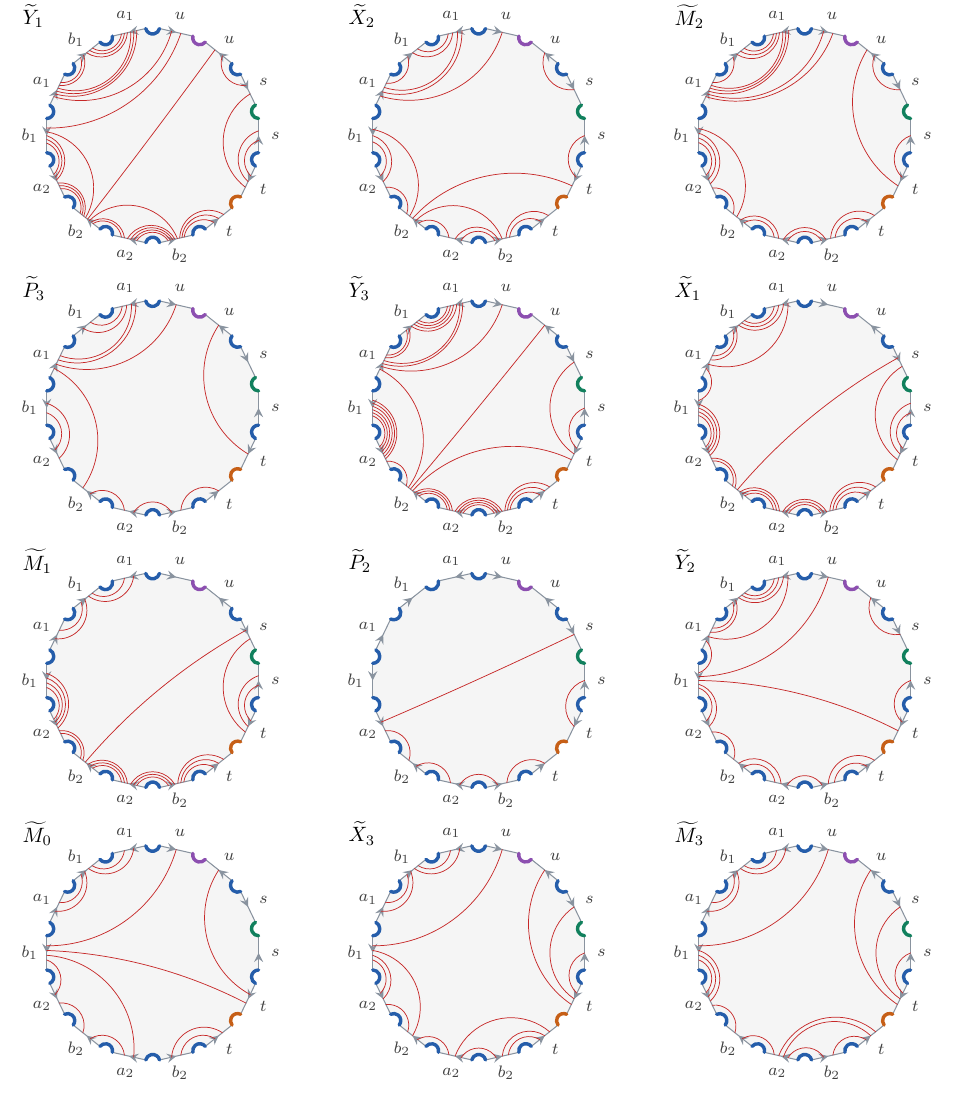}}
\caption{The remaining twelve lifted vanishing cycles in the written
product, from left to right, row by row. The cut surface and side
identifications are as in Figure~\ref{fig:xiao-marking}. The curves are
red; the thicker colored arcs are the four boundary components.}
\label{fig:xiao-lifts}
\end{figure}

\subsubsection*{Disjoint sections and their self-intersections.}
To add a section, keep the earlier marked points for the sections fixed and do the same
calculation in the punctured fiber. All local moves must avoid those
points so that the sections we get are disjoint.
The compatible representatives in Figures~\ref{fig:xiao-marking}
and~\ref{fig:xiao-lifts} give a simultaneous lift, as follows.

\begin{prop}\label{prop:xiao-four-sections}
Xiao's holomorphic $(6,7)$ fibration admits four pairwise disjoint smooth
sections of square $-2$, given by the following lift of Szab\'o's positive factorization:
\begin{equation}\label{eq:xiao-four-boundary}
 t_{\widetilde P_1}t_{\widetilde Y_1}t_{\widetilde X_2}t_{\widetilde M_2}
 t_{\widetilde P_3}t_{\widetilde Y_3}t_{\widetilde X_1}t_{\widetilde M_1}
 t_{\widetilde P_2}t_{\widetilde Y_2}t_{\widetilde M_0}t_{\widetilde X_3}t_{\widetilde M_3}
 =t_{\delta_1}^2t_{\delta_2}^2t_{\delta_3}^2t_{\delta_4}^2
 \quad\text{in }\M(\Sigma_2^4),
\end{equation}
where the vanishing cycles are as in
Figures~\ref{fig:xiao-marking} and~\ref{fig:xiao-lifts}.
The fibration has no holomorphic section. Every smooth section has square
$-2$, and four is the maximal number of pairwise disjoint smooth sections.
\end{prop}

\begin{proof}
Take a base point $q\in\delta_1$ at the common initial point of the joining
paths in Figure~\ref{fig:xiao-marking}. Put $R_4=[a_1,b_1][a_2,b_2]tsu$ in
$\pi_1(\Sigma_2^4,q)=\langle a_1,b_1,a_2,b_2,t,s,u\rangle$.
The paths $\lambda_2,\lambda_3,\lambda_4$ join $q$ to the other boundary
components and determine the based meridians $t,s,u$. Write $\mu_j$ for
the meridian at the endpoint of $\lambda_j$, oriented so that
$t_{\delta_j}(\lambda_j)\simeq\lambda_j\mu_j$.
Let $W_4$ be the product on the left of \eqref{eq:xiao-four-boundary}.
Tracing the loops and joining arcs through these Dehn twists gives
\[
\begin{aligned}
 (W_4)_*(x)&=R_4^2x\bar R_4^2 &&(x=a_1,b_1,a_2,b_2,t,s,u),\\
 W_4(\lambda_j)&\simeq R_4^2\lambda_j\mu_j^2 &&(j=2,3,4).
\end{aligned}
\]
These are the actions of the boundary multitwist in
\eqref{eq:xiao-four-boundary}. The four handle cut arcs and the three
joining arcs cut the surface into a disk, so Alexander's method
\cite{FarbMargalit:modbook} proves the relation. Capping all boundaries recovers Szab\'o's monodromy factorization. 
The boundary exponents give four disjoint sections of square $-2$; the
joining arcs record the framings invisible to the action on the based loops.

For the remaining assertions, let $F_0$ be a fiber of $f$, $G_0$ an
Albanese fiber of $X\to T^2$, and $K$ the canonical class. Following Polizzi
\cite[Propositions~6.15 and~6.18]{polizzi:pgq1}, choose elliptic components
$C_i$ of the seven reducible fibers, with $C_1,\ldots,C_4$ the degree--$1$
components in the four fibers whose components have degrees $1$ and $3$
over $T^2$. For a smooth section $S$, set
$\epsilon_i(S)=S\cdot C_i\in\{0,1\}$. Since $S^2\to T^2$ is
null-homotopic, $G_0\cdot S=0$. Make $S$ symplectic by adding a large
pullback of a base area form, as in Section~\ref{subsec:multisections}.
Polizzi's relation for the canonical class and adjunction give
\[
 2K=3F_0+G_0-\sum_{i=1}^4 C_i,\qquad
 S^2=\frac12\left(\sum_{i=1}^4\epsilon_i(S)-7\right)\in\{-3,-2\}.
\]

Since $\eu(X)=9$, $\sigma(X)=-5$, and $b_1(X)=2$, we have
$b_2^+(X)=3$ and $b_2^-(X)=8$. The seven disjoint $C_i$ have square $-1$,
so their orthogonal complement in $H_2(X;\R)$ has signature $(3,1)$.
Since $F_0$ is a nonzero class of square zero in this complement,
its orthogonal subspace there is nonnegative.
Projecting $K-2S$ orthogonally off the $C_i$ gives a class in this
subspace of square $18+8S^2\geq0$, excluding $-3$.
Thus every section has square $-2$ and misses exactly one of the four
degree--$1$ components.

Polizzi also shows that $K\cdot C>0$ for every holomorphic curve
$C\subset X$ \cite[Proposition~6.3]{polizzi:pgq1}.
A holomorphic section $S$ would instead satisfy
$K\cdot S=-2-S^2=0$ by adjunction, a contradiction.

For disjoint sections $S,T$, projecting $S-T$ in the same way gives
\[
 -4+\sum_{i=1}^7\bigl(\epsilon_i(S)-\epsilon_i(T)\bigr)^2\geq0.
\]
Their component choices must therefore differ in at least four fibers.
They cannot miss the same $C_j$ among $C_1,\ldots,C_4$: both would meet
the degree--$3$ component in that fiber and the degree--$1$ components
in the other three. Their choices could then differ only in the remaining
three reducible fibers, contradicting the inequality. Thus at most four
sections can be pairwise disjoint.
\end{proof}

\begin{samepage}
\begin{remark}\label{rmk:xiao-bisections}
There are three holomorphic $2$--sections of this fibration coming from the ramification locus of the
quotient map by the fiberwise hyperelliptic involution. Each pair
$\{r_i,r_i'\}$ traces a double cover of the base branched at the two
critical points corresponding to $X_i,Y_i$, and hence a sphere $S_i$.
Their self-intersections can be calculated from the boundary twist exponents
in the framed monodromy lifts, with the corrections for the two branches
at Lefschetz critical points, as in
\cite[Theorem~3.6]{baykur-hayano:multisections}; each has square $-3$.
Indeed, these are the rational curves studied by Polizzi
\cite[Corollary~6.8 and Propositions~6.11,~6.18]{polizzi:pgq1}.
The local perturbation in \cite[Remark~3.9]{baykur-hayano:multisections}
moves the branches off $\operatorname{Crit}(f)$, giving three disjoint
$2$--sections of square $-3$ in the convention of
Section~\ref{subsec:multisections}.
\end{remark}
\end{samepage}

\enlargethispage{0.3in}
\begingroup
\lstset{basicstyle=\ttfamily\small\linespread{0.96}\selectfont}
\section{GAP calculation for the genus--$9$ example}\label{app:gap}
% ==========================================================

The following code uses the presentation of $G$ from
Section~\ref{subsec:genus9-pi1}. Here \texttt{F} is free on $18$
generators, and \texttt{rels} consists of the surface relation followed
by the $48$ vanishing-cycle words; the suffix \texttt{p} in a relator
name denotes a prime. The point-pushing word \texttt{w} is not included
among the relators. The code computes its order in the abelianization
and in the maximal class--$2$ nilpotent quotient, checking all defining
relators under the resulting maps.

\medskip
\begin{lstlisting}
F:=FreeGroup(18);; gens:=GeneratorsOfGroup(F);;
a1:=gens[1];; b1:=gens[2];; a2:=gens[3];; b2:=gens[4];;
a3:=gens[5];; b3:=gens[6];; a4:=gens[7];; b4:=gens[8];;
a5:=gens[9];; b5:=gens[10];; a6:=gens[11];; b6:=gens[12];;
a7:=gens[13];; b7:=gens[14];; a8:=gens[15];; b8:=gens[16];;
a9:=gens[17];; b9:=gens[18];;
A1:=a1^-1;; B1:=b1^-1;; A2:=a2^-1;; B2:=b2^-1;; A3:=a3^-1;; B3:=b3^-1;;
A4:=a4^-1;; B4:=b4^-1;; A5:=a5^-1;; B5:=b5^-1;; A6:=a6^-1;; B6:=b6^-1;;
A7:=a7^-1;; B7:=b7^-1;; A8:=a8^-1;; B8:=b8^-1;; A9:=a9^-1;; B9:=b9^-1;;
R0:=a1*b1*A1*B1*a2*b2*A2*B2*a3*b3*A3*B3*a4*b4*A4*B4*a5*b5*A5*B5*a6*b6*A6*B6*a7*b7*A7*B7*a8*b8*A8*B8*a9*b9*A9*B9;;
rw3:=a1*a2*a3*b3*a4*a6*a7*B7*a8*a9*a1*b1*A1*B1*a2*a3*b3*a4*a5*a6*a7*a8*a9;;
ra3:=a2*B9*A9*b8*a8*B8*a9*a1*b1*A1*B1*a2*a3*b3*a4*a5*a6*a7*a8*a9*a1*B1*a2*a3*b3*a4*a6*a7*B7*a8*a9*a1*b1*A1;;
ra3p:=a1*a2*a3*b3*a4*a6*B9*a1*b1*A1*B1*a2*b2*A2*B2*a3*b3*A3*B3*a4*a5*b5*A5*a6*a7*B7*a8*a9*a1*b1*A1*B1*a2*a3*b3*a4*a5*B5*a6*a7*a8*a9;;
rx3:=a2*B9*A9*b8*a8*B8*a9*a1*b1*A1*B1*a2*a3*b3*a4*a5*B5*a6*a7*a8*a9*a1*B1*a2*a3*b3*a4*a6*B9*a1*b1*A1*B1*a2*b2*A2*B2*a3*b3*A3*B3*a4*a5*b5*A5*a6*a7*B7*a8*a9*a1*b1*A1;;
rb3:=a2*B9*A9*b8*a8*B8*A8*b7*A7*A6*B5*a6*a7*a8*a9*a1*B1*A1*A9*A8*b7*A7*B9*a1*b1*A1*B1*a2*b2*A2*B2*a3*b3*A3*B3*a4*a5*b5*A5*a6*a7*B7*a8*a9*a1*b1*A1;;
rb3p:=a2*B9*A9*b8*a8*B8*a9*a1*b1*A1*B1*a2*a3*b3*a4*a5*B5*A5*A4*A3*A2*B1*a2*a3*b3*a4*a6*B9*a1*b1*A1*B1*a2*b2*A2*B2*a3*b3*A3*B3*a4*a5*b5*A5*A4*B3*A3*A2*b1;;
ry3:=A2*B1*a2*a3*b3*a4*a5*B5*A5*A4*b3*a3*B3*A3*b2*a2*B2*A2*b1*a1*B1*A1*b9*a7*B7*a8*a9*a1*b1*a2*a3*a4*a5*b5*a6*a7*B7*a8*b8*A8*B8*a9*b9;;
rc3:=A2*B1*a2*a3*b3*a4*a6*a7*B7*a8*a9*a1*b1*a2*a3*a4*a5*a6*a7*B7*a8*b8*A8*B8*a9*b9;;
rc3p:=a1*a2*a3*a4*a5*b5*a6*a7*B7*a8*a9*a1*b1*A1*B1*a2*a3*b3*a4*a5*B5*A5*A4*b3*a3*B3*A3*b2*a2*B2*A2*b1*a1*B1*A1*b9*a7*B7*a8*a9;;
rz3:=a1*a2*a3*a4*a5*a6*a7*B7*a8*a9*a1*b1*A1*B1*a2*a3*b3*a4*a6*a7*B7*a8*a9;;
rx1:=a2*B2*a3*b3*A3*B9*A9*b8*a8*B8*A8*b7*a8*B8*a9*a1*b1*A1;;
rx2:=B9*a1*b1*A1*B1*a2*b2*A2*B2*a3*b3*A3*a4*B4*a5*b5*A5*a6*B6*a7*b7*A7;;
rw4:=B9*A9*b8*a8*B8*a9*a1*a2*a3*b3*a4*a3*B3*A3*b2*a2*B2*a3*b3*a4*a5*a6;;
ra4:=B9*A9*b8*a8*B8*A8*b7*a8*B8*a9*a1*a2*a3*b3*a4*a3*B3*A3*b2*a2*B2*A2*b1*a2*B2*a3*b3*a4*a5*a6;;
ra4p:=B9*A9*b8*a8*B8*a9*a1*a2*a3*b3*a4*B4*a5*b5*A5*a3*B3*A3*b2*a2*B2*a3*b3*a4*a5*a6*B6*a7*b7*A7;;
rx4:=B9*A9*b8*a8*B8*A8*b7*a8*B8*a9*a1*a2*a3*b3*a4*B4*a5*b5*A5*a3*B3*A3*b2*a2*B2*A2*b1*a2*B2*a3*b3*a4*a5*a6*B6*a7*b7*A7;;
rb4:=B9*A9*b8*a8*B8*A8*b7*B8*a9*a1*b1*A1*B1*a2*b2*A2*B2*a3*b3*A3*B4*a5*b5*A5*a3*B3*A3*b2*a2*B2*A2*b1*B2*a3*b3*A3*B6*a7*b7*A7;;
rb4p:=a2*B2*a3*b3*a4*a5*a6*B6*a7*b7*A7*A6*A5*A4*A3*A2*A1*A9*A8*b7*a8*B8*a9*a1*a2*a3*b3*a4*B4*a5*b5*A5*A4*B3*A3*A2*b1;;
ry4:=B1*a2*a3*b3*a4*a5*B5*A5*b4*a3*B3*A3*b2*a2*B2*A2*b1*a1*B1*A1*A9*b8*B7*a8*a9*a1*a2*a3*a4*a5*a6*a7*B7*A7*b6*a3*B3*A3*b2;;
rc4:=B1*a2*a3*b3*a4*a3*B3*A3*b2*a2*B2*A2*b1*a1*B1*A1*A9*b8*B7*a8*a9*a1*a2*a3*a4*a5*a6*a3*B3*A3*b2;;
rc4p:=a1*a2*a3*a4*a5*a6*a7*B7*A7*b6*a3*B3*A3*b2*a2*B2*a3*b3*a4*a5*B5*A5*b4*a3*B3*A3*b2*a2*B2*A2*b1*a1*B1*A1*A9*b8*a8*B8*a9;;
rz4:=a1*a2*a3*a4*a5*a6*a3*B3*A3*b2*a2*B2*a3*b3*a4*a3*B3*A3*b2*a2*B2*A2*b1*a1*B1*A1*A9*b8*a8*B8*a9;;
ry1:=B1*A1*A9*b8*B7*a8*a9*a1*a2*a3*B3*A3*b2;;
ry2:=B3*a4*a5*a6*a7*B7*A7*b6*B5*A5*b4;;
rw5:=a5*B7*a8*b8*A8*B8*a9*a1*b1*a1*B1*A1*b9*a9;;
ra5:=a5*B7*a8*a9*a1*a2*a1*B1*A1*b9*a9*B9*a1*b1*A1;;
ra5p:=a5*B5*A5*b4*a4*B4*a5*a6*B7*a8*b8*A8*B8*a9*a1*a2*b2*A2*B2*a3*b3*A3*B3*a4*b4*A4*B4*a5*b5*A5*a9;;
rx5:=a5*B5*A5*b4*a4*B4*a5*a6*B7*a8*a9*a1*a2*B1*a2*b2*A2*B2*a3*b3*A3*B3*a4*b4*A4*B4*a5*b5*A5*a9*B9*a1*b1*A1;;
rb5:=B1*A1*A9*b8*a8*B8*A8*b7*a7*B7*A7*b6*a6*B6*A6*b5*a9*B9*a1*b1*A1*A9*B5*A5*b4*a4*B4*a5*a6*B7*a8*a9*a1*a2;;
rb5p:=B9*a1*b1*A1*a5*B5*A5*b4*a4*B4*a5*a6*A5*a3*B3*A3*b2*a2*B2*A2*A1*A9*b8*a8*B8*a9*a1*a2*B1*a2*b2*A2*B2*a3*b3*A3*B3*a4*b4*A4*B4*a5*b5*A5;;
ry5:=a1*B1*A1*b9*B5*a6*b6*A6*B6*a7*b7*A7*B7*a8*b8*A8*B8*a9*a1*b1*A2*A1*A9*b8*A8*B8*a9*a1*a2*b2*A2*B2*a3*b3*A3*a5*A6*A5*b4*A4*B4*a5*b5*a9;;
rc5:=a1*B1*A1*b9*a9*B9*a1*b1*A2*A1*A9*b8*A8*B8*a9*a1*a2*b2*A2*B2*a3*b3*A3*a5*a9;;
rc5p:=B5*a6*b6*A6*B6*a7*b7*A7*B7*a8*b8*A8*B8*a9*a1*a2*b2*A2*B2*a3*b3*A3*a5*A6*A5*b4*A4*B4*a5*b5*a9;;
rz5:=a9*B9*a1*a2*b2*A2*B2*a3*b3*A3*a5*a9;;
rz1:=A9*a8*B8*a9*a1*a2*a2*B2*A2*b1*A1;;
rz2:=a4*B4*a5*a6*a6*B6*A6*b5*A5;;
rw6:=B9*A9*b8*a8*B8*A8*b7*a7*B9*a1*b1*A1*B1*a2*b2*A2*B2*a3*b3*A3*B3*a4*b4*A4*B4*a5*a6*a7*a8*a9*a1;;
ra6:=B9*A9*b8*a8*B8*A8*b7*a7*B7*a8*b8*A8*B9*a1*b1*A1*B1*a2*b2*A2*B2*a3*b3*A3*B3*a4*b4*A4*B4*a5*a6*a7*a8*a9*a1*B1*a2*b2*A2;;
ra6p:=B9*A9*b8*a8*B8*A8*b7*a7*B7*A7*b6*a7*B9*a1*b1*A1*B1*a2*b2*A2*B2*a3*b3*A3*B3*a4*b4*A4*a5*B5*a6*a7*a8*a9*a1;;
rx6:=B9*A9*b8*a8*B8*A8*b7*a7*B7*A7*b6*a7*B7*a8*b8*A8*B9*a1*b1*A1*B1*a2*b2*A2*B2*a3*b3*A3*B3*a4*b4*A4*a5*B5*a6*a7*a8*a9*a1*B1*a2*b2*A2;;
rb6:=A9*A8*A7*A6*b5*A5*B4*a5*a6*a7*a8*B8*A8*b7*A7*B6*a7*a8*a9*a1*a2*B2*A2*b1*A1;;
rb6p:=a2*B2*A2*b1*B9*A9*b8*a8*B8*A8*b7*a7*B7*A7*b6*a6*B6*A6*b5*a4*B4*A4*b3*a3*B3*A3*b2*a2*B2*A2*b1*a1*B1*A1*b9*a8*B8*A8*b7*B6*a7*b7*A7*B7*a8*b8*A8*B8*a9*b9;;
ry6:=B9*A9*b8*a8*B8*A8*b7*a7*B7*A7*b6*a6*B6*A6*b5*B4*a5*a6*a7*a8*B8*A8*b7*B6*a7*a8*a9*a1*a2*B2*A2*b1;;
rc6:=B9*A9*b8*a8*B8*A8*b7*a7*B7*A7*b6*a6*B6*A6*b5*a5*B5*a6*a7*a8*B8*A8*b7*a7*B7*a8*a9*a1*a2*B2*A2*b1;;
rc6p:=B9*A9*b8*a8*B8*A8*b7*a7*B7*A7*b6*a6*B6*A6*b5*B4*a5*a6*a7*B6*a7*a8*a9*a1;;
rz6:=B9*A9*b8*a8*B8*A8*b7*a7*B7*A7*b6*a6*B6*A6*b5*a5*B5*a6*a7*a7*B7*a8*a9*a1;;
rw1:=B1*a2*b2*B9*A9*b8*a8*a9;;
rw2:=B9*a1*b1*A1*B1*a2*b2*A2*B2*a3*b3*A3*B3*a4*b4*B5*a6*b6;;
rels:=[R0,rw3,ra3,ra3p,rx3,rb3,rb3p,ry3,rc3,rc3p,rz3,rx1,rx2,rw4,ra4,ra4p,rx4,rb4,rb4p,
  ry4,rc4,rc4p,rz4,ry1,ry2,rw5,ra5,ra5p,rx5,rb5,rb5p,ry5,rc5,rc5p,rz5,rz1,rz2,rw6,ra6,
  ra6p,rx6,rb6,rb6p,ry6,rc6,rc6p,rz6,rw1,rw2];;
w:=A9*B9*a1*b1*A1*B1*a2*b2*A2*B2*a3*b3*A3*B3*a4*b4*A4*B5*a6*b6*A6*B7*a8*a9*a1*a2*A1;;
G:=F/rels;; LoadPackage("nq");;
for c in [1,2] do
  hom:=NqEpimorphismNilpotentQuotient(G,c);
  imgs:=List(GeneratorsOfGroup(G),x->Image(hom,x));
  image:=MappedWord(w,gens,imgs);
  ok:=ForAll(rels,r->IsOne(MappedWord(r,gens,imgs)));
  Print("Class ",c,": relators = ",ok,", order(w) = ",Order(image),"\n");
od;
\end{lstlisting}

\bigskip
The recorded results, in this format, are:
\begin{lstlisting}
Class 1: relators = true, order(w) = 1
Class 2: relators = true, order(w) = infinity
\end{lstlisting}

\section{GAP calculation for the genus--$7$ example}
\label{app:genus7-gap}
% ==========================================================

The following code uses the presentation of $G_7$ from
Section~\ref{subsec:genus7-pi1}. Here \texttt{F} is free on $14$
generators, and \texttt{rels} consists of the surface relation followed
by the $48$ vanishing-cycle words. The words
\texttt{y1},\ldots,\texttt{y48} represent $x_1,\ldots,x_{48}$.
The point-pushing word \texttt{w} is not included among the relators.

\medskip
\begin{lstlisting}
F:=FreeGroup(14);; gens:=GeneratorsOfGroup(F);;
a1:=gens[1];; b1:=gens[2];; a2:=gens[3];; b2:=gens[4];;
a3:=gens[5];; b3:=gens[6];; a4:=gens[7];; b4:=gens[8];;
a5:=gens[9];; b5:=gens[10];; a6:=gens[11];; b6:=gens[12];;
a7:=gens[13];; b7:=gens[14];;
A1:=a1^-1;; B1:=b1^-1;; A2:=a2^-1;; B2:=b2^-1;; A3:=a3^-1;; B3:=b3^-1;;
A4:=a4^-1;; B4:=b4^-1;; A5:=a5^-1;; B5:=b5^-1;; A6:=a6^-1;; B6:=b6^-1;;
A7:=a7^-1;; B7:=b7^-1;;
R0:=a1*b1*A1*B1*a2*b2*A2*B2*a3*b3*A3*B3*a4*b4*A4*B4*a5*b5*A5*B5*a6*b6*A6*B6*a7*b7*A7*B7;;
y1:=B3*A3*B2*A2*b1*a1*B1*A1*A7*b6*a6*a7*a1*B1*a2*b2*A7*B6*A6*b5*A5*a7;;
y2:=A7*a5*B5*a6*a7*a1*b1*A1*A7*A6*B6*a7*a1*b1*A1*a2*B2*A2*b1*a1*B1*A1*b7;;
y3:=B3*A3*B2*A2*b1*a1*B1*A1*b7*A7*a5*B5*a6*a7*a1*a2*b2*A7*A6*B6*a7*a1*b1*A1*a2*B2*A2*b1*a1*B1*A1*b7;;
y4:=A7*B6*a7*a1*b1*A1*B2*A2*b1*a2*a3*b3;;
y5:=A7*B6*a7*a1*b1*A1*A7*A6*B6*a7*a1*b1*A1*B2*A2*b1*a2*B2*A2*b1*a1*B1*A1*b7*A7*a5*B5*a6*a7*a1*a2*a3*b3;;
y6:=A7*A6*B6*a7*a1*b1*A1*B2*A2*b1*a1*B1*A1*b7*A7*a5*B5*a6*a7*a1*a2*a3*b3;;
y7:=B7*a1*b1*A1*B1*a2*b2*A2*a1*B1*A1*b7*A7*a5*B5*a6*a7*a1*B1*A1*A7*b6*A6*a7;;
y8:=a1*B1*A1*A7*b6*a7*B7*a1*b1*A1*B1*a2*b2*A2*A2*B1*a2*a3*b3;;
y9:=A7*B6*a7*a1*b1*a2*a3*b3*a3*a4*b4*a5*B5*a6*b6*A6*B6*a7*b7*A7*a5*B5*A5*b4*a4*B4*A4*b3*a3*B7*a1*b1*A1*B1*a2*b2*A2*B2*A2*b1*a2*B2*A2*b1*a1*B1*A1*b7*A7*B6*a7*a1*b1*A1*a3*a4*b4*A4*a5*B5*a6*b6*A6*A5*A4*B7*a1*b1*A1*B1*a2*b2*A2*B2*a3*b3;;
y10:=A7*B6*a7*a1*b1*a2*a3*b3*a3*B7*a1*b1*A1*B1*a2*b2*A2*B2*A2*b1*a2*B2*A2*b1*a1*B1*A1*b7*A7*B6*a7*a1*b1*A1*a3*a4*b4*A4*B7*a1*b1*A1*B1*a2*b2*A2*B2*a3*b3;;
y11:=A3*B3*a4*b4*A4*B4*a5*b5*A5*a7*B7*A7*b6*a6*B6*A6*b5*A5*B4*A4*B7*a1*b1*A1*B1*a2*b2*A2*B2*A2*b1*a2*B2*A2*b1*a1*B1*A1*b7*A7*B6*a7*a1*b1*A1*B3*A3*b2*a2*B2*A2*b1*a1*B1*A1*b7*a4*a5*a6*B6*A6*b5*A5*B7*a1*b1*A1*B1*a2*b2*A2*B2*a3*b3*A7*B6*a7*a1*b1*a2*a2*B2*A2*b1*a1*B1*A1*b7;;
y12:=B7*a1*b1*A1*B1*a2*b2*A2*a3*a4*b4*A4*a5*B5*a6*b6*A6*A5*A4*B7*a1*b1*A1*B1*a2*b2*A2*B2*a3*b3*a3*a4*b4*a5*B5*a6*b6*A6*B6*a7*b7*A7*a5*B5*A5*b4*a4*B4*A4*b3*a3;;
y13:=A7*B6*a7*a1*b1*a2*a3*a4*b4*A4*a5*B5*a6*b6*A6*A5*A4*B7*a1*b1*A1*B1*a2*b2*A2*B2*A2*b1*a2*B2*A2*b1*a1*B1*A1*b7*A7*B6*a7*a1*b1*A1*a3*a4*b4*a5*B5*a6*b6*A6*B6*a7*b7*A7*a5*B5*A5*b4*a4*B4*A4*b3*a3;;
y14:=A7*B6*a7*a1*b1*a2*a2*B2*A2*b1*a1*B1*A1*b7*A3*B3*a4*b4*A4*a5*B5*a6*b6*A6*A5*A4*B7*a1*b1*A1*B1*a2*b2*A2*B2*A2*b1*a2*B2*A2*b1*a1*B1*A1*b7*A7*B6*a7*a1*b1*A1*B3*A3*b2*a2*B2*A2*b1*a1*B1*A1*b7*a4*a5*B5*a6*b6*A6*B6*a7*b7*A7*a5*B5*A5*b4*a4*B4*A4*b3*a3;;
y15:=B3*A3*b2*a2*B2*A2*b1*a1*B1*A1*b7*a4*B4*A4*A3*a2*B2*A2*b1*a1*B1*A1*b7*A7*B6*a7*a1*b1*a2*a2*B2*A2*b1*a1*B1*A1*b7*A3*B3*A3*A2*b1*a2*B2*A2*b1*a1*B1*A1*b7*A7*B6*a7*a1*b1*A1;;
y16:=B7*a1*b1*A1*B1*a2*b2*A2*a3*a4*b4*A4*a7*B7*A7*b6*a6*B6*A6*b5*a7*B7*A7*b6*a6*B6*A6*b5*A5*A4*B7*a1*b1*A1*B1*a2*b2*A2*B2*a3*b3*a3*a4*a5*a6*B6*A6*b5*A5*a4*B4*A4*b3*a3;;
y17:=B7*a1*b1*A1*B1*a2*b2*A2*a3*a4*b4*a5*B5*a6*b6*A6*B6*a7*a1*b1*A1*A7*B6*A6*A5*b4*a5*B5*a6*b6*A6*A5*A4*B7*a1*b1*A1*B1*a2*b2*A2*B2*a3*a4*b4*A4*a5*B5*a6*b6*a7*a1*a2*B2*A2*b1*a1*B1*A1*b7*A7*B5*a6*b6*A6*B6*a7*b7*A7*a5*B5*A5*b4*a4*B4*A4*b3*a3;;
y18:=B7*a1*b1*A1*B1*a2*b2*A2*a3*a4*b4*a5*B5*a6*b6*A6*A5*A4*B7*a1*b1*A1*B1*a2*b2*A2*B2*a3*a4*b4*A4*a5*B5*A5*b4*a4*B4*A4*b3*a3;;
y19:=A3*B3*a4*b4*A4*a5*B5*a6*b6*a7*a1*a2*B2*A2*b1*a1*B1*A1*b7*A7*B5*a6*b6*A6*B6*a7*b7*A7*a4*B4*A4*A3*b2*a2*B2*A2*b1*a1*B1*A1*b7*a4*a5*B5*a6*b6*A6*B6*a7*a1*b1*A1*A7*B6*A6*A5*A4*A3*a2*B2*A2*b1*a1*B1*A1*b7;;
y20:=A7*B5*a6*b6*A6*B6*a7*a1*b1*A1*A7*B6*A6*A5*b4*a5*B5*a6*b6*a7*a1*a2*B2*A2*b1*a1*B1*A1*b7;;
y21:=A7*A5*A4*B7*a1*b1*A1*B1*a2*b2*A2*B2*a3*a4*b4*A4*a7*B7*A7*b6*a6*B6*A6*b5*a7*B7*a1*b1*A1*A7*B6*A6*A5*b4*a4*B4*A4*b3*a3*B7*a1*b1*A1*B1*a2*b2*A2*a3*a4*a5*a6*a7*a1*a2*B2*A2*b1*a1*B1*A1*b7;;
y22:=A7*a6*B6*A6*b5*A5*A4*B7*a1*b1*A1*B1*a2*b2*A2*B2*a3*a4*b4*A4*a7*B7*A7*b6*a6*B6*A6*b5*a7*B7*a1*b1*A1*A7*B6*A6*b5*A5*a4*B4*A4*b3*a3*B7*a1*b1*A1*B1*a2*b2*A2*a3*a4*a5*a6*a7*a1*a2*B2*A2*b1*a1*B1*A1*b7;;
y23:=B7*a1*b1*A1*B1*a2*b2*A2*a3*a4*a5*a6*B6*A6*b5*A5*A4*B7*a1*b1*A1*B1*a2*b2*A2*B2*a3*a4*b4*A4*a7*B7*A7*b6*a6*B6*A6*b5*a4*B4*A4*b3*a3;;
y24:=A7*a6*B6*A6*b5*a7*B7*a1*b1*A1*A7*B6*A6*b5*A5*B4*a5*a6*a7*a1*a2*B2*A2*b1*a1*B1*A1*b7;;
y25:=A7*B4*a5*a6*a7*a1*a3*B3*A3*A2*b1*A1*A7*B6*A6*b5*B6*a7*a1*a2*a3*a4*b4*A4;;
y26:=A7*B6*A6*b5*B6*a7*a1*a2*B2*A2*b1*a1*B1*A1*b7*A7*B4*a5*a6*a7*a1;;
y27:=a4*B4*A4*A3*A2*A1*b7*A7*B4*a5*a6*a7*a1*a1*B1*a2*a3*b3*A3*A1*A7*A6*b5*B6*a7*a1*a2*B2*A2*b1*a1*B1*A1*b7;;
y28:=A7*B6*a7*a1*a3*B3*A3*A2*b1*a1*B1*a2*a3*a4*b4*A4;;
y29:=A7*B6*A6*b5*B6*a7*a1*a3*B3*A3*A2*b1*a1*B1*A1*b7*A7*B4*a5*a6*a7*a1*a1*B1*a2*a3*a4*b4*A4;;
y30:=A1*A7*A6*b5*B6*a7*a1*a3*B3*A3*b2*A2*A1*b7*A7*B4*a5*a6*a7*a1*a1*B1*a2*a3*a4*b4*A4;;
y31:=A7*B5*a6*a7*a1*A7*A6*A5*b4*a7*B7*a1*a2*B2*A2*b1*a1*B1*A1*b7;;
y32:=a4*B4*A4*A3*A2*b1*a1*B1*a2*B2*a3*b3*A3*a2*B2*A2*b1*a1*B1*A1*b7*A7*B6*a7*a1;;
y33:=a1*a1*B1*a2*a3*a4*b4*A4*a3*B3*A3*b2*a3*a4*a5*B6*a7*b7*A7*B4*a5*B6*a7*b7*A7*B7*a1*b1*A1*B1*a2*b2*A2*a3*B3*A3*A2*b1*a1*B1*A1*b7*A7*B6*a7*a1*a3*B3*A3*b2*a3*a4*b4*A4*B3*a4*b4*A4*B4*a5*b5*A4*B3*a4*b4*A4;;
y34:=a1*a1*B1*a2*a3*a4*b4*A4*a3*B3*A3*b2*a3*B3*A3*A2*b1*a1*B1*A1*b7*A7*B6*a7*a1*a3*B3*A3*b2*a3*a4*b4*A4*B3*a4*b4*A4;;
y35:=a1*a1*B1*a2*B2*a3*b3*A3*a2*B2*A2*b1*a1*B1*A1*b7*a7*B7*A7*b6*A5*b4*a7*B7*A7*b6*A5*A4*B3*A3*A2*b1*a1*B1*A1*b7*A7*B6*a7*a1*a4*B4*A4*b3*a4*B5*A5*b4;;
y36:=a3*B3*A3*b2*a3*a4*b4*A4*B3*a4*b4*A4*B4*a5*b5*A4*B3*a4*b4*A4*a3*B3*A3*b2*a3*a4*a5*B6*a7*b7*A7*B4*a5*B6*a7*b7*A7*B7*a1*b1*A1*B1*a2*b2*A2;;
y37:=a1*a1*B1*a2*a3*a4*b4*A4*B3*a4*b4*A4*B4*a5*b5*A4*B3*A3*A2*b1*a1*B1*A1*b7*A7*B6*a7*a1*a3*B3*A3*b2*a3*a4*a5*B6*a7*b7*A7*B4*a5*B6*a7*b7*A7;;
y38:=a1*a1*B1*a2*B2*a3*b3*A3*B2*a3*b3*A3*B3*a4*b4*A4*B4*a5*b5*A4*B3*A3*A2*b1*a1*B1*A1*b7*A7*B6*a7*a1*a4*B4*A4*b3*a4*B4*a5*B6*a7*b7*A7*B4*a5*B6*a7*b7*A7;;
y39:=a4*B4*A4*b3*a4*B4*A4*A3*B2*a3*b3*A3*a2*B2*A2*b1*a1*B1*A1*b7*a1*a1*B1*a2*B2*a3*b3*A3*B2*a3*b3*A3*a4*B4*A4*A3*A2*b1*a1*B1*A1*b7*A7*B6*a7*a1;;
y40:=a7*B7*A7*b6*A5*b4*a7*B7*A7*b6*A5*b4*A4*B3*a4*b4*A4*a3*B3*A3*b2*a3*a4*B5*A5*b4*a4*B4*A4*b3*a3*B3*A3*b2*a3*B3*A3*b2*a3*a4*b4*A4*B3*a4*b4*A4;;
y41:=a3*B3*A3*b2*a3*a4*a5*B6*A6*b5*a5*A4*B3*a4*b4*A4*B3*a4*b4*A4*B4*a5*a6*a7*a1*a1*B1*A1*b7*A7*B4*a5*B6*a7*b7*A7*B4*a5*B6*a7*b7*A7*B7*a1*b1*A1*B1*a2*b2*A2;;
y42:=a3*B3*A3*b2*a3*a4*a5*A4*B3*a4*b4*A4*B3*a4*b4*A4*B4*a5*B6*a7*b7*A7*B7*a1*b1*A1*B1*a2*b2*A2;;
y43:=a3*B3*A3*b2*a3*a4*B5*a6*b6*A5*b4*A4*B3*a4*b4*A4*B3*a4*b4*A4*a7*B7*A7*b6*A5*b4*a7*B7*a1*b1*A1*A1*A7*A6*A5*b4*a4*B4*A4*b3*a3*B3*A3*b2;;
y44:=A7*B4*a5*B6*A6*b5*a5*B5*a6*a7*a1*a1*B1*A1*b7;;
y45:=a3*B3*A3*b2*a3*a4*B5*a6*a7*a1*a1*B1*A1*b7*A7*A4*B3*a4*b4*A4*B3*a4*b4*A4*a7*B7*A7*b6*A5*b4*a7*B7*A7*b6*A6*B6*a7*b7*A7*B7*a1*b1*A1*B1*a2*b2*A2;;
y46:=a3*B3*A3*b2*a3*a4*B5*a6*a7*a1*a1*B1*A1*b7*A7*A5*b4*A4*B3*a4*b4*A4*B3*a4*b4*A4*a7*B7*A7*b6*A5*b4*a7*B7*A7*b6*A6*A5*b4*a4*B4*A4*b3*a3*B3*A3*b2;;
y47:=a3*B3*A3*b2*a3*a4*A5*b4*A4*B3*a4*b4*A4*B3*a4*b4*A4*a7*B7*A7*b6*A5*b4*a4*B4*A4*b3*a3*B3*A3*b2;;
y48:=A7*A5*b4*a7*B7*A7*b6*A6*b5*A5*B5*a6*a7*a1*a1*B1*A1*b7;;
rels:=[R0,y1,y2,y3,y4,y5,y6,y7,y8,y9,y10,y11,y12,y13,y14,y15,y16,y17,y18,y19,y20,y21,y22,
  y23,y24,y25,y26,y27,y28,y29,y30,y31,y32,y33,y34,y35,y36,y37,y38,y39,y40,y41,y42,y43,y44,
  y45,y46,y47,y48];;
gamma4L:=a1*B1*A1*A7*b6*a7*a1*B1*a2*b2*A2*a1*B1*A1*b7;;
gamma4R:=a4*B4*A4*b3*a4*B4*A4*b3*A3*B2*a3*b3*A3*B2*a3*b3*A3*a2*B2*A2*b1*a1*B1*A1*b7*a4*B4*A4*b3*a4*B4*A4*b3*A3*B2*a3*b3*A3*B2*a3*b3*A3*a2*B2*A2*b1*a1*B1*A1*b7*a4*B4*A4*b3*a4*B4*A4*b3*a4*B4*A4*A3*B2*a3*b3*A3*B2*a3*b3*A3*a2*B2*A2*b1*a1*B1*A1*b7*a4*B4*A4*b3*a4*B4*A4*b3*a4*B4*A4*b3*A3*B2*a3*b3*A3*a2*B2*A2*b1*a1*B1*A1*b7*a4*B4*A4*b3*a4*B4*A4*b3*a4*B4*A4*b3*A3*B2*a3*b3*A3*B2*a3*b3*A3;;
gamma3L:=A7*B6*a7*a1*b1*A1*a2*B2*A2*b1*b1*b1*A1*A7*B6*a7*a1*b1*A1*a2*B2*A2*b1*a1*B1*A1*b7;;
gamma3R:=B3*A3;;
w:=gamma4L*gamma4R*gamma3L*gamma3R;;
\end{lstlisting}

\medskip
Continuing with these definitions, the following code computes the
order of \texttt{w} in the abelianization and in the maximal integral
class--$5$ nilpotent quotient, checking all defining relators under
the resulting maps.

\medskip
\begin{lstlisting}[firstline=2]
G:=F/rels;; LoadPackage("nq");;
for c in [1,5] do
  hom:=NqEpimorphismNilpotentQuotient(G,c);
  imgs:=List(GeneratorsOfGroup(G),x->Image(hom,x));
  image:=MappedWord(w,gens,imgs);
  ok:=ForAll(rels,r->IsOne(MappedWord(r,gens,imgs)));
  Print("Class ",c,": relators = ",ok,", order(w) = ",Order(image),"\n");
od;
\end{lstlisting}

\bigskip
The recorded results, in this format, are:
\begin{lstlisting}
Class 1: relators = true, order(w) = 1
Class 5: relators = true, order(w) = 2
\end{lstlisting}

\endgroup

%\bigskip

\bibliography{refs1}
\bibliographystyle{plain}
\setlength{\parskip}{0pt}

\end{document}